\documentclass{amsart}
\usepackage{graphicx} 
\usepackage{verbatim}
\usepackage[utf8]{inputenc}
\usepackage{xcolor}
\usepackage{float}
\usepackage{amsmath}
\usepackage{amssymb}
\usepackage{mathtools}
\usepackage{esint}
\usepackage[colorlinks=true,allcolors=magenta]{hyperref}
\usepackage{varioref}
\usepackage{amsthm}
\usepackage{comment}
\usepackage{enumitem}
\usepackage{ifthen}
\usepackage[capitalize,noabbrev]{cleveref}

\usepackage{cancel}

\usepackage[toc,page]{appendix}

\usepackage{mathrsfs}

\newtheorem{theorem}{Theorem}[section]

\newtheorem{lemma}[theorem]{Lemma}
\newtheorem{proposition}[theorem]{Proposition}
\newtheorem{assumption}[theorem]{Assumption}
\theoremstyle{definition}
\newtheorem{definition}[theorem]{Definition}
\newtheorem{example}[theorem]{Example}
\newtheorem{remark}[theorem]{Remark}

\def\Cc{\mathbb{C}}

\renewcommand{\d}{\mathrm{d}}
\renewcommand{\H}{\mathrm{H}}
\renewcommand{\L}{\mathrm{L}}
\newcommand{\W}{\mathrm{W}}

\newcommand{\Acl}{A_{cl}}
\newcommand{\phicoeff}{\gamma}

\newcommand{\re}{\mathrm{Re}\,}
\newcommand{\Pt}[1][t]{\mathbf{P}_{#1}}
\newcommand{\F}{\mathbb{F}}
\newcommand{\diag}{\mathrm{diag}}
\newcommand{\rank}{\mathrm{rank}}

\newcommand{\doo}{\partial}

\newcommand{\WBi}{\tilde W_{B,1}}
\newcommand{\WBh}{\tilde W_{B,2}}
\newcommand{\WCo}{\tilde W_C}

\newcommand{\AB}{A\&B}
\newcommand{\CD}{C\&D}

\newcommand{\XLP}{X_{LP}}
\newcommand{\TLP}{\mc{T}}
\newcommand{\DLP}{D_{LP}}

\newcommand*{\setm}[2]{\{\,#1\mid#2\,\}}   

\newcommand{\zinf}[1][0]{[#1,\infty)}

\newcommand{\zabr}[2]{(#1,#2]}
\newcommand{\zabl}[2]{[#1,#2)}

\newcommand{\Lploc}[1][p]{\L_{\textup{loc}}^{#1}}
\newcommand{\Lp}[1][p]{\L^{#1}}
\newcommand{\Hloc}[1][n]{\H_{\textup{loc}}^{#1}}
\newcommand{\Wloc}[1][n]{\W_{\textup{loc}}^{#1,\infty}}

\newcommand{\R}{\mathbb{R}}
\newcommand{\C}{\mathbb{C}}

\newcommand{\N}{\mathbb{N}}
\newcommand{\T}{\mathbb{T}}
\newcommand{\gs}{\sigma}
\newcommand{\Lin}{\mathcal{L}}
\newcommand{\citel}[2]{\cite[#2]{#1}}
\newcommand{\eq}[1]{\begin{align*}#1 \end{align*}}
\newcommand{\eqn}[1]{\begin{align}#1 \end{align}}
\newcommand{\pmat}[1]{\begin{bmatrix}#1 \end{bmatrix}}
\newcommand{\pmatsmall}[1]{\begin{bsmallmatrix}#1 \end{bsmallmatrix}}
\newcommand{\iprod}[2]{\langle #1,#2\rangle}
\newcommand*{\Iprod}[3][default]{\ifthenelse{\equal{#1}{default}}{\left\langle#2,#3\right\rangle}{\ldelim{#1}{\langle}#2,#3\rdelim{#1}{\rangle}}}
\newcommand{\norm}[1]{\| #1\|}
\newcommand{\Norm}[1]{\left\| #1\right\|}
\newcommand{\Dom}{D}

\newcommand{\mc}[1]{\mathcal{#1}}
\newcommand{\ran}{\textup{Ran}}
\renewcommand{\ker}{\textup{Ker}}
\newcommand{\gl}{\lambda}
\newcommand{\inv}{^{-1}}

\newcommand{\gd}{\delta}
\newcommand{\eps}{\varepsilon}
\newcommand{\abs}[1]{| #1|}

\newcommand{\Dcomp}{D_{\textup{comp}}}

\title[Well-Posedness and Stability Under Monotone Feedback]{Well-Posedness and Stability of Infinite-Dimensional Systems Under Monotone Feedback}

\begin{document}

\author[A.\ Hastir]{Anthony Hastir}
\address[A.\ Hastir]{School of Mathematics and Natural Sciences, University of Wuppertal, 42119, Wuppertal, Germany}
\email{hastir@uni-wuppertal.de}

\author[L.~Paunonen]{Lassi Paunonen}
\address[L.~Paunonen]{Mathematics Research Centre, Tampere University, P.O.~ Box 692, 33101 Tampere, Finland}
 \email{lassi.paunonen@tuni.fi}




\begin{abstract}
We study the well-posedness and stability of an impedance passive infinite-dimensional linear system under nonlinear feedback
 of the form
 $u(t)=\phi(v(t)-y(t))$, where $\phi$ is a monotone function.
Our first main result introduces conditions guaranteeing the existence of classical and generalised solutions in a situation where the original linear system is well-posed.
In the absence of the external input $v$
 we establish the existence of strong and generalised solutions 
under strictly weaker conditions.
Finally, 
we introduce conditions guaranteeing that the origin is a globally asymptotically stable equilibrium point of the closed-loop system.
Motivated by the analysis of partial differential equations with nonlinear boundary conditions,
we use our results to investigate
 the well-posedness and stability of abstract boundary control systems, 
 port-Hamiltonian systems, a Timoshenko beam model, and a two-dimensional boundary controlled heat equation.
\end{abstract}

\maketitle

\section{Introduction}

In this article we study the well-posedness and stability of abstract infinite-dimen\-sion\-al linear systems under nonlinear output feedback. 
These topics
 have been investigated actively in the literature both in the framework of abstract systems~\cite{Sle89,LasSei03,Ber12,SinWei23}, 
 as well as in the case of controlled partial differential equations~\cite{Van98,JolLau20}.
The situation where the original system is linear and the feedback is nonlinear leads to the class of infinite-dimensional \emph{Lur'e systems}~\cite{LogRya00,GraCal11,GuiLog19}.
The existing literature especially demonstrates that 
several classes of linear models with \emph{passivity} properties lead to well-behaving closed-loop systems under feedback with suitable directional properties. 
We focus on exactly this case.

More precisely, we consider the existence of solutions and stability of impedance passive linear systems under monotone output feedback. As our main results we introduce mild conditions 
under which the closed-loop system has well-defined generalised and classical solutions. Moreover, 
we introduce conditions for
the global asymptotic stability of the equilibrium point at the origin for the closed-loop system.
Our results especially generalise the well-posedness results in~\cite{GraCal11,HasCal19,Aug19} and the stability results in~\cite{Sle89,Ber12,CurZwa16}
and they can, in particular, be used to analyse
 the existence and asymptotic behaviour of solutions of
 partial differential equations with nonlinear feedback or boundary damping.

To introduce our main results, let $\Sigma=(\T,\Phi,\Psi,\F)$ be a \emph{well-posed linear system}~\cite{Sta05book,TucWei14} on a Hilbert space $X$. 
The system $\Sigma$ is called \emph{impedance passive} if its input and output spaces are equal, i.e., $Y=U$, and if for every initial state $x_0$ and input $u$ its 
state trajectory $x$ and output $y$ satisfy (see Section~\ref{sec:Prem_Sys} for details)
\eq{
\Vert x(t)\Vert_X^2 - \Vert x_0\Vert_X^2
\leq 2\re\int_0^t\langle u(s),y(s)\rangle_U\d s, \qquad t\geq 0.
}

For an impedance passive well-posed linear system $\Sigma$
we consider output feedback of the form $u(t)=\phi(v(t)-y(t))$, see Figure~\ref{fig:Diag_CL}.
 Here the function $\phi$ describes the nonlinearity of the feedback and $v$ is the new input of the system.
Applying this output feedback leads formally to the system of equations
\begin{subequations}
\label{eq:IntroLPSemigStateandOutput}
\eqn{
\label{eq:IntroLPSemigState}
x(t) &= \mathbb{T}_t x_{0} +  \Phi_t \Pt[t]\phi(v- y), \qquad t\geq 0,\\
\label{eq:IntroLPSemigOutput}
 y &= \Psi_\infty x_{0}  + \mathbb{F}_\infty \phi(v-y).
}
\end{subequations}
However,
 the existence of the closed-loop state trajectory $x$ and output $y$ satisfying~\eqref{eq:IntroLPSemigStateandOutput} depends on the properties of the system $\Sigma$ and the function $\phi$.

\begin{figure}[h!]
    \centering
    \includegraphics[width=0.5\linewidth]{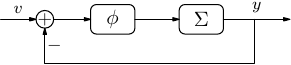}
    \caption{The closed-loop system resulting from the nonlinear feedback $u(t) = \phi(v(t)-y(t)).$}
    \label{fig:Diag_CL}
\end{figure}

In the case of impedance passive systems
the most natural class of
feedback nonlinearities $\phi$
consists of functions
which are \emph{monotone} in the sense that
 $\re\iprod{\phi(u_2)-\phi(u_1)}{u_2-u_1}\geq 0$ for $u_1,u_2\in U$.
Our first result utilises the strictly stronger condition~\eqref{eq:IntroPhiAss} to guarantee the existence of solutions to~\eqref{eq:IntroLPSemigStateandOutput}. 
The estimate for the differences of solutions in the claim shows 
that if $\phicoeff=1$, then the closed-loop system equipped with the output $\phi(v-y)$ is \emph{incrementally scattering passive} in the sense of~\cite{SinWei22}. Moreover, in the absence of the external input the closed-loop system is contractive.

\begin{theorem}
\label{thm:IntroWP}
Suppose that $\Sigma = (\T,\Phi,\Psi,\F)$ is an impedance passive well-posed linear system whose transfer function $P$ satisfies $P(\gl)+P(\gl)^\ast\ge c_\gl I$ for some $\gl,c_\gl>0$.
Moreover, assume that
\eqn{
\label{eq:IntroPhiAss}
\re \iprod{\phi(u_2)-\phi(u_1)}{u_2-u_1}_U\geq \phicoeff\norm{\phi(u_2)-\phi(u_1)}_U^2, \qquad u_1,u_2\in U
}
for some constant $\phicoeff>0$.
Then for any $x_0\in X$ and $u\in \Lploc[2](0,\infty;U)$
the equations~\eqref{eq:IntroLPSemigStateandOutput} have a unique solution
$x\in C(\zinf;X)$ and $y\in \Lploc[2](0,\infty;U)$.
Moreover, if $(x_1,y_1)$ and $(x_2,y_2)$
are two pairs of solutions of~\eqref{eq:IntroLPSemigStateandOutput} corresponding 
to the initial states $x_{01}\in X$ and $x_{02}\in X$
and the inputs 
$v_1\in \Lploc[2](0,\infty;U)$
and
$v_2\in \Lploc[2](0,\infty;U)$, respectively, then
for all $t\geq 0$
\eq{
\MoveEqLeft\norm{x_2(t)-x_1(t)}_X^2 + \phicoeff\int_0^t \norm{\phi(v_2(s)-y_2(s))-\phi(v_1(s)-y_1(s))}_{U}^2\d s\\
&\leq \norm{x_2(0)-x_1(0)}_X^2+\frac{1}{\phicoeff}\int_0^t \norm{v_2(s)-v_1(s)}_{U}^2\d s.
}
\end{theorem}
It should be noted that condition~\eqref{eq:IntroPhiAss} implies that the function $\phi:U\to U$ is globally Lipschitz continuous. The full version of the result is presented in Theorem~\ref{thm:LPSolutionsMain}. Theorem~\ref{thm:LPSolutionsMain} also shows that for initial conditions $x_0\in X$ and inputs $v\in \Hloc[1](0,\infty;U)$ satisfying a natural compatibility condition
 the state trajectory $x$ and output $y$ of the equation~\eqref{eq:IntroLPSemigStateandOutput} 
have additional smoothness properties and can be interpreted as classical state trajectory and output of the closed-loop system.

Theorem~\ref{thm:IntroWP} can also be applied in the case $v\equiv 0$ to guarantee the existence of solutions of the closed-loop system without an external input.
However,
in this case the existence of solutions can also be obtained under strictly weaker conditions. 
As our second main result in Theorem~\ref{thm:nonlinACPsol} we will show that the closed-loop system arising from feedback $u(t)=\phi(-y(t))$ 
has well-defined strong and generalised solutions provided that 
$\phi:U\to U$ is a continuous monotone function satisfying 
$\re \iprod{\phi(u_2)-\phi(u_1)}{u_2-u_1}> 0$ whenever $\phi(u_1)\neq \phi(u_2)$.
This result generalises the well-posedness results in~\cite{GraCal11,Tro14,Aug19}.
Our result in  Theorem~\ref{thm:nonlinACPsol} is also applicable in the case of systems which are impedance passive but not necessarily well-posed. This situation is especially encountered in the study of multi-dimensional wave equations with collocated boundary inputs and outputs.

In the second main part of the article we introduce conditions guaranteeing the stability of an impedance passive system under the feedback $u(t)=\phi(-y(t))$.
The state trajectories and outputs of a well-posed linear system $\Sigma$ can alternatively be described using the associated \emph{system node}~\citel{TucWei14}{Sec.~4}.
In this setting our feedback
 leads formally to a system of equations 
\eqn{
\label{eq:SysNodeSysIntro}
\pmat{\dot x(t)\\ y(t)}
= \pmat{A\&B \\ C\&D} \pmat{x(t)\\ \phi(-y(t))}, \qquad t\geq 0.
}
Our third main result below introduces conditions under which
the origin is a globally asymptotically stable equilibrium point of~\eqref{eq:SysNodeSysIntro}.
The main assumptions are 
that
the original linear system
 becomes asymptotically stable under \emph{linear} negative feedback $u(t)=-y(t)$ and 
that there exist 
$\alpha, \beta, \delta>0$ 
such that 
\eqn{
\label{eq:IntroPhiNonvanish}
\begin{cases}
 \re\langle \phi(u),u\rangle\geq \alpha\Vert u\Vert^2 &\mbox{when } \Vert u\Vert<\delta, \mbox{ and}\\
  \re\langle \phi(u),u\rangle\geq \beta & \mbox{when } \Vert u\Vert\geq \delta.
\end{cases}
}
 The result
does not require the original linear system to be well-posed, and it
 also guarantees that for all initial states $x(0)=x_0\in X$ the 
 system~\eqref{eq:SysNodeSysIntro}
 has a \emph{generalised solution} in the sense defined in Section~\ref{sec:WPnoinputs}.

\begin{theorem}
\label{thm:StabSecondIntro}
Suppose that the system node in~\eqref{eq:SysNodeSysIntro}
is impedance passive,
its transfer function $P$ satisfies $P(\gl)+P(\gl)^\ast\geq c_\gl I$ for some $c_\gl,\gl>0$, and
 $\ker(C)$ is dense in $X$.
 Let $\phi:U\to U$ be a locally Lipschitz continuous monotone function such that
$\phi(0)=0$,
 $\re \iprod{\phi(u_2)-\phi(u_1)}{u_2-u_1}> 0$ whenever $\phi(u_1)\neq \phi(u_2)$, 
 and there exist
$\alpha, \beta, \delta>0$ 
such that~\eqref{eq:IntroPhiNonvanish} hold.
Assume further that the semigroup associated with the original system under linear negative feedback $u(t)=-y(t)$ is strongly stable.
Then 
the origin is a globally asymptotically stable equilibrium point of~\eqref{eq:SysNodeSysIntro}, i.e., for every initial state $x_0\in X$ the system~\eqref{eq:SysNodeSysIntro} has a well-defined generalised solution $x\in C(\zinf;X)$ and 
$\norm{x(t)}\to 0$ as $t\to\infty$.
\end{theorem}

This result is presented in Theorem~\ref{thm:StabSecond}. 
Our second stability result in \cref{prp:StabWellPosed}
can be used to study the convergence of individual orbits of
\eqref{eq:SysNodeSysIntro} and it
can also be combined with well-posedness results from the literature, e.g.~\citel{TucWei14}{Thm.~7.2}, to study stability under feedback $u(t)=\phi(-y(t))$ where $\phi$ is not monotone.

In the final two sections we study partial differential equations with nonlinear boundary conditions.
We first present versions of our main results for the class of abstract boundary control systems~\cite{Sal87a,MalSta06} in Section~\ref{sec:BCS}.
Subsequently, in Section~\ref{sec:Examples} we employ our results in establishing the well-posedness and stability of the class of infinite-dimensional port-Hamiltonian systems~\cite{JacobZwart}, a Timoshenko beam model, and a two-dimensional heat equation with nonlinear boundary feedback.

Both the well-posedness and stability of infinite-dimensional systems with nonlinearities have been studied actively in the literature, and these two aspects are often considered together.
Earlier studies have introduced results for
abstract classes of systems and for concrete partial differential equations, often focusing either on monotone feedback for systems with collocated inputs, or on nonlinear dampings.
The well-posedness and stability under nonlinear feedback of abstract linear systems 
 have been studied especially in~\cite{Sle89, CurOos01, 
AlaAmm11,
AlaPri17,
CurZwa16, MarAnd17, 
MarChi20,
JacSch20,
LaaTab21,
BinTab23}
 in the case of bounded input operators, and in
\cite{SeiLi01, LasSei03, 
 Ber12, 
Aug19,
GuiLog19,
SchZwa21} for systems with unbounded input operators.
Some studies also focus on well-posedness and stability of infinite-dimensional Lur'e systems with either monotone or non-monotone nonlinear feedbacks~\cite{LogRya00,GraCal06,GraCal11}.
In addition, the stability of wave equations with nonlinear boundary dampings has been studied separately, for instance, in~\cite{Har85, Van98, JolLau20, VanFerCDC20}.
The well-posedness of infinite-dimensional linear and nonlinear systems under nonlinear feedback have been studied especially in~\cite{TucWei14,NatBen16,HasCal19,SinWei22}.
In particular, in~\cite{SinWei22,Sin23phd,SinWei23} 
the authors introduced the class of \emph{well-posed nonlinear systems} and studied passivity properties in this framework.

The existing results which are closest to our stability results have
 been presented in~\cite{LasSei03,Ber12,CurZwa16}.  In particular, in \cite{LasSei03} the authors study a large class of abstract systems with nonlinear boundary damping and introduce conditions for stability based on a comparison with the stability properties of a linear semigroup.  In~\cite{Ber12} the authors present conditions for the well-posedness and asymptotic stability of abstract systems with possibly unbounded collocated input and output operators and a large class of monotone nonlinear feedbacks. In our stability results the class of systems is allowed to be slightly larger than in~\cite{Ber12}.
Our Theorem~\ref{thm:StabSecond} directly generalises the stability results in~\citel{Ber12}{Cor.~3.11} and~\citel{CurZwa16}{Thm.~2.2} which were stated for systems with collocated input and output operators and semigroup generators with compact resolvent. 
Our results on well-posedness differ from many of the existing results due to the fact that Theorem~\ref{thm:IntroWP} focuses on the situation where the external input acts \emph{through} the nonlinearity.
Very recently during the preparation of this manuscript, the well-posedness of well-posed \emph{nonlinear} systems under
this same kind of feedback 
 was investigated in~\cite{MarWei25preprint}. 
This reference 
establishes the existence of
generalised solutions of the closed-loop system
 under the same condition~\eqref{eq:IntroPhiAss} as in Theorem~\ref{thm:IntroWP}.
 In addition, the approaches taken in this article and the reference~\cite{MarWei25preprint} are both based on 
the analysis of an associated Lax--Phillips semigroup.

The paper is organized as follows. In Section~\ref{sec:Prem_Sys} we review the definitions of well-posed linear systems and system nodes.
In Section \ref{sec:WP_Nonlin} we present our results on well-posedness under nonlinear feedback. Our results on strong stabilization by nonlinear output feedback are presented in Section \ref{sec:StrongStab}.
In Section~\ref{sec:BCS} we present versions of our results for the class of abstract boundary control systems.
In Section~\ref{sec:Examples} we employ our results in analysing the well-posedness and stability of partial differential equations with nonlinear boundary conditions, namely, infinite-dimensional port-Hamiltonian systems, a Timoshenko beam model, and a two-dimensional heat equation.

\subsection*{Acknowledgments}
This research was supported by the Research Council of Finland grant 349002 and the German Research Foundation (DFG)
grant HA 10262/2-1. The authors would like to thank Nicolas Vanspranghe for his helpful advice on nonlinear semigroups. They also thank Hilla Paunonen for her help on the proof of Lemma~\ref{lem:PhiProps}.

\subsection*{Notation} If $X$ and $Y$ are Banach spaces and $A:\Dom(A)\subset X\rightarrow Y$ is a linear operator we denote by $\Dom(A)$, $\ker(A)$ and $\ran(A)$ the domain, kernel, and range of $A$, respectively. The space of bounded linear operators from $X$ to $Y$ is denoted by $\Lin(X,Y)$ and we write $\Lin(X)$ for $\Lin(X,X)$. If \mbox{$A:X\rightarrow X$,} then $\gs(A)$
and $\rho(A)$ denote the spectrum
and the \mbox{resolvent} set of $A$, respectively. 
The inner product on a Hilbert space is denoted by $\iprod{\cdot}{\cdot}$ and all our Hilbert spaces are assumed to be complex.
For $T\in \Lin(X)$ on a Hilbert space $X$ we define $\re T = \frac{1}{2}(T+T^\ast)$.
For $\tau>0$ and $u:\zinf\to U$  
we define $\Pt[\tau]u: [0,\tau]\to U$ as the truncation of the function $u$ to the interval $[0,\tau]$.
We denote $\Cc_+ = \{\gl\in \Cc \ | \ \re\gl>0\}$.

\section{Preliminaries}\label{sec:Prem_Sys}

Let $U, X$ and $Y$ be (complex) Hilbert spaces.
Throughout the paper we consider a well-posed linear system $\Sigma = (\Sigma_t)_{t\ge 0} = (\T_t,\Phi_t,\Psi_t,\F_t)_{t\ge0}$ defined in the sense of~\citel{TucWei14}{Def.~3.1}
with input space $U$,  state space $X$,  and output space $Y$.
We define the \emph{mild state trajectory} $x\in C(\zinf;X)$ and \emph{mild output} $y\in\Lploc[2](0,\infty;Y)$
corresponding to the initial state 
$x_0\in X$  and the input $u\in \Lploc[2](0,\infty;U)$  so that 
\begin{subequations}
\label{eq:PrelimMildStateOut}
\eqn{
x(t) &= \mathbb{T}_t x_0 + \Phi_t\Pt u,\label{Eq_State}\\ 
\mathbf{P}_t y &= \Psi_t x_0 + \mathbb{F}_t\Pt u,\label{Eq_Output} 
}
\end{subequations}
for all $t\geq 0$.
We denote the \emph{extended output map} and \emph{extended input-output map}~\citel{TucWei14}{Sec.~3} of $\Sigma$ by $\Psi_\infty:X\to \Lploc[2](0,\infty;Y)$ and $\F_\infty:\Lploc[2](0,\infty;U)\to \Lploc[2](0,\infty;Y)$, respectively.
Using these operators the output $y$ corresponding to the initial state $x_0$ and input $u$ can be equivalently expressed as $y=\Psi_\infty x_0 + \F_\infty u$.
We especially study \emph{impedance passive} well-posed linear systems which are defined in the following way.
\begin{definition}\label{def:Imp_Passive}
A well-posed system $\Sigma$ is called \emph{impedance passive} if $Y=U$ and
 if every mild state trajectory $x$ and mild output $y$ corresponding to 
an initial state $x_0\in X$ and an input $u\in \Lploc[2](0,\infty;U)$ satisfy
\eq{
\Vert x(t)\Vert_X^2 - \Vert x_0\Vert_X^2
\leq 2\re\int_0^t\langle u(s),y(s)\rangle_U\d s, \qquad t\geq 0.
}
\end{definition}

Recall that 
a strongly continuous semigroup $\T$ 
on $X$ is called \emph{strongly stable} if $\norm{\T_tx_0}\to 0$ as $t\to \infty$ for all $x_0\in X$.
In our stability analysis in Section~\ref{sec:StrongStab} we will utilise the following result on well-posed systems with strongly stable semigroups.
\begin{lemma}[{\citel{Sta05book}{Lem.~8.1.2(iii)}}]
\label{Lemma_Conv_InputMap}
Let $\Sigma = (\T,\Phi,\Psi,\F)$ be a well-posed linear system
on the Hilbert spaces $(U,X,Y)$. Assume that the semigroup $\T$ is strongly stable and that there exists $M>0$ such that $\norm{\Phi_t}\le M$ for all $t\ge 0$.
Then for every $u\in \Lp[2](0,\infty;U)$ we have $\norm{\Phi_t\Pt u}_X\to 0$ as $t\to\infty$.
\end{lemma}

In the following we recall the concept of a \emph{system node} which is closely related to well-posed linear systems. 
For a generator $A$ of a strongly continuous semigroup we denote by $X_{-1}$ the completion of the space $(X,\norm{(\gl_0-A)\inv \cdot}_X)$ where  $\gl_0\in\rho(A)$ is fixed. It is well-known that $A:\Dom(A)\subset X\to X$ extends to an operator in $\Lin(X,X_{-1})$~\citel{Sta05book}{Sec.~3.6}, and we will also denote this extension by $A$.

\begin{definition}[\textup{\citel{Sta05book}{Def.~4.7.2}}]
\label{def:SysNode} 
Let $X$, $U$, and $Y$ be Hilbert spaces.
A closed operator 
\eq{
S:=\pmat{A\&B\\ C\&D} : \Dom(S) \subset X\times U \to X\times Y
}
is called a \emph{system node} on the spaces $(U,X,Y)$ if it has the following properties.
\begin{itemize}
\item The operator $A: \Dom(A)\subset X\to X$ defined by $Ax =A\&B \pmatsmall{x\\0}$ for $x\in \Dom(A)=\{x\in X\ | \ (x,0)^\top\in \Dom(S)\}$ generates a strongly continuous semigroup on $X$.
\item The operator $A\& B$ (with domain $\Dom(S)$) can be extended to an operator $[A,\ B]\in \Lin(X\times U,X_{-1})$.
\item $\Dom(S) = \{(x,u)^\top \in X\times U\ | \ Ax+Bu\in X\}$.
\end{itemize}
\end{definition}

The operator $A$ is called the \emph{semigroup generator} of $S$.
Every well-posed linear system $\Sigma = (\T,\Phi,\Psi,\F)$ is associated with a unique system node $S$~\citel{TucWei14}{Sec.~4}.
In particular, if $x_0\in X$ and $u\in \Hloc[1](0,\infty;U)$ are such that $(x_0,u(0))^\top \in \Dom(S)$, then the mild state trajectory $x$ and output $y$ of $\Sigma$ defined by~\eqref{eq:PrelimMildStateOut} satisfy $x\in C^1(\zinf;X)$, $x(0)=x_0$ and $y\in \Hloc[1](0,\infty;Y)$ and
\eqn{
\label{eq:PrelimSysnodeEqn}
\pmat{\dot x(t)\\y(t)} = S \pmat{x(t)\\ u(t)}, \qquad t\geq 0.
}
Motivated by this, a system node $S$ is called \emph{well-posed} if it is the system node of some well-posed linear system $\Sigma$.

We note that the conditions in Definition~\ref{def:SysNode} imply that $C\& D\in \Lin(\Dom(S),Y)$. 
The \emph{transfer function} 
$P:\rho(A)\to \Lin(U,Y)$ of a system node $S$ on $(U,X,Y)$ is defined
by
\eq{
P(\gl)u = C\& D \pmat{(\gl-A)\inv Bu\\ u}, \qquad u\in U, \ \gl\in\rho(A).
}
In addition, the \emph{output operator} $C\in \Lin(\Dom(A),Y)$ of $S$ is defined by $Cx= C\& D \pmatsmall{x\\0}$ for all $x\in \Dom(A)$.
We will now define the \emph{impedance passivity} of a system node. In the literature, it is more customary to use a definition based on the solutions of the equation~\eqref{eq:PrelimSysnodeEqn}, but the definitions are equivalent by~\citel{Sta02}{Thm.~4.2}. Due to this same result, a well-posed system node $S$ is impedance passive if and only if its associated well-posed linear system $\Sigma$ is impedance passive.

\begin{definition}\label{def:ImpPass}
A system node $S$ on $(U,X,Y)$ is \emph{impedance passive} if $Y=U$ and 
\eq{
    \re \left\langle A\& B
        \pmat{x_0\\ u_0}
,x_0\right\rangle_X \leq \re \left\langle C\& D\pmat{x_0\\ u_0 },u_0\right\rangle_U
}
for all $(x_0, u_0)^\top\in \Dom(S)$.
\end{definition}

It follows from Definition~\ref{def:ImpPass}
that the transfer function of an impedance passive system node satisfies $\re P(\gl)\geq 0$ for all $\gl\in\Cc_+$. In addition, the transfer function of an impedance passive well-posed system node $S$ coincides with the transfer function of the associated well-posed linear system $\Sigma$ on $\Cc_+$.

The main part of the 
following result, namely, that $-K$ is a system node admissible feedback operator and that the resulting system node $S^K$ satisfies~\eqref{eq:SysNodeFBStability}, was shown in~\citel{CurWei19}{Thm.~4.2}.
We present an alternative proof to also establish a formula for the semigroup generator $A^K$.
The proof also shows that if $\re K\ge I$, then $S^K$ and its associated well-posed linear system $\Sigma^K$ are \emph{scattering passive} in the sense of~\citel{Sta02}{Sec.~3} and~\citel{Sta05book}{Ch.~11}.

\begin{theorem}
\label{thm:SysNodeFeedback}
Let $S$ be an impedance passive system node on $(U,X,U)$ and let $K\in \Lin(U)$ be such that $\re K\geq c_KI$ for some $c_K>0$.
There exists a unique system node $S^K$ on $(U,X,U)$
such that the operator
\eq{
M=\pmat{I&0\\0&I}+\pmat{0\\K(C\&D)}
}
maps $\Dom(S)$ continuously one-to-one onto
 $\Dom(S^K)$ and
$S=S^KM$.
The system node $S^K$ is well-posed
 and its semigroup generator $A^K: \Dom(A^K)\subset X\to X$ is given by
\begin{subequations}
\label{eq:SysNodeFBAK}
\eqn{
\Dom(A^K) &= \Bigl\{ x\in X\, \Bigm| \exists v\in U \mbox{ s.t.}
\ Ax - BKv\in X,\
v = C\& D \pmatsmall{x\\-Kv}
\Bigr\}\\
A^K x &= Ax - BKv(x),
\qquad x\in \Dom(A^K),
}
\end{subequations}
where $v(x)$ is the element `$v$' in the definition of $\Dom(A^K)$.
If we denote the well-posed linear system associated to $S^K$ by $\Sigma^K=(\mathbb{T}^K,\Phi^K,\Psi^K,\mathbb{F}^K)$, then
\eqn{
\label{eq:SysNodeFBStability}
\Norm{\pmat{\T_t^K&\sqrt{c_K}\Phi_t^K\\\sqrt{c_K}\Psi_t^K&c_K\F_t^K}}_{\Lin(X\times \Lp[2](0,t;U),X\times \Lp[2](0,t;Y))}\leq 1, \qquad t\geq 0.
}
In particular, 
$\Psi_\infty^K x_0\in \Lp[2](0,\infty;Y)$ and $\F_\infty^K u\in \Lp[2](0,\infty;Y)$ for all $x_0\in X$ and $u\in \Lp[2](0,\infty;U)$. 
\end{theorem}

\begin{proof}
Denote by $\Acl$ the operator defined in~\eqref{eq:SysNodeFBAK}.
We first note that $\Acl$ is linear and that it is also single-valued due to the fact that `$v$' in $\Dom(\Acl)$ is unique. Indeed, if $x\in \Dom(\Acl)$ and $v_1,v_2\in U$ have the properties of $v$ in $\Dom(\Acl)$, then $(0,K(v_1-v_2))^\top = (x,-Kv_2)^\top -(x,-Kv_1)^\top \in \Dom(S)$ and $v_2-v_1=C\& D(0,K(v_1-v_2))^\top $ together with the impedance passivity of $S$ imply that
\eq{
 c_K\norm{v_2-v_1}^2
&\leq \re \iprod{v_2-v_1}{K(v_2-v_1)}
 =\re \Iprod{C\&D \pmat{0\\K(v_1-v_2)}}{K(v_2-v_1)}\\
&\leq \re \Iprod{A\&B \pmat{0\\K(v_1-v_2)}}{0}=0,
 }
showing that $v_2=v_1$.

Let $\gl>0$.
Since $\re P(\gl)\geq 0$ and $\re K\ge c_KI$ for $c_K>0$,
 \citel{Pau19}{Lem.~A.1(a)} implies that
 $I+KP(\gl)=K (K\inv+P(\gl))$ and $I+P(\gl)K$ are boundedly invertible.
We can define $Q_\gl: \Dom(A)\to X$ by $Q_\gl x = x - R_\gl BK(I+P(\gl)K)\inv Cx$, $x\in \Dom(A)$.
If we denote $R_\gl=(\gl-A)\inv$ for brevity, it is straightforward to check that $Q_\gl R_\gl (\gl-\Acl)x =x$ for $x\in \Dom(\Acl)$, and that $\ran(Q_\gl R_\gl)\subset \Dom(\Acl)$ and $(\gl-\Acl)Q_\gl R_\gl=I$. 
In particular, $Q_\gl$ is injective, $\ran(Q_\gl)=\Dom(\Acl)$, and
 $\ran(\gl-\Acl)=X$.
If $x\in X$ and $u,v\in U$ are such that $(x,u-Kv)^\top \in \Dom(S)$ and $v=C\& D \pmatsmall{x\\u-Kv}$, then
the impedance passivity of $S$ and $\re K\geq c_KI$ imply that
\eq{
\re \Iprod{A\& B \pmat{x\\u-Kv}}{x}_X
&\le\re \Iprod{C\& D \pmat{x\\u-Kv}}{u-Kv}_U\\
&\le\re \iprod{v}{u}_U - c_K\norm{v}_U^2.
}
In particular, choosing $x\in \Dom(\Acl)$ and $u=0$ and letting $v=v(x)$ be the element in the definition of $x\in\Dom(\Acl)$ shows that $\re \iprod{\Acl x}{x}_X\leq -c_K\norm{v}_U^2\leq 0$. Thus $\Acl$ is $m$-dissipative and generates a contraction semigroup on $X$ by the Lumer--Phillips Theorem. 
In particular, $\ran(Q_\gl)=\Dom(\Acl)$ is dense in $X$, and therefore~\citel{Sta05book}{Thm.~7.4.9} implies that the operator
$-K$ is a system node admissible output feedback operator for $S$ in the sense of~\citel{Sta05book}{Def.~7.4.2} and that the semigroup generator $A^K$ of $S^K$ is exactly $\Acl$.
In particular, 
there exists a unique system node $S^K = \pmatsmall{(A\& B)^K\\ (C\& D)^K}$ on $(U,X,U)$ such that
 $M$ maps $\Dom(S)$ continuously one-to-one onto $\Dom(S^K)$ and $S = S^KM$.

To show that $S^K$ is well-posed and that~\eqref{eq:SysNodeFBStability} holds, let $(x,u)^\top\in \Dom(S^K)$. If we define $v= (C\& D)^K \pmatsmall{x\\u}$, then~\citel{Sta05book}{Def.~7.4.2} implies that $(x,u-Kv)^\top \in \Dom(S)$, $v=C\&D \pmatsmall{x\\u-Kv}$, and $(A\&B)^K \pmatsmall{x\\u} = A\&B \pmatsmall{x\\u-Kv}$. 
Therefore our above estimate and Young's inequality imply that
\eq{
\re \Iprod{(A\&B)^K \pmat{x\\u}}{x}_X
&\leq \re \Iprod{(C\&D)^K \pmat{x\\u}}{u}_U
  -c_K\Norm{(C\&D)^K \pmat{x\\u}}_U^2\\
&\leq \frac{1}{2c_K} \norm{u}_U^2
-\frac{c_K}{2}\Norm{(C\&D)^K \pmat{x\\u}}_U^2.
}
Because of this,~\citel{Sta05book}{Thm.~11.1.5} implies
that the system node $\pmatsmall{I&0\\0&\sqrt{c_K}} S^K\pmatsmall{I&0\\0&\sqrt{c_K}}$ is scattering passive in the sense of~\citel{Sta05book}{Def.~11.1.2}. 
In particular, this system node
 and $S^K$ are well-posed and the well-posed system $\Sigma^K$ associated to $S^K$ satisfies~\eqref{eq:SysNodeFBStability} by~\citel{Sta05book}{Lem.~11.1.3 \& Thm.~11.1.4}.
\end{proof}

\begin{remark}
\label{rem:SysNodeFBRem}
The arguments in the proof show that part of Theorem~\ref{thm:SysNodeFeedback} holds more generally in the case where $S$ is an impedance passive node, $\re K\geq 0$, the operator in~\eqref{eq:SysNodeFBAK} is single-valued, and  $I+P(\gl)K$ is boundedly invertible for some $\gl\in\C_+$  (which is in particular true if $\re P(\gl)\geq c_\gl I$ for some $\gl\in \C_+$ and $c_\gl>0$).
Under these weaker assumptions there exists a unique system node $S^K$ such that $M$ maps $\Dom(S)$ continuously one-to-one onto $\Dom(S^K)$ and $S=S^KM$, and the semigroup generator of $S^K$ is $A^K$ in~\eqref{eq:SysNodeFBAK}. The system node $S^K$ is impedance passive, but not necessarily well-posed.
\end{remark}

\section{Well-Posedness Under Nonlinear Feedback}\label{sec:WP_Nonlin}

In this section we study the existence of state trajectories and outputs of  an impedance passive linear system under a nonlinear feedback.
We separately consider situations where the closed-loop system has an input and an output (in Section~\ref{sec:WPinputs}), and where the input of the closed-loop system is set to zero (in Section~\ref{sec:WPnoinputs}). In the latter case we prove the existence of solutions under weaker conditions.

\subsection{Well-Posedness With External Inputs}
\label{sec:WPinputs}

In this section we study the existence of solutions 
of the closed-loop system
\begin{equation}
\label{eq:SysNodeSysMain}
\left[\begin{matrix}\dot{x}(t)\\y(t)\end{matrix}\right] = \pmat{A\&B\\ C\& D}\left[\begin{matrix}x(t)\\\phi(u(t)- y(t))\end{matrix}\right], 
\qquad 
t\geq 0,
\end{equation}
where $\phi: U\to U$
 is a globally Lipschitz continuous function and
\begin{equation}
\label{eq:Op_S}
S=\pmat{A\&B\\ C\&D}, \qquad
D(S) 
 = \left\{(x,u)^\top
 \in X\times U \ \middle| \  Ax + Bu\in X\right\}
\end{equation} is an impedance passive system node on 
the Hilbert spaces $(U,X,U)$.
 Before stating our main result we define the concepts of
classical and generalised solutions of the system~\eqref{eq:SysNodeSysMain}, cf.~\cite[Def. 4.2]{TucWei14}, \citel{SinWei23}{Def.~5.1}.

\begin{definition}
\label{def:SysNodeStates}
Let $S$ be the system node~\eqref{eq:Op_S} on $(U,X,U)$.
A triple $(x,u,y)$ is called a \emph{classical solution} of \eqref{eq:SysNodeSysMain} on $\zinf$ if
\begin{itemize}
    \item $x\in C^1(\zinf;X)$,
$u\in C(\zinf;U)$, and
 $y\in C(\zinf;U)$;
    \item $\left[\begin{smallmatrix}x(t)\\\phi(u(t)-y(t))\end{smallmatrix}\right]\in D(S)$ for all $t\geq 0$;
    \item \eqref{eq:SysNodeSysMain} holds for every $t\geq 0$.
\end{itemize}
A triple $(x,u,y)$ is called a \emph{generalised solution} of~\eqref{eq:SysNodeSysMain} on $\zinf$ if 
\begin{itemize}
    \item $x\in C(\zinf;X)$, 
    $u\in\Lploc[2](0,\infty;U)$, and $y\in\Lploc[2](0,\infty;U)$;
    \item there exists a sequence $(x^k,u^k,y^k)$ of classical solutions of  \eqref{eq:SysNodeSysMain} on $\zinf$ such that for every $\tau>0$ we have
$(\Pt[\tau] x^k,\Pt[\tau] u^k,\Pt[\tau] y^k)^\top\to (\Pt[\tau] x,\Pt[\tau] u,\Pt[\tau] y)^\top$ as $k\to \infty$ in $C([0,\tau];X)\times \L^2(0,\tau;U) \times \L^2(0,\tau;U)$.
\end{itemize}
\end{definition}

We make the following assumption on the nonlinearity $\phi$.
The condition in particular implies
that $\phi$ is monotone and globally Lipschitz continuous with Lipschitz constant $1/\phicoeff$.
Examples of functions satisfying \cref{ass:PhiAss} are provided in \cref{ex:scalarSatComplex}.

\begin{assumption}
\label{ass:PhiAss}
The function $\phi: U\to U$ on the Hilbert space $U$
satisfies
\eq{
\re \iprod{\phi(u_2)-\phi(u_1)}{u_2-u_1}_U\geq \phicoeff\norm{\phi(u_2)-\phi(u_1)}_U^2, \qquad u_1,u_2\in U
}
for some $\phicoeff > 0$.
\end{assumption}

The following theorem is the main result of this section.
It shows that under Assumption~\ref{ass:PhiAss} the closed-loop system~\eqref{eq:SysNodeSysMain} resulting from nonlinear feedback has well-defined classical and generalised solutions. The estimate~\eqref{eq:ContractivityEstim} shows that if $\phicoeff = 1$, then the closed-loop system considered
with the input $u$ and the output
$\phi(u-y)$ is \emph{incrementally scattering passive} in the sense of~\cite{SinWei22}. Moreover, in the absence of an external input the closed-loop system is contractive.
In order to study the classical solutions of~\eqref{eq:SysNodeSysMain} we define a subset $\Dcomp$ of $X\times U$ by
\eq{
\Dcomp = \Bigl\{ (x,u)^\top\hspace{-.2ex}\in X\times U \hspace{-.3ex}\Bigm|\hspace{-.2ex} \exists v\in U \hspace{-.2ex}\mbox{ s.t. }
Ax + B\phi(u- v)\in X,
v = C\& D \pmatsmall{x\\\phi(u-v)}
\Bigr\}.
}

\begin{theorem}
\label{thm:LPSolutionsMain}
Suppose that $S$ is an 
impedance passive well-posed system node on the Hilbert spaces $(U,X,U)$ and let $\Sigma = (\T,\Phi,\Psi,\F)$ be the associated well-posed linear system.
Moreover, suppose that the tranfer function $P$ of $\Sigma$ satisfies $\re P(\gl)\geq c_\lambda I$ for some $\gl,c_\gl>0$, and 
that $\phi:U\to U$ satisfies Assumption~\textup{\ref{ass:PhiAss}} for some $\phicoeff>0$.
If $x_0\in X$ and $u\in \Lploc[2](0,\infty;U)$, then the system~\eqref{eq:SysNodeSysMain} has a 
 generalised solution $(x,u,y)$ satisfying $x(0)=x_0$. This solution satisfies $t\mapsto \phi(u(t)-y(t))\in\Lploc[2](0,\infty;U)$ and
\begin{subequations}
\label{eq:LPSemigStateAndOutput}
\eqn{
\label{eq:LPSemigState}
x(t) &= \mathbb{T}_t x_{0} +  \Phi_t \Pt[t]\phi(u- y), \qquad t\geq 0,\\
\label{eq:LPSemigOutput}
 y &= \Psi_\infty x_{0}  + \mathbb{F}_\infty \phi(u-y).
}
\end{subequations}
If $x_0\in X$ and $u\in \Hloc[1](0,\infty;U)$ are such that $(x_0,u(0))^\top\in \Dcomp$, then this $(x,u,y)$ is a classical solution of~\eqref{eq:SysNodeSysMain} on $\zinf$ and $y\in \Hloc[1](0,\infty;U)$.

If $(x_1,u_1,y_1)$ and $(x_2,u_2,y_2)$ are two generalised solutions of~\eqref{eq:SysNodeSysMain}, then for all $t\geq 0$ we have
\begin{subequations}
\label{eq:ContractivityEstim}
\begin{align}
\MoveEqLeft\norm{x_2(t)-x_1(t)}_X^2 + \phicoeff\int_0^t \norm{\phi(u_2(s)-y_2(s))-\phi(u_1(s)-y_1(s))}_{U}^2\d s\\
&\leq \norm{x_2(0)-x_1(0)}_X^2+\frac{1}{\phicoeff}\int_0^t \norm{u_2(s)-u_1(s)}_{U}^2\d s.
\end{align}
\end{subequations}
\end{theorem}

The proof of Theorem~\ref{thm:LPSolutionsMain} is presented in Section~\ref{sec:WPproofs}.

\begin{remark}
\label{rem:SysNodeWPIncrEstimate}
Combining the estimate~\eqref{eq:ContractivityEstim} with the formula~\eqref{eq:LPSemigOutput}
shows that
 for every $t>0$ there exists $M_t>0$ such that
if
$(x_1,u_1,y_1)$ and $(x_2,u_2,y_2)$ are two generalised solutions of~\eqref{eq:SysNodeSysMain},
then
\eq{
\MoveEqLeft\norm{x_2(t)-x_1(t)}_X^2 + \int_0^t \norm{y_2(s)-y_1(s)}_{U}^2\d s\\
&\leq M_t\norm{x_2(0)-x_1(0)}_X^2+M_t\int_0^t \norm{u_2(s)-u_1(s)}_{U}^2\d s.
}
This estimate in particular implies that
if $x_1(0)=x_2(0)$ and $u_1=u_2$, then necessarily $y_1=y_2$ and 
$x_1(t)=x_2(t)$ for $t\geq 0$. This means that the generalised ``state trajectory'' $x$ and ``output'' $y$ of~\eqref{eq:SysNodeSysMain} are uniquely determined by $x(0)$ and $u$.
\end{remark}

\begin{remark}
By~\citel{Log20}{Thm.~4.4} 
the condition
that 
 $\re P(\gl)\geq c_\gl I$ for some $c_\gl>0$
 in Theorem~\ref{thm:LPSolutionsMain}
is satisfied for one 
$\gl>0$ if and only if it is satisfied for all $\gl>0$ (here $c_\gl>0$ is allowed to depend on $\gl>0$).
The proofs of Lemma~\textup{\ref{lem:PhiProps}}, Proposition~\textup{\ref{prp:LPmaxDissipative}}, and Theorem~\textup{\ref{thm:LPSolutionsMain}} below
will demonstrate that 
this condition 
in Theorem~\ref{thm:LPSolutionsMain}
    can be replaced with any other assumption
    which guarantees that
for any $v\in U$ 
    the equation  $z+P(\gl)\phi(z)=v$ has a unique solution which is determined by $z=h(v)$, where $h:U\to U$ is a globally Lipschitz continuous function.
\end{remark}

\begin{remark}
We note that many of the arguments in the proof of Theorem~\ref{thm:LPSolutionsMain} remain valid even without the assumption that the system node $S$ is well-posed. 
Indeed, the well-posedness of the system node will be used in deducing that if $(x_0,u(0))^\top\in \Dcomp$, then $x\in C^1(\zinf;X)$ and that~\eqref{eq:LPACPx} holds for all $t\geq 0$, and in deducing that the system has a well-defined generalised output.
In the situation where the system node $S$ is impedance passive but not necessarily well-posed the arguments in the proof show that for all $x_0\in X$ and $u\in \Hloc[1](0,\infty;U)$ satisfying $(x_0,u(0))^\top\in \Dcomp$ the system~\eqref{eq:SysNodeSysMain} has a \emph{strong solution} in the sense that
\begin{itemize}
    \item $x\in \Wloc[1](0,\infty;X)$, 
    $u\in C(\zinf;U)$, and $y\in C(\zinf;U)$;
    \item $\left[\begin{smallmatrix}x(t)\\\phi(u(t)-y(t))\end{smallmatrix}\right]\in D(S)$ for all $t\geq 0$;
    \item \eqref{eq:SysNodeSysMain} holds for almost every $t\geq 0$.
\end{itemize}
The arguments in the proof will also show that any two such strong solutions satisfy the estimate~\eqref{eq:ContractivityEstim}.
Moreover, if $x_0\in X$ and $u\in \Lploc[2](0,\infty;U)$, then~\eqref{eq:SysNodeSysMain} still has a ``generalised state trajectory" $x\in C(\zinf;X)$ in the sense that 
    there exists a sequence $(x^k,u^k,y^k)$ of strong solutions of~\eqref{eq:SysNodeSysMain} such that for every $\tau>0$ we have
$(\Pt[\tau] x^k,\Pt[\tau] u^k)^\top\to (\Pt[\tau] x,\Pt[\tau] u)^\top$ as $k\to \infty$ in $C([0,\tau];X)\times \L^2(0,\tau;U)$.
Without well-posedness of the system node, the arguments in
 the proof of Theorem~\ref{thm:LPSolutionsMain} do not guarantee the existence of a generalised output,
but they show that there exists $\psi\in\Lploc[2](0,\infty;U)$ such that the sequence $(x^k,u^k,y^k)$ of strong solutions above satisfies $\Pt[\tau]\phi(u^k-y^k)\to \Pt[\tau]\psi$ as $k\to\infty$ for all $\tau>0$.
\end{remark}

\begin{example}\label{ex:scalarSatComplex}
Let $\psi:\zinf\to \zinf$ be a non-decreasing function which is globally Lipschitz continuous with Lipschitz constant $L_\psi$ and satisfies $\psi(0)=0$. Then $\phi: U \to U$ defined by $\phi(0)=0$ and $\phi(u)=\psi(\norm{u})\norm{u}\inv u$ for $u\neq 0$ satisfies \cref{ass:PhiAss} with $\phicoeff=1/L_\psi$. 
In particular, the \emph{saturation function} 
$\phi:\mathbb{C}\to\mathbb{C}$ 
defined so that
 ${\phi}(u) = u$ whenever $\vert u\vert< 1$ and ${\phi}(u) = u/\vert u\vert$ whenever $\vert u\vert\geq 1$ satisfies Assumption~\ref{ass:PhiAss} with $\phicoeff=1$.
Moreover, if $\phi_1:U_1\to U_1$ and $\phi_2:U_2\to U_2$ satisfy \cref{ass:PhiAss} for some $\phicoeff$, then also $\phi : U\to U$ defined by $\phi(u)=(\phi_1(u_1),\phi_2(u_2))^\top$ for $u=(u_1,u_2)^\top\in U:=U_1\times U_2$ satisfies \cref{ass:PhiAss} with the same $\phicoeff$.
\end{example}

\subsection{Well-Posedness Without External Inputs}
\label{sec:WPnoinputs}

In this section we consider the system
\eqn{
\label{eq:SysNodeSysNoInput}
\pmat{\dot x(t)\\ y(t)}
= S \pmat{x(t)\\ \phi(-y(t))}, \qquad t\geq 0,
}
where $S$ is an impedance passive system node.
We note that Theorem~\ref{thm:LPSolutionsMain} with $u(t)\equiv 0$ already provides conditions for the existence of generalised and classical solutions of this system. 
In this section we show the existence of solutions without well-posedness of $S$ and
under considerably weaker assumptions on  $\phi: U\to U$.
The system~\eqref{eq:SysNodeSysNoInput} can be formulated as a nonlinear abstract Cauchy problem 
\eqn{
\label{eq:nACP}
\dot x(t) = A_\phi(x(t)), \qquad t\geq 0,
}
when we define the operator $A_\phi :\Dom(A_\phi)\subset X\to X$ so that 
\eq{
\Dom(A_\phi) &= \Bigl\{ x\in X\, \Bigm| \exists v\in U \mbox{ s.t.}
\ Ax + B\phi(-v)\in X,\
v = C\& D \pmatsmall{x\\\phi(-v)}
\Bigr\}\\
A_\phi x &= Ax + B\phi(-v(x)),
\qquad x\in \Dom(A_\phi),
}
where $v(x)$ is the element $v$ in the definition of $\Dom(A_\phi)$.

\begin{definition}
\label{def:SysNodeNoInputSolutions}
Let $S$ be a system node on the spaces $(U,X,U)$.
\begin{itemize}
\item
The function $x\in \Hloc[1](0,\infty;X)$ is called a \emph{strong solution} of~\eqref{eq:nACP} if $x(t)\in \Dom(A_\phi)$ for all $t\geq 0$ and if the identity in~\eqref{eq:nACP} holds for almost every $t\geq 0$.
\item
The function $x\in C(\zinf;X)$ is called a \emph{generalised solution} of \eqref{eq:nACP} on $\zinf$ if 
there exists a sequence
 $(x^k)_{k\in\N}$ 
of strong solutions of~\eqref{eq:nACP} on $\zinf$ such that $\norm{\Pt[\tau] x^k-\Pt[\tau] x}_{C([0,\tau];X)}\to 0 $ as $k\to\infty$ for all $\tau>0$.
\end{itemize}
\end{definition}

It follows from the definition that if $x$ is a strong solution of~\eqref{eq:nACP}, we can define a corresponding \emph{output} so that $y(t)=v(x(t))$ (the element `$v$' corresponding to $x(t)\in\Dom(A_\phi)$) for $t\geq 0$.
Then $x$ and $y$ are such that the identity in~\eqref{eq:SysNodeSysNoInput} holds for almost every $t\geq 0$ (and the second line holds for all $t\ge 0$).

The following result introduces conditions for the existence of strong and generalised solutions of~\eqref{eq:SysNodeSysNoInput}. In particular, the proof shows that under the given assumptions 
 $A_\phi$ is a single-valued $m$-dissipative operator.
The theorem is related to the well-posedness results in~\cite{GraCal11,Tro14,Aug19}.

\begin{theorem}
\label{thm:nonlinACPsol}
Let $S$ be an impedance passive system node whose transfer function satisfies $\re P(\gl)\geq c_\gl I$ for some $\gl,c_\gl>0$, and let
$\phi:U\to U$ be a continuous monotone function
satisfying
  $\re \iprod{\phi(u_2)-\phi(u_1)}{u_2-u_1}> 0$ whenever $\phi(u_1)\neq \phi(u_2)$.
If $x_1$ and $x_2$ are two generalised solutions of~\eqref{eq:nACP}, then
\eqn{
\label{eq:NLFBContractivity}
\norm{x_2(t)-x_1(t)}_X \leq \norm{x_2(0)-x_1(0)}_X, \qquad t\geq 0.
}
For every 
$x_0\in \overline{\Dom(A_\phi)}$
the equation~\eqref{eq:nACP} has a 
unique generalised solution
$x$ satisfying $x(0)=x_0$.
If $x_0\in \Dom(A_\phi)$, then
this $x$ is a strong solution of~\eqref{eq:nACP} and its corresponding output $y$  defined by $y(t)=v(x(t))$, $t\geq 0$, is a
 right-continuous function 
such that $y\in \Lp[\infty](0,\infty;U)$
 and $\phi(-y)\in \Lp[\infty](0,\infty;U)$.

If $S$ is well-posed and its associated well-posed system is $\Sigma = (\T,\Phi,\Psi,\F)$, then for every $x_0\in \Dom(A_\phi) $ the strong solution $x$ and the corresponding output $y$ satisfy
\begin{subequations}
\label{eq:NLFBStateandOutput}
\eqn{
\label{eq:NLFBState}
x(t) &= \mathbb{T}_t x_{0} +  \Phi_t \Pt[t]\phi(- y), \qquad t\geq 0,\\
\label{eq:NLFBOutput}
 y &= \Psi_\infty x_{0}  + \mathbb{F}_\infty \phi(-y).
}
\end{subequations}
\end{theorem}

The following lemma lists
selected sufficient conditions for $\Dom(A_\phi)$ to be dense in $X$.
The proofs of \cref{thm:nonlinACPsol} and \cref{lem:DomAphiDense} will be presented in \cref{sec:WPproofs}.

\begin{lemma}
\label{lem:DomAphiDense}
Suppose that $S$ is an impedance passive system node whose transfer function satisfies $\re P(\gl)\geq c_\gl I$ for some $\gl,c_\gl>0$ and that $\phi: U\to U$ is a continuous monotone function. The set $\Dom(A_\phi)$ is dense in $X$ whenever one of the following conditions holds.
\begin{itemize}
\item[\textup{(a)}] The kernel $\ker(C)$ is dense in $X$.
\item[\textup{(b)}] 
$\norm{(\gl-A)\inv B}\to 0$ as $\gl\to\infty$
and
$\phi$ is bounded, i.e., $\sup_{u\in U}\norm{\phi(u)}<\infty$.
\item[\textup{(c)}] $P(\gl)\to 0$ as $\gl \to \infty$ and $\phi$ is globally Lipschitz continuous.
\item[\textup{(d)}]
 $B\in \Lin(U,X)$, 
$C$ extends to an operator $C\in \Lin(X,U)$, and
 $C\& D = [C,\ 0]$.
\item[\textup{(e)}]
 $S$ is well-posed
 and $\phi$ satisfies Assumption~\textup{\ref{ass:PhiAss}} for some $\phicoeff>0$.
\item[\textup{(f)}]
$S$ is well-posed with associated well-posed linear system $\Sigma = (\T,\Phi,\Psi,\F)$, 
$A_\phi$ is single-valued,
and
 for every $x_0\in X$ there exist $x\in C(\zinf;X)$ and $y\in\Lploc[2](0,\infty;U)$ which satisfy~\eqref{eq:NLFBStateandOutput}.
\end{itemize}
\end{lemma}

\begin{remark}
\label{rem:nACPremAphiSingleValued}
The proof of Theorem~\ref{thm:nonlinACPsol} will show that the assumption that $\re \iprod{\phi(u_2)-\phi(u_1)}{u_2-u_1}> 0$ whenever $\phi(u_1)\neq \phi(u_2)$ in \cref{thm:nonlinACPsol} can be replaced with the assumption that $A_\phi$ is single-valued or, 
equivalently, that the element $v$ in the definition of $\Dom(A_\phi)$ is unique. This is in particular true if the input operator $B$ of the system node is \emph{completely unbounded} in the sense that $\ran(B)\cap X=\{0\}$. Indeed, if $x\in \Dom(A_\phi)$ and if $v_1,v_2\in U$ are two elements satisfying the conditions of ``$v$" in $\Dom(A_\phi)$, then $X\ni Ax+B\phi(-v_2)-(Ax+B\phi(-v_1))=B(\phi(-v_2)-\phi(-v_1))$. This implies that $\phi(-v_1)=\phi(-v_2)$, and further that $v_1=v_2$, finally showing that $A_\phi$ is single-valued.
\end{remark}

\subsection{Proofs of the Well-Posedness Results}
\label{sec:WPproofs}

In the beginning of the proof of Theorem~\ref{thm:LPSolutionsMain} we will show that 
if $\phi : U\to U$ satisfies Assumption~\ref{ass:PhiAss} for some $\phicoeff>0$, we can reduce our analysis to the case where $\phicoeff=1$ with suitable rescaling. 
More precisely, if $\phi$ satisfies Assumption~\ref{ass:PhiAss} for some $\phicoeff>0$, then the scaled function $\phi_\phicoeff := \sqrt{\phicoeff} \phi(\sqrt{\phicoeff} \cdot)$ satisfies it with $\phicoeff=1$.
Because of this, in \cref{lem:PhiProps} and \cref{prp:LPmaxDissipative} we may also restrict our attention to this case. 
More general versions of these two results can be obtained by applying them to the function $\phi_\phicoeff$ and the impedance passive system node $S_\phicoeff = \pmatsmall{1&0\\0&1/\sqrt{\phicoeff}}S\pmatsmall{1&0\\0&1/\sqrt{\phicoeff}}$.

\begin{lemma}
\label{lem:PhiProps}
Let $U$ be a Hilbert space and $\phi:U\to U$.
 Then the following hold.
\begin{itemize}
\item[\textup{(a)}]
Suppose that $\phi$ is continuous and monotone 
and that 
$Q\in \Lin(U)$ satisfies $\re Q\geq cI$ for some $c>0$.
For all 
$u,r\in U$ the equation $y=r+Q\phi(u-y)$ has a unique solution which is determined by $y=g(u,r)$, where $g:U\times U\to U$ is a globally Lipschitz continuous function.
\item[\textup{(b)}] 
The function $\phi$ satisfies
 Assumption~\textup{\ref{ass:PhiAss}} with $\phicoeff=1$ if and only if
\eq{
2\re \iprod{\phi(v_2-y_2)-\phi(v_1-y_1)}{y_2-y_1}\leq  \norm{v_2-v_1}^2-\norm{\phi(v_2-y_2)-\phi(v_1-y_1)}^2
}
for all $v_1,v_2,y_1,y_2\in U$.
\end{itemize}
\end{lemma}

\begin{proof}
To prove part (a),
assume that
 $\re Q\geq cI$ for some $c>0$.
Defining $z=u-y$ and $v = u-r$  the equation $y=r+Q\phi(u-y)$ becomes $z+Q\phi(z)=v$.
We have from~\citel{Pau19}{Lem.~A.1(a)} that $Q$ is invertible and $\re Q\inv \geq c \norm{Q}^{-2}I$.
We can therefore write $Q\inv = S+J$, where $S$ is self-adjoint and $ S\geq c\norm{Q}^{-2}I$ and $J^\ast = -J$.
Then 
\eq{
z+Q\phi(z)= v 
\qquad 
\Leftrightarrow
\qquad 
Sz + Jz+\phi(z)= Q\inv v .
}
If we denote $\varphi(z)=Jz+\phi(z)$, then the monotonicity of $\phi$ and $J^\ast = -J$ imply that $\varphi$ is monotone as well.
As the next step we can note that 
\eq{
z+Q\phi(z)= v 
\qquad 
&\Leftrightarrow
\qquad 
S^{1/2} z + S^{-1/2}\varphi(S^{-1/2}S^{1/2} z)= S^{-1/2}Q\inv v \\
&\Leftrightarrow
\qquad 
\tilde z + \psi(\tilde z)= S^{-1/2}Q\inv v ,
}
 where we have denoted $\tilde z=S^{1/2}z$ and $ \psi(\tilde z)=S^{-1/2}\varphi(S^{-1/2} \tilde z)$.
Since $S^{-1/2}$ is self-adjoint, we have 
$\re \iprod{\psi(z_2)-\psi(z_1)}{z_2-z_1}  \geq 0 $ for $z_1,z_2\in U$.
Thus the continuous function 
$- \psi:U\to U$ defines a
(single-valued) maximally dissipative operator 
and since $U$ is a Hilbert space, this operator is $m$-dissipative by~\citel{Miy92book}{Cor.~2.27}.
Because of this, $\ran (I+ \psi)=U$ and by~\citel{Miy92book}{Cor.~2.10} the equation 
$\tilde z + \psi(\tilde z)= S^{-1/2}Q\inv v$ has a unique solution for every $v\in U$. 
Thus $ y=u-z 
=u-S^{-1/2}(I+ \psi)\inv(S^{-1/2}Q\inv (u-r))$.
We have from~\citel{Miy92book}{Cor.~2.10} that $(I+ \psi)\inv$ is a contraction, and thus in particular globally Lipschitz continuous. This further implies that the mapping $(u,r)\mapsto y$ is globally Lipschitz continuous.

It remains to prove part (b). 
 If we denote $w=v_2-v_1$, $u_1=v_1-y_1$ and $u_2=v_2-y_2$, then
$y_2-y_1 = v_2-u_2 +u_1-v_1 = w-u_2+u_1$, and 
we can see that 
 the estimate in (b) 
is equivalent to the property that
\eq{
\norm{w}^2
-2\re \iprod{\phi(u_2)-\phi(u_1)}{w-u_2+u_1}
-\norm{\phi(u_2)-\phi(u_1)}^2
\geq  0
}
for all $w,u_1,u_2\in U$.
Denoting $\gd = u_2-u_1$ and $\Delta = \phi(u_2)-\phi(u_1)$ for brevity,
the expression on the left-hand side of the above inequality can be rewritten as 
\eq{
\MoveEqLeft\norm{w}^2
-2\re \iprod{\phi(u_2)-\phi(u_1)}{w-u_2+u_1}
-\norm{\phi(u_2)-\phi(u_1)}^2\\
&=
\norm{w}^2 -2\re \iprod{\Delta}{w} +2\re \iprod{\Delta}{\gd} -\norm{\Delta}^2\\
&=\norm{w-\Delta}^2  +2\left[\re \iprod{\Delta}{\gd} -\norm{\Delta}^2\right].
}
This expression is non-negative for all $w,u_1,u_2\in U$ if and only if $\re \iprod{\Delta}{\gd}\geq \norm{\Delta}^2 $ for all $u_1,u_2\in U$, which is exactly the condition in \cref{ass:PhiAss}.
\end{proof}

We recall that a possibly nonlinear single-valued operator $A: \Dom(A)\subset X\to X$ on a Hilbert space $X$ is called \emph{dissipative} if $\re \iprod{A(x_2)-A(x_1)}{x_2-x_1}\leq 0$ for all $x_1,x_2\in \Dom(A)$~\citel{Miy92book}{Def.~2.4}. Such a dissipative operator $A$ is \emph{maximally dissipative} if $A$ does not have a proper (possibly multi-valued) dissipative extensions~\citel{Miy92book}{Def.~2.5}. Since $X$ is a Hilbert space, a dissipative operator $A$ is maximally dissipative if and only if it is $m$-dissipative, i.e., $\ran(\gl-A)=X$ for one/all $\gl>0$~\citel{Miy92book}{Lem.~2.12--2.13 \& Cor.~2.27}.
We will now define the generator of a so-called \emph{Lax--Phillips semigroup}
associated to our system~\eqref{eq:SysNodeSysMain}. It is worthwile to note that 
in this construction we consider~\eqref{eq:SysNodeSysMain} to be equipped with the output $\phi(u(t)-y(t))$.
\begin{definition}
\label{def:LPgen}
Let $S$ be an impedance passive system node on the Hilbert spaces $(U,X,U)$ and let $\phi:U\to U$. Define $\XLP:=\L^2(-\infty,0;U)\times X\times \L^2(0,\infty;U)$ and define the operator
 $\mc{A}: \Dom(\mc{A})\subset \XLP \to \XLP$
by
\eq{
&\Dom(\mc{A}) = \Bigl\{ (y,x,u)^\top\in \H^1(-\infty,0;U)\times X\times \H^1(0,\infty;U)\ \Bigm| \exists v\in U \mbox{ s.t. }
\\
& \hspace{1cm}
Ax + B\phi(u(0)- v)\in X, \  
v = C\& D \pmatsmall{x\\\phi(u(0)-v)}, \mbox{ and }
y(0)=\phi(u(0)-v)
 \Bigr\}\\
&\mc{A}\bigl( \pmatsmall{y\\x\\u}\bigr) = \pmat{  y'\\Ax + B\phi(u(0)- v(u(0),x))\\  u'},
\qquad (y,x,u)^\top \in \Dom(\mc A),
}
where $v(u(0),x)$ is the element $v$ in the definition of $\Dom(\mc{A})$. 
\end{definition}

\begin{proposition}
\label{prp:LPmaxDissipative}
Suppose that $S$ is an impedance passive system node $S$ on $(U,X,U)$
whose transfer function satisfies $\re P(\gl)\ge c_\gl I$ for some $\gl,c_\gl>0$ and
that the function $\phi:U\to U$ satisfies Assumption~\textup{\ref{ass:PhiAss}} with $\phicoeff=1$.
Then $\mc{A}$ in Definition~\textup{\ref{def:LPgen}} is a single-valued, densely defined and $m$-dissipative (nonlinear) operator.
\end{proposition}

\begin{proof}
 We first note that the property $Ax + B\phi(u(0)- v)\in X$ in the definition of $\Dom(\mc{A})$  implies that
$(x,\phi(u(0)-v))^\top\in \Dom(S)=\Dom(C\&D)$.
We begin by showing that $\Dom(\mc{A})$ is dense in $\XLP$.
To this end, let $\lambda>0$
and denote $R_\gl = (\gl-A)\inv$. 
If we let $x_0\in \Dom(A)$ and define $x=x_0+R_\gl B\phi(0)$ and $v=Cx_0 + P(\gl)\phi(0)$, then
$Ax+B\phi(0)=Ax_0 +\gl R_\gl B\phi(0)\in X$ and the properties of the system node $S$ imply that
 $C\& D \pmatsmall{x\\\phi(0)}= Cx_0+P(\gl)\phi(0)=v$.
 This means that if we define $D_0:=\{x_0+R_\gl B\phi(0)\,|\  x_0\in \Dom(A)\}$ and $\DLP:=\{(y,x,u)^\top\in \H^1(-\infty,0;U)\times D_0 \times \H^1(0,\infty;U)\;|\; u(0)= C\&D \pmatsmall{x\\\phi(0)}, y(0)=\phi(0)\}$, then we have $\DLP\subset \Dom(\mc{A})\subset \XLP$. 
 But since $D_0$ is dense in $X$, it is easy to see that $\DLP$ is dense in $\XLP$. This further implies that $\Dom(\mc A)$ is dense in $\XLP$.

We will now show that $\mc{A}$ is a single-valued operator.
This is clearly true if the element $v$ in the definition of $\Dom(\mc{A})$ is unique\footnote{\cref{lem:PhiProps}(b) can be used to show that also the converse holds.}.
Let $(y,x,u)^\top\in \Dom(\mc{A})$
and let $v_1\in U$ and $v_2\in U$ be two elements such that 
\eq{
Ax + B\phi(u(0)- v_k)\in X \mbox{ and } v_k = C\& D \pmatsmall{x\\\phi(u(0)-v_k)}, \qquad k=1,2.
}
If we denote 
$\phi_1 := \phi(u(0)-v_1)$ and
$\phi_2 := \phi(u(0)-v_2)$  for brevity, then the first equations imply $X\ni Ax+B\phi_2 - (Ax+B\phi_1) = A0+ B(\phi_2-\phi_1)$, and thus $(0,\phi_2-\phi_1)^\top\in \Dom(S)$. 
Moreover, $C\&D \pmatsmall{0\\\phi_2-\phi_1} = C\& D \bigl(\pmatsmall{x\\\phi_2}-\pmatsmall{x\\\phi_1}\bigr)= v_2-v_1 $.
The impedance passivity of $S$ and Assumption~\ref{ass:PhiAss} imply that
\eq{
0 &= \re\iprod{A\&B \pmatsmall{0\\\phi_2-\phi_1}}{0}_X
\geq  \re\iprod{C\&D \pmatsmall{0\\\phi_2-\phi_1}}{\phi_1-\phi_2}_U\\
& = \re\iprod{v_2-v_1}{\phi_1-\phi_2}_U\\
&= \re\iprod{u(0)-v_1 - (u(0)-v_2)}{\phi(u(0)-v_1)-\phi(u(0)-v_2)}_U\\
&\geq \norm{\phi(u(0)-v_1)-\phi(u(0)-v_2)}^2.
}
Thus $\phi_2=\phi_1$, which implies $v_2 = C\& D \pmatsmall{x\\\phi_2}=C\& D \pmatsmall{x\\\phi_1} = v_1$.
Thus the element $v$ in the definition of $\Dom(\mc{A})$ is unique and $\mc{A}$ is single-valued.

To show that $\mc{A}$ is dissipative, let 
$(y_1,x_1,u_1)^\top,(y_2,x_2,u_2)^\top\in \Dom(\mc{A})$. 
Denote $v_1:=v(u_1(0),x_1)$, $v_2:= v(u_2(0),x_2)$
$\phi_1 := \phi(u_1(0)-v(u_1(0),x_1))$, and
$\phi_2 := \phi(u_2(0)-v(u_2(0),x_2))$ 
 for brevity.
Using the impedance passivity of the system node $S$,~\citel{FarWeg16}{Thm.~5}, and
 Lemma~\ref{lem:PhiProps}(b) we can estimate
\begin{align*}
&\re \Bigl\langle\mc{A}\bigl( \pmatsmall{y_2 \\ x_2\\u_2}\bigr)-\mc{A}\bigl( \pmatsmall{y_1\\ x_1\\u_1}\bigr), \pmatsmall{y_2\\x_2\\u_2}- \pmatsmall{y_1\\x_1\\u_1}\Bigr\rangle\\
&= \re \iprod{ y_2' -  y_1'}{y_2-y_1}_{\L^2(-\infty,0;U)} + \re \iprod{Ax_2+B\phi_2-Ax_1 - B\phi_1}{x_2-x_1}_X\\
&\quad + \re \iprod{ u_2' -  u_1'}{u_2-u_1}_{\L^2(0,\infty;U)}\\
&=
\frac{1}{2}\int_{-\infty}^0 \frac{\d}{\d t} \left\Vert y_2(t)-y_1(t)\right\Vert_U^2 \d t
+ \re \iprod{A(x_2-x_1)+B(\phi_2-\phi_1)}{x_2-x_1}_X \\
&\quad + \frac{1}{2}\int_0^\infty \frac{\d}{\d t} \left\Vert u_2(t)-u_1(t)\right\Vert_U^2 \d t
\\
&\leq 
\frac{1}{2}\left\Vert y_2(0)-y_1(0)\right\Vert^2+ \re \iprod{C\& D\pmatsmall{x_2-x_1\\\phi_2-\phi_1}}{\phi_2-\phi_1}_X-\frac{1}{2}\left\Vert u_2(0)-u_1(0)\right\Vert_U^2\\
&=
\frac{1}{2}\left\Vert\phi(u_2(0)-v_2)-\phi(u_1(0)-v_1)\right\Vert_U^2-\frac{1}{2}\left\Vert u_2(0)-u_1(0)\right\Vert_U^2\\
&\quad + \re \iprod{v_2-v_1}{\phi(u_2(0)-v_2)-\phi(u_1(0)-v_1)}_U
 \leq 0.
\end{align*}
This completes the proof that $\mc{A}$ is dissipative.

In the last part of the proof, we will show that $\mc{A}$ is $m$-dissipative. By~\citel{Miy92book}{Lem.~2.13} it suffices to show that
 $\ran(\lambda-\mc{A})=\XLP
 $ for some $\lambda > 0$.
To this end, let 
$(y_1,x_1,u_1)^\top\in \XLP$
be arbitrary and let $\lambda>0$ be such that $\re P(\lambda)\geq c_\lambda I$ for some $c_\gl>0$.
If we define 
$u: \zinf \to U$
 by  $u(t) = \int_{ t}^\infty e^{\lambda(t-s)}u_1(s)\d s $, then similarly as in~\citel{StaWei02}{Prop.~6.4} we have $u\in \H^1(0,\infty;U)$ and $\lambda u- u'=u_1$.
In particular, we have $u(0)=\hat u_1(\lambda)\in U$.
The definition of $\mc{A}$ shows that we have
$(y,x,u)^\top\in \Dom(\mc{A})$ and
 $(\lambda-\mc{A})\pmatsmall{y\\x\\u}=\pmatsmall{y_1\\x_1\\u_1}$ if 
$y\in \H^1(-\infty,0;U)$,
$x\in X$ and $v\in U$ are such that
\eqn{
\label{eq:LPmaxequations}
\begin{cases}
\lambda y- y' = y_1, \qquad \mbox{on }(-\infty,0)\\
y(0) = \phi (\hat u_1(\lambda)-v)\\
(\lambda-A)x - B\phi(\hat u_1(\lambda)-v) = x_1\\
v = C\& D \pmatsmall{x\\ \phi(\hat u_1(\lambda)-v)}.
\end{cases}
}
To solve equations~\eqref{eq:LPmaxequations} 
we denote $R_\lambda:=(\lambda-A)\inv$ for brevity
and note that by Lemma~\ref{lem:PhiProps}(a)
we can define $v$ as the (unique) element of $U$ such that
\begin{equation*}
        v = CR_\lambda x_1 + P(\lambda)\phi(\hat u_1(\lambda)-v).
\end{equation*}
If we define $x = R_\lambda x_1 + R_\lambda B\phi(\hat u_1(\lambda)-v)$, 
then $(\lambda-A)x - B\phi(\hat u_1(\lambda)-v) = x_1$ in~\eqref{eq:LPmaxequations}, and
a direct computation 
using the properties of system nodes
shows that
\eq{
\MoveEqLeft C\& D \pmat{x\\ \phi(\hat u_1(\lambda)-v)}
=C\& D \pmat{R_\lambda x_1 + R_\lambda B \phi(\hat u_1(\lambda)-v)\\ \phi(\hat u_1(\lambda)-v)}\\
&= CR_\lambda x_1 + P(\lambda)\phi(\hat u_1(\lambda)-v)
= v.}
Finally, we define $y$ as the solution of the boundary value problem
\eq{
&\lambda y(t)- y'(t) = y_1(t), \qquad t\in(-\infty,0)\\
&y(0) = \phi (\hat u_1(\lambda)-v).
}
Since $y_1\in \L^2(-\infty,0;U)$, we have 
$
y(t) = e^{\lambda t} \phi(\hat u_1(\lambda)-v) + \int_t^0 e^{\lambda (t-s)}y_1(s)\d s
$
for $t\le 0$,
and $y\in \H^1(-\infty,0;U)$ and $y(0)=\phi(\hat u_1(\lambda)-v)$, as required.
Since 
$(y_1,x_1,u_1)^\top\in \XLP$
was arbitrary, we have shown that 
$\ran (\lambda-\mc{A})=\XLP$,
and thus $\mc{A}$ is $m$-dissipative.
\end{proof}

\begin{remark}
\label{rem:PhiMinimality}
The proof of Proposition~\textup{\ref{prp:LPmaxDissipative}} shows that the condition for $\phi$ in Assumption~\textup{\ref{ass:PhiAss}} is in a certain sense \emph{minimal} for  the dissipativity of $\mc{A}$ and, consequently, also for the estimate~\eqref{eq:ContractivityEstim} in Theorem~\textup{\ref{thm:LPSolutionsMain}}.
 Indeed, if $S$ is \emph{impedance energy preserving}, then $\re \iprod{Ax+Bu}{x}= \re\iprod{C\&D \pmatsmall{x\\ u}}{u}_U$ for all $(x,u)^\top\in \Dom(S)$ and  Lemma~\textup{\ref{lem:PhiProps}} and the estimates in the proof of Proposition~\textup{\ref{prp:LPmaxDissipative}(b)} imply that $\mc{A}$ is dissipative if and only if $\phi$ satisfies Assumption~\textup{\ref{ass:PhiAss}}.
\end{remark}

We are now in a position to prove \cref{thm:LPSolutionsMain}.

\begin{proof}[Proof of Theorem~\textup{\ref{thm:LPSolutionsMain}}]
We begin by noting that if we define $\tilde u = \phicoeff^{-1/2}u$ and $\tilde y = \phicoeff^{-1/2}y$, then~\eqref{eq:SysNodeSysMain} can be rewritten as
\eq{
\left[\begin{matrix}\dot{x}(t)\\\tilde y(t)\end{matrix}\right] = S_\phicoeff\left[\begin{matrix}x(t)\\\phi_\phicoeff(\tilde u(t)-\tilde  y(t))\end{matrix}\right], \qquad t\ge 0,
}
where $\phi_\phicoeff:=\sqrt{\phicoeff}\phi(\sqrt{\phicoeff}\cdot): U\to U$  
and where
\eq{
S_\phicoeff&= \pmat{1&0\\0&\phicoeff^{-1/2}}\pmat{A\&B\\ C\&D}\pmat{1&0\\0&\phicoeff^{-1/2}}, \\
D(S_\phicoeff) 
 &= \left\{(x,u)^\top\in X\times U \ \middle| \  Ax + \phicoeff^{-1/2} Bu\in X\right\}
}
is an impedance passive system node on the spaces $(U,X,U)$.
The function $\phi_\phicoeff $ satisfies Assumption~\ref{ass:PhiAss} with ``$\phicoeff=1$''.
It is straightforward to check that the claims of the theorem for the system node $S_\phicoeff$, the associated well-posed linear system $\Sigma_\phicoeff$, and
 $\phi_k$ with $\phicoeff=1$ immediately imply the original claims for $S$, $\Sigma$, and $\phi$.
Because of this, we can without loss of generality assume that $\phicoeff=1$.

We next note that it suffices to prove the claims in the situation where $u\in \L^2(0,\infty;U)$ or $u\in \H^1(0,\infty;U)$ instead of $u\in \Lploc[2](0,\infty;U)$ or $u\in \Hloc[1](0,\infty;U)$, respectively. Indeed, if $u\in \Lploc[2](0,\infty;U)$ (resp. $u\in \Hloc[1](0,\infty;U)$), for an arbitrary $T>0$ we can define 
$u_T\in \L^2(0,\infty;U)$ (resp. $u_T\in \H^1(0,\infty;U)$) so that $u$ and $u_T$ coincide on $[0,T]$ and $u_T$ vanishes on $\zinf[T+1]$.
Once we prove the claims of Theorem~\ref{thm:LPSolutionsMain} for such truncated input functions, we establish the existence of classical or generalised solutions $(x_T,u_T,y_T)$ satisfying $x_T(0)=x_0$. 
Moreover, by Remark~\ref{rem:SysNodeWPIncrEstimate} the state trajectory $x_T$ and output $y_T$ on $[0,T]$ are uniquely determined by $x_0$ and $u_T$ on $[0,T]$. 
Because of this we can define $x\in C(\zinf;X)$ and $y\in \Lploc[2](0,\infty;U)$ such that
 $x$ coincides with $x_T$ and $y$ coincides with $y_T$ on $[0,T]$ for every $T>0$.
 It is straightforward to deduce from \cref{def:SysNodeStates} and the properties of the solutions $(x_T,u_T,y_T)$ for $T>0$ that $(x,u,y)$ is a 
generalised solution of~\eqref{eq:SysNodeSysMain} whenever $x_0\in X$ and $u\in \Lploc[2](0,\infty;U)$ and that it is a
classical solution of~\eqref{eq:SysNodeSysMain} 
and $y\in \Hloc[1](0,\infty;U)$
whenever $u\in \Hloc[1](0,\infty;U)$ and  $(x_0,u(0))^\top\in \Dcomp$. Finally, the formulas~\eqref{eq:LPSemigStateAndOutput} and the estimate~\eqref{eq:ContractivityEstim} for two solutions follow directly from the corresponding properties of the solutions $(x_T,u_T,y_T)$ for $T>0$.

Proposition~\ref{prp:LPmaxDissipative} shows that $\mc{A}$ is single-valued, densely defined, and $m$-dissipative (equivalently, maximally dissipative). 
We begin by showing the existence of classical solutions of~\eqref{eq:SysNodeSysMain}. To this end, let $x_0\in X$ and $u\in \H^1(0,\infty;U)$ be such that $(x_0,u(0))^\top\in\Dcomp$.
We can then choose
 $y_0\in \H^1(-\infty,0;U)$ so that $(y_0,x_0,u)^\top\in \Dom(\mc{A})$.
We have from~\citel{Miy92book}{Thm.~4.20} that $\mc A$ is the generator of a nonlinear contraction semigroup $\TLP$ on $\XLP=\L^2(-\infty,0;U)\times X\times \L^2(0,\infty;U)$. We define $x_{e0}=(y_0,x_0,u)^\top$ and $x_e(t)=\TLP_t (x_{e0})$ for $t\ge 0$. We then have from~\citel{Miy92book}{Thm.~4.10, Cor.~3.7} that
$x_e:\zinf\to \XLP $
 is a Lipschitz continuous, right-differentiable, and almost everywhere differentiable function 
such that $x_e(0)=(y_0,x_0,u)^\top$, 
 $x_e(t)\in \Dom(\mathcal{A})$ for all $t\geq 0$,
 $t\mapsto\mathcal{A}(x_e(t))$ is right-continuous,
and
\eqn{
\label{eq:nACPinSysNodeProof}
\dot x_e(t)=\mathcal{A}(x_e(t)) \qquad \mbox{for a.e. } t\ge 0. 
}
Together these properties also imply that $x_e\in \Wloc[1](0,\infty;\XLP )$.
For every $t\geq 0$ we can decompose $x_e(t)$ into parts as
$x_e(t)=(\tilde y(t,\cdot),x(t),\tilde u(t,\cdot))^\top$, and
we denote by
$y(t):=v(\tilde u(t,0),x(t))$ the element `$v$' in the definition of $\Dom(\mathcal{A})$ corresponding
to $(\tilde y(t,\cdot),x(t),\tilde u(t,\cdot))^\top\in \Dom(\mc{A})$.
With this notation equation~\eqref{eq:nACPinSysNodeProof} is equivalent to the equations
\begin{subequations}
\eqn{
\label{eq:LPACPyt}
 \frac{\partial\tilde y}{\partial t}(t,s) &= \frac{\partial \tilde y}{\partial s}(t,s)\\
\label{eq:LPACPytBC}
\tilde y(t,0) &= \phi(\tilde u(t,0)-y(t))\\
\label{eq:LPACPx}
  \dot x(t) &= Ax(t) + B\phi(\tilde u(t,0)-y(t))\\
\label{eq:LPACPut}
 \frac{\partial\tilde u}{\partial t}(t,r)
 &=\frac{\partial\tilde u}{\partial r}(t,r) \\
\label{eq:LPACPy}
y(t) &= C\& D \pmatsmall{x(t)\\ \phi(\tilde u(t,0)-y(t))}
}
\end{subequations}
for a.e. $t\geq0$, a.e. $s\in \zabr{-\infty}{0}$, and a.e. $r\in\zinf$. Moreover, equations~\eqref{eq:LPACPytBC} and~\eqref{eq:LPACPy} hold
for \emph{every} $t\geq 0$, and we have $\tilde u(t,\cdot)\in \H^1(0,\infty;U)$ and $\tilde y(t,\cdot)\in \H^1(-\infty,0;U)$ for all $t\geq 0$.
The property $\tilde u(0,\cdot) = u\in \H^1(0,\infty;U)$ together with the equation~\eqref{eq:LPACPut} 
imply, as in~\citel{TucWei09book}{Ex.~2.3.7}, that
$\tilde u(t,r) = u(t+ r)$ for all $t\geq 0$ and $r\in \zinf$.
In particular, we have $\tilde u(t,0)=u(t)$ for all $t\geq 0$.
The properties of $x_e$ imply that $x\in \Wloc[1](0,\infty;X)$.

We will now show that $t\mapsto \phi(u(t)-y(t))\in \Hloc[1](0,\infty;U)$. We begin by showing that $t\mapsto \phi(u(t)-y(t))$ is locally square integrable. 
To this end, let $\lambda>0$ be such that $\re P(\lambda)\geq c_\lambda I$ for some $c_\gl>0$.
Equation~\eqref{eq:LPACPx} implies that for a.e. $t\geq 0$ we have
$\lambda x(t)-\dot x(t)=(\lambda-A)x(t)-B\phi(t)$, and further $x(t)=R_\lambda(\gl x(t)-\dot x(t))+R_\lambda B\phi(t)$, where we have denoted $R_\lambda:=(\lambda-A)\inv$ and $\phi(t):=\phi(u(t)-y(t))$.
Substituting this into~\eqref{eq:LPACPy} and using $R_\lambda(\gl x(t)-\dot x(t))\in \Dom(A)$ we get, for a.e. $t\geq 0$,
\eq{
y(t)&= C\&D \pmat{x(t)\\\phi(t)}
= CR_\lambda(\gl x(t)-\dot x(t)) + C\&D \pmat{R_\lambda B\phi(t)\\\phi(t)}\\
&=  CR_\lambda(\gl x(t)-\dot x(t)) + P(\lambda)\phi(u(t)-y(t)).
}
We have from Lemma~\ref{lem:PhiProps}(a) that there exists a globally Lipschitz continuous function $g:U\times U\to U$ such that $y(t)=g(u(t),CR_\lambda(\gl x(t)-\dot x(t)))$ for a.e. $t\geq 0$. 
But since $u\in \Hloc[1](0,\infty;U)$ and $t\mapsto CR_\lambda(\gl x(t)-\dot x(t))\in \Lploc[\infty](0,\infty;U)$, we deduce that $y\in \Lploc[2](0,\infty;U)$, and the global Lipschitz continuity of $\phi$ also implies
$t\mapsto \phi(u(t)-y(t))\in \Lploc[2](0,\infty;U)$.
Since $t\mapsto \phi(u(t)-y(t))\in \Lploc[2](0,\infty;U)$, we have from~\citel{TucWei09book}{Ex.~4.2.7} that the unique solution of~\eqref{eq:LPACPyt} for a.e. $t\geq 0$ and a.e. $s\leq 0$
satisfies 
$\tilde y(t,s) =\tilde y( t+ s,0)$ for all $s\in [-t,0]$.
Equation~\eqref{eq:LPACPytBC} for all $t\geq 0$ therefore implies that 
for every fixed $\tau>0$
 we have
\eq{
\phi(u(t)-y(t)) = \phi(\tilde u(t,0)-y(t)) 
= \tilde y(\tau,t- \tau), \qquad t\in [0,\tau].
}
But since $\tilde y(\tau,\cdot)\in \H^1(-\infty,0;U)$ and since this identity holds for every $\tau>0$, we can deduce that $t\mapsto \phi(u(t)-y(t))\in \Hloc[1](0,\infty;U)$.
Combining this property, our knowledge that $\dot x\in \L^\infty(0,\infty;X)$,
and the fact that~\eqref{eq:LPACPx} holds for a.e. $t\geq 0$
with the results~\citel{TucWei14}{Prop.~4.6 \& 4.7} shows that 
$x\in C^1(\zinf;X)$, $y\in \H^1(0,\infty;U)$, that \eqref{eq:LPACPx} and~\eqref{eq:LPACPy} in fact hold for all $t\ge 0$, and that $x$ and $y$ satisfy~\eqref{eq:LPSemigStateAndOutput}.
We therefore have from Definition~\ref{def:SysNodeStates} that $(x,u,y)$ is a classical solution  of~\eqref{eq:SysNodeSysMain}.

We will now show that classical and generalised solutions of~\eqref{eq:SysNodeSysMain} satisfy the incremental estimate~\eqref{eq:ContractivityEstim}. 
Let $(x_1,u_1,y_1)$ and $(x_2,u_2,y_2)$ be two classical solutions of~\eqref{eq:SysNodeSysMain} and denote
$\phi_k(s)=\phi(u_k(s)-y_k(s))$ for $k=1,2$ for brevity.
Using the fact that~\eqref{eq:SysNodeSysMain} holds for almost every $t\geq 0$, the impedance passivity of the system node $S$, and Lemma~\ref{lem:PhiProps}(b) we have that
\begin{align*}
    \frac{\d}{\d s}\Vert x_2(s)-x_1(s)\Vert^2 
&=2\re\Iprod{A\& B \pmat{x_2(s)-x_1(s)\\\phi_2(s)-\phi_1(s)}}{x_2(s)-x_1(s)}_X\\
&\leq 2\re\Iprod{C\& D \pmat{x_2(s)-x_1(s)\\\phi_2(s)-\phi_1(s)}}{\phi_2(s)-\phi_1(s)}_U\\
&= 2\re\Iprod{y_2(s)-y_1(s)}{\phi_2(s)-\phi_1(s)}_U\\
    &\leq \Vert u_2(s)-u_1(s)\Vert^2 - \Vert\phi(u_2(s)-y_2(s))-\phi(u_1(s)-y_1(s))\Vert^2
\end{align*}
 for almost every $s\geq 0$.
Rearranging the terms and integrating the above estimate with respect to $s$ from $0$ to $t$ shows that the classical solutions
 satisfy the estimate~\eqref{eq:ContractivityEstim}.
Moreover, it is easy to see that this property together with 
 Definition~\ref{def:SysNodeStates} and the global Lipschitz continuity of $\phi$ imply 
that also all generalised solutions (as limits of classical solutions) satisfy~\eqref{eq:ContractivityEstim}.

Now let $x_0\in X$ and $u\in \L^2(0,\infty;U)$.
We will show that~\eqref{eq:SysNodeSysMain} has a generalised solution $(x,u,y)$ satisfying $x(0)=x_0$.
Let $\tau>0$ be arbitrary
 and define $ y_0=0\in \L^2(-\infty,0;U)$ 
and $x_{e0}=( y_0,x_0,u)^\top\in \XLP $.
We define
$x_e(t)=\TLP_t(x_{e0}) $, $t\geq 0$,
where $\TLP$ is again the nonlinear contraction semigroup generated by $\mc A$.
Then $x_e\in C(\zinf;\XLP )$, $x_e(0)=x_{e0}$
 and $x_e(t)$ has the structure $x_e(t)=(\tilde y(t,\cdot),x(t),\tilde u(t,\cdot))^\top\in \XLP $ for all $t\geq 0$.
Since $\Dom(\mc{A})$ is dense in $\XLP $ by Proposition~\ref{prp:LPmaxDissipative}, there exists a sequence 
$(x_{e0}^k)_{k\in\N}\subset \Dom(\mc{A})$ such that $\norm{x_{e0}^k-x_{e0}}\to 0$ as $k\to\infty$.
If we define $x_e^k(t)=\TLP_t(x_{e0}^k)$ for $t\geq 0$, then 
the contractivity of $\TLP$ implies 
\eqn{
\label{eq:LPgensolconv}
\sup_{t\in [0,\tau]}
\norm{x_e(t)-x_e^k(t)}_{\XLP }
\leq
\norm{x_{e0}-x_{e0}^k}_{\XLP }\to 0, \qquad k\to\infty.
}%
Similarly as above, for every $k\in\N$ the element $x_e^k =(\tilde y^k,x^k,\tilde u^k)^\top$ 
 defines 
a classical solution $(x^k,u^k,y^k)$
of~\eqref{eq:SysNodeSysMain}.
 Moreover, these classical solutions satisfy the estimate~\eqref{eq:ContractivityEstim}.
The property that $x_e^k(0)\to x_e(0)$ in $\XLP $ implies that $\tilde u^k(0,\cdot)\to u$ in $\L^2(0,\infty;U)$.
As above, we have $\tilde u^k(t,r) = u^k(t+ r)$ for all $t\geq 0$ and $r\ge 0$, and thus we can deduce that $\Pt[\tau]u^k\to \Pt[\tau]u$ in $\L^2(0,\tau;U)$ as $k\to\infty$. 
The continuity of $x_e$
implies $x\in C(\zinf;X)$
 and~\eqref{eq:LPgensolconv}
implies $\norm{\Pt[\tau]x^k-\Pt[\tau]x}_{C([0,\tau];X)}\to 0$ as $k\to\infty$.

In the final step of the proof we will show that there exists $y\in \Lploc[2](0,\infty;U)$ such that $\Pt[\tau]y^k\to \Pt[\tau]y$ in $\L^2(0,\tau;U)$ as $k\to\infty$.
This will establish that $(x,u,y)$ is a generalised solution of~\eqref{eq:SysNodeSysMain}.
Similarly as above, we can see that $(x^k,u^k,y^k)$ satisfy%
\begin{subequations}%
\eqn{
x^k(t)& = \T_tx^k(0)+\Phi_t \Pt[t] \phi(u^k-y^k)
\label{eq:StateEq_Strong} \\
 \Pt[t] y^k &= \Psi_t x^k(0) + \mathbb{F}_t\Pt[t]  \phi(u^k-y^k).\label{eq:OutputEq_Strong}
}    
\end{subequations}
for all $k\in\N$ and $t\ge 0$. Moreover, the  estimate~\eqref{eq:ContractivityEstim} implies that for every $t>0$
\eq{
\norm{\Pt[t] \phi(u^k\hspace{-.2ex}-\hspace{-.2ex}y^k)-\Pt[t] \phi(u^n\hspace{-.2ex}-\hspace{-.2ex}y^n)}_{\L^2(0,t)}^2\leq
\Vert x^k(0)\hspace{-.2ex}-\hspace{-.2ex}x^n(0)\Vert^2
    + \norm{\Pt[t]u^k\hspace{-.2ex}-\hspace{-.2ex}\Pt[t]u^n}_{\L^2(0,t)}^2
}
for all $k,n\in\N$.
For every $t>0$ the sequences $(x^k(0))_k\subset X$ and $(\Pt[t]u^k)_k\subset \L^2(0,t;U)$ are convergent, and the above estimate implies that $(\Pt[t]\phi(u^k-y^k))_k\subset \L^2(0,t;U)$ is a Cauchy sequence, and therefore has a limit $\psi_t\in \L^2(0,t;U)$. The above estimate also implies that if $t_2>t_1>0$ then $\Pt[t_1]\psi_{t_2}=\psi_{t_1}$.
We can therefore define $\psi\in \Lploc[2](0,\infty;U)$ such that $\Pt[t]\psi=\psi_t$ for every $t>0$, and define $y\in \Lploc[2](0,\infty;U)$ by $y=\Psi_\infty x_0+\F_\infty\psi$. 
Since $\Sigma $ is a well-posed system and since $\Pt[\tau]\phi(u^k-y^k)\to \psi_\tau$ in $\L^2(0,\tau;U)$ as $k\to\infty$, the identity~\eqref{eq:OutputEq_Strong} with $t=\tau$ and the definition of $y$ imply that $\Pt[\tau] y^k\to \Pt[\tau] y$ in $\L^2(0,\tau;U)$ as $k\to\infty$.
Since $\tau>0$ was arbitrary, we have that $(x,u,y)$ is a generalised solution of~\eqref{eq:SysNodeSysMain}.
Since $\phi$ is globally Lipschitz, we also have that $\phi(u-y)\in\Lploc[2](0,\infty;U)$ and
$\Pt[\tau]\phi(u^k-y^k) 
\to
\Pt[\tau]\phi(u-y)$ as $k\to\infty$. The fact that we also have
$\Pt[\tau]\phi(u^k-y^k) 
\to \Pt[\tau]\psi_\tau$
and the uniqueness of limits imply that $\Pt[\tau]\psi_\tau = \Pt[\tau]\phi(u-y)$. 
Thus the definition of $y$ implies that~\eqref{eq:LPSemigOutput} holds.
Finally, the well-posedness of $\Sigma$ implies that the right-hand side of~\eqref{eq:StateEq_Strong} converges to $\T_tx(0)+\Phi_t\Pt[t]\phi(u-y)$ for every $t>0$ as $k\to\infty$. Since we also have $x^k(t)\to x(t)$ as $k\to \infty$ for all $t>0$, 
 the uniqueness of limits implies that~\eqref{eq:LPSemigState} holds.
\end{proof}

\begin{proof}[Proof of Theorem~\textup{\ref{thm:nonlinACPsol}}]
We will first show that $A_\phi$ is a single-valued  $m$-dissipative operator.
The operator $A_\phi$ is single-valued if
the element `$v$' in the definition of $\Dom(A_\phi)$ is unique
(and these two properties are, in fact, equivalent).
For proving the uniqueness of $v$ we will use our assumption that $\re \iprod{\phi(u_2)-\phi(u_1)}{u_2-u_1}> 0$ whenever $\phi(u_1)\neq \phi(u_2)$.
Let $x\in \Dom(A_\phi)$
and let $v_1\in U$ and $v_2\in U$ be two elements satisfying
\eq{
Ax + B\phi(- v_k)\in X \mbox{ and } v_k = C\& D \pmat{x\\\phi(-v_k)}, \qquad k=1,2.
}
If we denote 
$\phi_1 := \phi(-v_1)$ and
$\phi_2 := \phi(-v_2)$  for brevity, then  $X\ni Ax+B\phi_2 - (Ax+B\phi_1) = A0+ B(\phi_2-\phi_1)$, and thus $(0,\phi_2-\phi_1)^\top\in \Dom(S)$. 
Moreover, $C\&D \pmatsmall{0\\\phi_2-\phi_1} = C\& D \bigl(\pmatsmall{x\\\phi_2}-\pmatsmall{x\\\phi_1}\bigr)= v_2-v_1 $.
The impedance passivity of $S$ 
 implies that
\eq{
0 &= \re\iprod{A\&B \pmatsmall{0\\\phi_2-\phi_1}}{0}_X
\geq  \re\iprod{C\&D \pmatsmall{0\\\phi_2-\phi_1}}{\phi_1-\phi_2}_U\\
 &= \re\iprod{v_2-v_1}{\phi_1-\phi_2}_U
\geq 0.
}
Thus $\re\iprod{v_2-v_1}{\phi_1-\phi_2}_U = 0$, which based on our assumption implies that $\phi(-v_1)=\phi(-v_2)$, and further that $v_2-v_1=C\& D  \pmatsmall{0\\\phi_2-\phi_1}=0$. Thus the element $v$ in the definition of $\Dom(A_\phi)$ is unique and $A_\phi$ is single-valued.

To show that $A_\phi$ is dissipative, let $x_1,x_2\in \Dom(A_\phi)$ and let $v_1,v_2\in U$ denote the corresponding elements in the definition of $\Dom(A_\phi)$.
Then the impedance passivity of $S$ and the monotonicity of $\phi$ imply that 
\eq{
\MoveEqLeft\re\iprod{A_\phi(x_2)-A_\phi(x_1)}{x_2-x_1}
 = \re\iprod{Ax_2 + B\phi(-v_2) -(Ax_1 + B\phi(-v_1))}{x_2-x_1}\\
 &= \re\Bigl\langle A\& B \pmat{x_2-x_1\\ \phi(-v_2) - \phi(-v_1)},x_2-x_1\Bigr\rangle\\
&\leq \re\Bigl\langle C\&D \pmat{x_2-x_1\\ \phi(-v_2) - \phi(-v_1)},\phi(-v_2) - \phi(-v_1)\Bigr\rangle\\
&= \re\iprod{v_2-v_1}{\phi(-v_2) - \phi(-v_1)}
\leq 0.
}
Thus $A_\phi$ is dissipative.

We will now show that $A_\phi$ is $m$-dissipative. By~\citel{Miy92book}{Lem.~2.13} it suffices to show that
 $\ran(\lambda-A_\phi)=X $ for some $\lambda > 0$.
To this end, let 
$x_1\in X$
be arbitrary and let $\lambda>0$ be such that $\re P(\lambda)\geq c_\lambda I$ for some $c_\gl>0$.
By Lemma~\ref{lem:PhiProps}(a) we can define
 $v$ as the (unique) element in $U$ such that
\begin{equation*}
        v = CR_\lambda x_1 + P(\lambda)\phi(-v).
\end{equation*}
If we define $x = R_\lambda x_1 + R_\lambda B\phi(-v)$, 
then $(\lambda -A)x - B\phi(-v) = x_1$, and
a direct computation 
shows that
\eq{
\MoveEqLeft C\& D \pmat{x\\ \phi(-v)}
=C\& D \pmat{R_\lambda x_1 + R_\lambda B \phi(-v)\\ \phi(-v)}
= CR_\lambda x_1 + P(\lambda)\phi(-v)
= v.}
Thus $x\in \Dom(A_\phi)$ and $(\gl-A_\phi) (x) = x_1$. Since $x_1\in X$ was arbitrary, we conclude that $A_\phi$ is $m$-dissipative.

To verify the estimate~\eqref{eq:NLFBContractivity}, let
 $x_1,x_2\in C(\zinf;U)$ be two strong solutions of~\eqref{eq:nACP} on $\zinf$ 
with corresponding outputs $y_1$ and $y_2$.
If we denote $\phi_k(t)=\phi(-y_k(t))$, $k=1,2$, for brevity,
then
 the impedance passivity of $S$
and the monotonicity of $\phi$ imply that
  for a.e. $t\ge 0$
\eq{
\frac{1}{2}\frac{\d}{\d t}\norm{x_2(t)-x_1(t)}^2
&=\re\Bigl\langle A\& B \pmat{x_2(t)-x_1(t)\\\phi_2(t)-\phi_1(t)},x_2(t)-x_1(t)\Bigr\rangle_X\\
&\leq \re\Bigl\langle C\& D \pmat{x_2(t)-x_1(t)\\\phi_2(t)-\phi_1(t)},\phi_2(t)-\phi_1(t)\Bigr\rangle_U\\
&= \re\Iprod{y_2(t)-y_1(t)}{\phi_2(t)-\phi_1(t)}_U\leq 0.
}
Because of this, $\norm{x_2(t)-x_1(t)}\leq \norm{x_2(0)-x_1(0)}$ for all $t\geq 0$.
It is straightforward to verify that 
also the generalised solutions of~\eqref{eq:nACP} (as limits of strong solutions) satisfy the same estimate.
In particular, strong and generalised solutions of~\eqref{eq:nACP} are uniquely determined by $x(0)$.

We will now study the existence of strong solutions of~\eqref{eq:nACP}.
 To this end, let $x_0\in \Dom(A_\phi)$.
Since $A_\phi$ is a single-valued and $m$-dissipative (equivalently, maximally dissipative) operator we have from~\citel{Miy92book}{Thm.~4.20} that $ A_\phi$ is the generator of a nonlinear contraction semigroup $ \T^\phi$ on $\overline{\Dom(A_\phi)}$, which is a closed and convex subset of $X$. We define $x(t)=\T^\phi_t (x_0)$ for $t\ge 0$.
 We then have from~\citel{Miy92book}{Thm.~4.10, Cor.~3.7} that
$x:\zinf\to X $
 is a locally Lipschitz continuous, right-differentiable, and almost everywhere differentiable function 
such that $x(0)=x_0$, 
 $x(t)\in \Dom(A_\phi)$ for all $t\geq 0$,
 $t\mapsto A_\phi(x(t))$ is right-continuous,
and
\eqn{
\label{eq:nACPinSGproof}
\dot x(t)=A_\phi(x(t)) \qquad \mbox{for a.e. } t\ge 0. 
}
Moreover, by contractivity of $\T^\phi$ and by~\citel{Miy92book}{Cor.~3.7} both $x$ and $\dot x$ are essentially bounded.
Together these properties imply that $x\in \W^{1,\infty}(0,\infty;X )$, and thus $x$ is a strong solution of~\eqref{eq:nACP}.
For $t\geq 0$
we denote by
$y(t):=v(x(t))$ the element `$v$' in the definition of $\Dom(A_\phi)$ corresponding to the element $x(t)$.
Equation~\eqref{eq:nACPinSGproof} and the structure of $A_\phi$ imply that
\begin{subequations}
\eqn{
\label{eq:nSysinSGproof_x}
  \dot x(t) &= Ax(t) + B\phi(-y(t))\\
\label{eq:nSysinSGproof_y}
y(t) &= C\& D \pmatsmall{x(t)\\ \phi(-y(t))}
}
\end{subequations}
for a.e. $t\geq0$.
 Moreover, the second equation also holds
for every $t\geq 0$.

We will now
 show that $y$ is right-continuous and that $y,\phi(-y)\in \Lp[\infty](0,\infty;U)$.
To this end, let $\lambda>0$ be such that $\re P(\lambda)\geq c_\lambda I$ for some $c_\gl>0$.
By~\citel{Miy92book}{Cor.~3.7} the function $x:\zinf\to X$ is right-differentiable, its right-derivative $z:=D^+ x$ is right-continuous and non-increasing on $\zinf$, and $z(t)=Ax(t)+B\phi(-y(t))$ for $t\geq 0$.
In particular, $z\in\Lp[\infty](0,\infty;X)$.
For $t\geq 0$ we therefore have
$\lambda x(t)-z(t)=(\lambda-A)x(t)-B\phi(-y(t))$, and further $x(t)=R_\lambda(\gl x(t)-z(t))+R_\lambda B\phi(-y(t))$, where we have denoted $R_\lambda:=(\lambda-A)\inv$.
Substituting this into~\eqref{eq:nSysinSGproof_y} and using $R_\lambda(\gl x(t)-z(t))\in \Dom(A)$ we get for $t\geq 0$
\eq{
y(t)&= C\&D \pmat{x(t)\\\phi(-y(t))}
=  CR_\lambda(\gl x(t)-z(t)) + P(\lambda)\phi(-y(t)).
}
We have from Lemma~\ref{lem:PhiProps}(a) that there exists a globally Lipschitz continuous function $g:U\times U\to U$ such that $y(t)=g(0,CR_\lambda(\gl x(t)-z(t)))$ for $t\geq 0$. 
But since  $t\mapsto CR_\lambda(\gl x(t)-z(t))$ is right-continuous and essentially bounded,
 we deduce that also $y$ is right-continuous and $y\in \Lp[\infty](0,\infty;U)$.
Finally, since $\re P(\gl)\geq c_\gl I$,  we have that $P(\gl)$ is invertible and $\phi(-y(t))=P(\gl)\inv (y(t)-CR_\lambda(\gl x(t)-z(t)))$ for $t\geq 0$. Thus
also $t\mapsto \phi(-y(t))$ is right-continuous and $t\mapsto \phi(-y(t))\in \Lp[\infty](0,\infty;U)$.

Assume for the moment that $S$ is well-posed and denote the associated well-posed system by $\Sigma = (\T,\Phi,\Psi,\F)$. We will show that $x$ and $y$ satisfy~\eqref{eq:NLFBStateandOutput}.
The property $\phi(-y(\cdot))\in \Lploc[2](0,\infty;U)$, our knowledge that $\dot x\in \L^\infty(0,\infty;X)$,
the fact that~\eqref{eq:nSysinSGproof_x} holds for a.e. $t\geq 0$ together with~\citel{TucWei14}{Prop. 4.7} imply that 
$x(t)=\T_t x_0+\Phi_t\Pt \phi(-y)$ for all $t\geq 0$.
We will now show that $y$ satisfies~\eqref{eq:NLFBOutput}.
Denote $\psi = \phi(-y(\cdot))$.
We saw above that $y,\psi\in\Lp[\infty](0,\infty;U)$, and since $ \norm{Ax(\cdot)+ B\psi}=\norm{\dot x}\in \Lp[\infty](0,\infty)$, we have
$(x,\phi(-y))^\top \in \Lp[\infty](0,\infty;\Dom(S))$. Since $C\&D\in \Lin(\Dom(S),U)$, taking Laplace transforms on both sides of~\eqref{eq:nSysinSGproof_x} and~\eqref{eq:nSysinSGproof_y} shows that for $\gl\in\Cc_+$ we have $\hat x(\gl) = (\gl-A)\inv x_0+(\gl-A)\inv B\hat \psi(\gl)$ and
\eq{
\hat y(\gl) 
&= C\& D \pmat{\hat x(\gl)\\ \hat\psi(\gl)}
= C(\gl-A)\inv x_0+P(\gl)\hat\psi(\gl).
}
Together with~\citel{TucWei14}{Prop.~4.7} this implies that
 the Laplace transform of $y$ coincides with the Laplace transform of $\Psi_\infty x_0+\F_\infty \phi(-y)$ on some open right half plane of $\Cc$.
 The uniqueness of Laplace transforms implies that $y=\Psi_\infty x_0 + \F_\infty \phi(-y)$. 
This completes the proof that if $S$ is well-posed, then $x$ and $y$ satisfy~\eqref{eq:NLFBStateandOutput}.

To prove the existence of a generalised solution, assume that $x_0\in \overline{\Dom(A_\phi)}$. 
Then there exists a sequence $(x^k_0)_{n\in\N}\subset \Dom(A_\phi)$ such that $\norm{x_0^k-x_0}\to 0$ as $k\to\infty$. If we define $x(t)=\T^\phi_t (x_0)$ and $x^k(t)=\T^\phi_t (x_0^k)$ for $t\geq 0$ and for $k\in\N$,  then the contractivity of the semigroup $\T^\phi$ implies that 
\eq{
\sup_{t\geq 0} \norm{x^k(t)-x(t)}
\leq \norm{x_0^k-x_0}\to 0
}
as $k\to\infty$. 
This implies that $x$ is a generalised solution of~\eqref{eq:nACP}.
\end{proof}

\begin{proof}[Proof of Lemma~\textup{\ref{lem:DomAphiDense}}]
We first note that by~\citel{Log20}{Thm.~4.4} 
we have  $\re P(\gl)\geq c_\gl I$ for some $c_\gl>0$ 
for one $\gl>0$ if and only if such a $c_\gl>0$ exists for all $\gl>0$.
Because of this, Lemma~\ref{lem:PhiProps}(a) implies that for every $\gl>0$ and $r\in U$ the equation $v=r+P(\gl)\phi(-v)$ has a unique solution $v\in U$ which is determined by $v=g_\gl(r)$, where $g_\gl:U\to U$ is a globally Lipschitz continuous function.
Using the properties of $S$ and similar computations as in the proof of Theorem~\ref{thm:nonlinACPsol} it is straightforward to verify that for any $\gl>0$
\eq{
\Dom(A_\phi) = \bigl\{ x_0+(\gl-A)\inv B\phi(-v) \, \bigm| \, x_0\in \Dom(A), \ 
v=g_\gl(Cx_0)
\bigr\}.
}
For brevity, we denote $R_\gl=(\gl-A)\inv$ for $\gl>0$ throughout the proof.

To prove part (a) it suffices to note that
if we let $\gl>0$ and define $v_0=g_\gl(0)\in U$, then $v_0=P(\gl)\phi(-v_0)$ and
 the above characterisation of $\Dom(A_\phi)$ implies that $\{x_0+R_\gl B\phi(-v_0)\,|\ x_0\in \ker(C)\}\subset \Dom(A_\phi)\subset X$. But since $\ker(C)$ is dense in $X$ and $R_\gl B\phi(-v_0)\in X$, the set
 $\Dom(A_\phi)$ must be dense in $X$.

In part (b), let $M>0$ be such that $\norm{\phi(u)}\leq M$ for all $u\in U$. Let $x\in X$ and $\eps>0$ be arbitrary.
Since $\norm{R_\gl B}\to 0$ as $\gl\to \infty$ by assumption, we can choose $\gl>0 $ such that $\norm{R_\gl B}<\eps/(2M)$. If we choose $x_0\in \Dom(A)$ such that $\norm{x_0-x}<\eps/2$ and $v=g_\gl(Cx_0)\in U$, then $x_0+ R_\gl B\phi(-v)\in \Dom(A_\phi)$ and 
\eq{
\norm{x-x_0-R_\gl B\phi(-v)}<\frac{\eps}{2} + \norm{R_\gl B}\norm{\phi(-v)}<\eps.
}
Since $x\in X$ and $\eps>0$ were arbitrary, we have that
$\Dom(A_\phi) $ is dense in $X$.

To prove part (c), assume that $P(\gl)\to 0$ as $\gl \to \infty$ and that $\phi$ is globally Lipschitz continuous with Lipschitz constant $L_\phi>0$.
For any $u\in U$ and   $\gl>0$ we have that $(R_\gl Bu,u)^\top\in \Dom(S)$ and the impedance passivity of $S$ implies 
\eq{
\re \iprod{P(\gl)u}{u} 
=\re\Bigl\langle C\& D \pmat{R_\gl Bu\\u},u\Bigr\rangle
\ge \re \iprod{AR_\gl Bu + Bu}{R_\gl Bu}
= \gl \norm{R_\gl Bu}^2.
}
This implies that $\norm{R_\gl B}^2\leq \gl\inv \norm{P(\gl)}\to 0$ as $\gl\to\infty$.
Let
 $x\in X$ and $\eps>0$ be arbitrary and
let $(\gl_n)_{n\in\N}\subset (0,\infty)$ be such that $\gl_n\to \infty$ as $n\to\infty$.
Choose $x_0\in \Dom(A)$ such that $\norm{x-x_0}<\eps/2$ and define $v_n=g_{\gl_n}(Cx_0)$ for $n\in\N$. Then $v_n= Cx_0 + P(\gl_n)\phi(-v_n)$, $n\in\N$. 
Since $P(\gl_n)\to 0$ as $n\to \infty$, there exists $n_0\in \N$ such that $\norm{P(\gl_n)}<1/(2L_\phi)$ for all $n\geq n_0$. For all $n\geq n_0$ we therefore have
\eq{
\norm{v_n} \le \norm{Cx_0} + \norm{P(\gl_n)}\norm{\phi(-v_n)}
\le \norm{Cx_0} +  \frac{1}{2L_\phi}L_\phi \norm{v_n},
}
which implies that $\norm{v_n}\le 2 \norm{Cx_0}$, and further that $\norm{\phi(-v_n)}\le 2L_\phi \norm{Cx_0}$, for all $n\ge n_0$. 
Since $R_{\gl_n}B\to 0$ as $n\to\infty$ we can choose $m\ge n_0$ such that $\norm{R_{\gl_m}B}<\eps/(4L_\phi \norm{Cx_0})$.
Then
$x_0 + R_{\gl_m}B\phi(-v_m)\in \Dom(A_\phi)$ and
\eq{
\norm{x-x_0 -R_{\gl_m}B\phi(-v_m)}<\frac{\eps}{2} + \norm{R_{\gl_m}B} \norm{\phi(-v_m)}< \frac{\eps}{2} + \frac{\eps}{2} = \eps.
}
Since $x\in X$ and $\eps>0$ were arbitrary, we have that $\Dom(A_\phi) $ is dense in $X$.

To prove part (d), assume that $B\in \Lin(U,X)$ and $C\& D=[C,\ 0]$ with $C\in \Lin(X,U)$. If $x\in \Dom(A)$ and if we choose $v=Cx$, then $Ax+B\phi(-v)\in X$ and $v=C\&D \pmatsmall{x\\ \phi(-v)}=Cx$. Thus $x\in \Dom(A_\phi)$. Since $x\in \Dom(A)$ was arbitrary, we have $\Dom(A)\subset \Dom(A_\phi)\subset X$, and this implies that $\Dom(A_\phi)$ is dense in $X$.

To prove part (f), we will first show that the
maps $x_0\mapsto x(t)$, $t\geq 0$, define a (nonlinear) semigroup $\T^0$ of 
contractions so that $\T_t^0(x_0)=x(t)$, $t\ge 0$.
Let $x_1,x_2\in C(\zinf;X)$ and $y_1,y_2\in \Lploc[2](0,\infty;U)$ be solutions of~\eqref{eq:NLFBStateandOutput} corresponding to two initial conditions $x_{10}\in X$ and $x_{20}\in X$. Then
\eq{
x_2(t)-x_1(t) &= \mathbb{T}_t (x_{10}-x_{20}) +  \Phi_t \Pt[t](\phi(- y_2)-\phi(-y_1)), \qquad t\geq 0,\\
 y_2-y_1 &= \Psi_\infty (x_{20}-x_{10})  + \mathbb{F}_\infty (\phi(- y_2)-\phi(-y_1)).
}
Denoting $\phi_k=\phi(-y_k(\cdot))$, $k=1,2$,
 the impedance passivity of
$\Sigma$ and the monotonicity of $\phi$ imply that for all $t\ge 0$
\eq{
\norm{x_2(t)-x_1(t)}^2 - \norm{x_{20}-x_{10}}^2
\leq 2\re \int_0^t \iprod{\phi_2(s)-\phi_1(s)}{y_2(s)-y_1(s)}\d s\leq 0.
}
Thus 
$\norm{x_2(t)-x_1(t)}^2 \le \norm{x_{20}-x_{10}}^2$ for all $t\geq 0$.
The properties of $\Sigma$ therefore imply that the 
maps $x_0\mapsto x(t)$, $t\geq 0$, indeed define a semigroup $\T^0$ of contractions on $X$ in the sense of~\citel{Miy92book}{Def.~3.1}.
We have from the proof of Theorem~\ref{thm:nonlinACPsol} and~\cref{rem:nACPremAphiSingleValued} that $A_\phi$ is the infinitesimal generator~\citel{Miy92book}{Def.~3.2} of a semigroup $\T^\phi$ of contractions on $\overline{\Dom(A_\phi)}$, and that $\T^0_t(x_0)=\T^\phi_t(x_0)$ for all  $t\geq 0$ and $x_0\in\Dom(A_\phi)$.
This further implies that the infinitesimal generator $A_0$ of $\T^0$ is a dissipative extension of $A_\phi$.
But since $X$ is a Hilbert space, $A_\phi$ is a maximal dissipative operator~\citel{Miy92book}{Def.~2.5}, and thus necessarily $A_\phi = A_0$.
Finally, we have from~\citel{Miy92book}{Thm.~4.20} that $\Dom(A_\phi)=\Dom(A_0)$ is dense in $X$.

Finally, to prove part (e) we note that
 Assumption~\ref{ass:PhiAss} in particular implies that $\re \iprod{\phi(u_2)-\phi(u_1)}{u_2-u_1}> 0$ whenever $\phi(u_1)\neq \phi(u_2)$.
Therefore the argument in the proof of Theorem~\ref{thm:nonlinACPsol} shows that $A_\phi$ is single-valued and our claim follows from part (f) and Theorem~\ref{thm:LPSolutionsMain}.
\end{proof}

\section{Global Asymptotic Stability Under Nonlinear Feedback}\label{sec:StrongStab}

In this section we study the stability of the nonlinear system
\begin{equation}
\label{eq:SysNodeSysStab}
\left[\begin{matrix}\dot{x}(t)\\y(t)\end{matrix}\right] = \pmat{A\&B\\ C\& D}\left[\begin{matrix}x(t)\\\phi(- y(t))\end{matrix}\right], 
\qquad 
t\geq 0,
\end{equation}
where $S=\pmatsmall{\AB\\ \CD}$ is an impedance passive system node on $(U,X,U)$ and $\phi: U\to U$. 
Similarly as in Section~\ref{sec:WPnoinputs}, equation~\eqref{eq:SysNodeSysStab} can be reformulated as a nonlinear abstract Cauchy problem
\eqn{
\label{eq:nACPStabSec}
\dot x(t) = A_\phi(x(t)), \qquad t\geq 0,
}
with $A_\phi :\Dom(A_\phi)\subset X\to X$ defined by
\eq{
\Dom(A_\phi) &= \Bigl\{ x\in X\, \Bigm| \exists v\in U \mbox{ s.t.}
\ Ax + B\phi(-v)\in X,\
v = C\& D \pmatsmall{x\\\phi(-v)}
\Bigr\}\\
A_\phi x &= Ax + B\phi(-v(x)),
\qquad x\in \Dom(A_\phi),
}
where $v(x)$ is the element $v$ in the definition of $\Dom(A_\phi)$. 
We investigate the stability of~\eqref{eq:nACPStabSec} 
for functions $\phi:U\to U$
satisfying the following conditions.

\begin{assumption}\label{assum:phi}
The function $\phi:U\to U$ is locally Lipschitz continuous, $\phi(0) = 0$ and there exists constants 
$\alpha, \beta, \delta>0$ 
such that 
\begin{itemize}
    \item $\re\langle \phi(u),u\rangle\geq \alpha\Vert u\Vert^2$ when $\Vert u\Vert<\delta$;
    \item $\re\langle \phi(u),u\rangle\geq \beta$ when $\Vert u\Vert\geq \delta$.
\end{itemize}
\end{assumption}

\begin{example}\label{ex:scalarSatComplex_AssumStab}
If $\psi:\zinf\to \zinf$ 
satisfies Assumption \ref{assum:phi} with $U=\zinf$
 then it is easy to check that also the function $\phi: U \to U$ defined by $\phi(0)=0$ and $\phi(u)=\psi(\norm{u})\norm{u}\inv u$ for $u\neq 0$ satisfies \cref{assum:phi}. 
In particular, the saturation function 
$\phi:\mathbb{C}\to\mathbb{C}$ 
defined so that
 ${\phi}(u) = u$ whenever $\vert u\vert< 1$ and ${\phi}(u) = u/\vert u\vert$ whenever $\vert u\vert\geq 1$ satisfies Assumption~\ref{assum:phi}.
Moreover, if $\phi_1:U_1\to U_1$ and $\phi_2:U_2\to U_2$ satisfy \cref{assum:phi},
then the same is true for $\phi : U\to U$ defined by $\phi(u)=(\phi_1(u_1),\phi_2(u_2))^\top$ for $u=(u_1,u_2)^\top\in U:=U_1\times U_2$.
\end{example}

Our main result establishes the asymptotic convergence of generalised solutions of~\eqref{eq:nACPStabSec} to the origin under 
the assumption that \emph{linear} feedback of the form $u=-y+v$ achieves strong stability of the closed-loop semigroup $\T^K$. 
As shown in \cref{thm:SysNodeFeedback}, the generator $A^K$ of $\T^K$ is given by
\begin{subequations}
\label{eq:StabAK}
\eqn{
\Dom(A^K) &= \Bigl\{ x\in X\, \Bigm| \exists v\in U \mbox{ s.t.}
\ Ax - Bv\in X,\
v = C\& D \pmatsmall{x\\-v}
\Bigr\}\\
A^K x &= Ax - Bv(x),
\qquad x\in \Dom(A^K),
}
\end{subequations}
where $v(x)$ is the element $v$ in the definition of $\Dom(A^K)$. 
Conditions for the strong and exponential stability based on the observability properties of the original system node $S$ have been presented, for instance, in~\cite{Sle74,Ben78a,AmmTuc01,CurWei06,CurWei19,ChiPau23}. 
We recall that 
definitions for the solutions and output of~\eqref{eq:nACPStabSec} were introduced in
  Definition~\textup{\ref{def:SysNodeNoInputSolutions}} and conditions for
 $\Dom(A_\phi)$ to be dense and
 $A_\phi$ to be single-valued 
 were presented in
\cref{lem:DomAphiDense}
and
 \cref{rem:nACPremAphiSingleValued}, respectively.

\begin{theorem}
\label{thm:StabSecond}
Suppose that $S$ is an impedance passive system node on  $(U,X,U)$
and that its transfer function satisfies $\re P(\gl)\ge c_\gl I$ for some $\gl,c_\gl>0$. Let 
$\phi:U\to U$ be a continuous monotone function such that Assumption~\textup{\ref{assum:phi}} holds and $A_\phi$ is single-valued.
Moreover, assume that the semigroup generated by $A^K$ in~\eqref{eq:StabAK} 
is strongly stable.
Then every generalised solution $x$ of~\eqref{eq:nACPStabSec} on $\zinf$ satisfies
$\norm{x(t)}\to 0$ as $t\to\infty$.
If $x$ is a strong solution of~\eqref{eq:nACPStabSec} on $\zinf$, then the corresponding strong output $y$ satisfies
 $y,\phi(-y)\in \Lp[2](0,\infty;U)\cap \Lp[\infty](0,\infty;U)$.

 If $\Dom(A_\phi)$ is dense in $X$, then 
the origin is a globally asymptotically stable equilibrium point of~\eqref{eq:nACPStabSec}, 
that is,
 for every initial state $x_0\in X$ 
the generalised solution $x$ on $\zinf$ of~\eqref{eq:nACPStabSec} satisfying $x(0)=x_0$ exists and satisfies $\norm{x(t)}\to 0$ as $t\to\infty$.
\end{theorem}

\begin{proof} 
Let $x\in \Hloc[1](0,\infty;X)$ be a strong solution of~\eqref{eq:nACPStabSec} and denote its corresponding output by $y$. 
We have by definition that $x(0)\in \Dom(A_\phi)$, and 
thus \cref{thm:nonlinACPsol} and \cref{rem:nACPremAphiSingleValued} imply that $y$ satisfies $y,\phi(-y)\in \Lp[\infty](0,\infty;U)$.
We will now show that $y,\phi(-y)\in \Lp[2](0,\infty;U)$.
Since $x$ and $y$ are a strong solution and a strong output, respectively, of~\eqref{eq:nACPStabSec}, 
the impedance passivity of $S$
implies that for a.e. $s\geq 0$ we have
\eq{
\frac{1}{2}\frac{\d }{\d s}\norm{x(s)}_X^2
&=\re \iprod{\dot x(s)}{x(s)}
=\re \Iprod{A\& B \pmat{x(s)\\\phi(-y(s))}}{x(s)}\\
&\leq \re \Iprod{C\& D \pmat{x(s)\\\phi(-y(s))}}{\phi(-y(s))}
=\re \iprod{y(s)}{\phi(-y(s))}.
}
Integrating the above estimate with respect to $s$ from $0$ to $t$ 
implies that
\begin{equation}
2 \int_0^t\re\langle \phi(-y(s)),-y(s)\rangle\d s\le  \|x(0)\|^2-\|x(t)\|^2 \leq  \|x(0)\|^2 , \qquad t\ge 0.
\label{Inequality_Sigma_Phi_Passive}
\end{equation}
Let $\alpha,\beta,\delta>0$ be as in 
 Assumption~\ref{assum:phi}.
 For a fixed representative of the equivalence class $y\in\Lp[\infty](0,\infty;U)$, define $g=\norm{y(\cdot)}$ and define $\Omega_2=g\inv(\zabl{0}{\gd})$ and $\Omega_1=\zinf\setminus\Omega_2$.
Then $\Omega_1$ and $\Omega_2$ are measurable and disjoint subsets of $\zinf$ and
\eq{
\begin{cases}
\|y(t)\|\geq \delta, \qquad \mbox{for a.e. } t\in \Omega_1\\
\|y(t)\|< \delta, \qquad \mbox{for a.e. } t\in \Omega_2.
\end{cases}
}
 Assumption~\ref{assum:phi} and~\eqref{Inequality_Sigma_Phi_Passive} imply that
\begin{align*}
    \beta\mu(\Omega_1)\leq \int_{\Omega_1} \re\langle \phi(-y(s)),-y(s)\rangle\d s 
    \le \frac{1}{2}\norm{x(0)}^2.
\end{align*}
Thus $\mu(\Omega_1)\leq \frac{1}{2\beta}\norm{x(0)}^2$ and, in particular, 
 $\Omega_1$ has finite measure. On the other hand, Assumption~\ref{assum:phi} and~\eqref{Inequality_Sigma_Phi_Passive} also imply that
\begin{equation}
    \alpha\int_{\Omega_2}\|y(s)\|^2\d s\leq \int_{\Omega_2}\re\langle \phi(-y(s)),-y(s)\rangle\d s 
    \le \frac{1}{2}\norm{x(0)}^2
    <\infty.
\label{Estimate_y_Omega_1}
\end{equation}
 Since $y\in\Lp[\infty](0,\infty;U)$, we have
$\|y\|_{\L^2(\Omega_1;U)}<\infty$. Combined with
\eqref{Estimate_y_Omega_1} this shows that $y\in\L^2(0,\infty;U)$ as claimed.
Since $\phi(-y)\in\Lp[\infty](0,\infty;U)$ and $\mu(\Omega_1)<\infty$, we have $\|\phi(-y)\|_{\L^2(\Omega_1;U)}<\infty$.
On the other hand, 
since $\phi$ is locally Lipschitz continuous and $\norm{y(t)}<\gd$ for a.e. $t\in \Omega_2$, there exists $L_\delta>0$ such that $\|\phi(-y(t))\|\leq L_\delta\|y(t)\|$ for a.e. $t\in\Omega_2$.
Hence
\begin{align*}
    \int_{\Omega_2} \|\phi(-y(s))\|^2\d s \leq L_\delta^2\int_{\Omega_2} \|y(s)\|^2\d s<\infty.
\end{align*}
Combining the above properties shows that $\phi(-y)\in\Lp[2](0,\infty;U)$ as claimed.

Let $K=I\in \Lin(U)$ and let $S^K$ and $M$ be as in
 \cref{thm:SysNodeFeedback}. 
Denote the well-posed linear system associated to $S^K$ by 
$\Sigma^K=(\T^K,\Phi^K,\Psi^K,\F^K)$.
Since $y,\phi(-y)\in \Lp[2](0,\infty;U)$
and since $x$ and $y$ satisfy~\eqref{eq:SysNodeSysStab} for a.e. $t\geq 0$, 
the identity $S=S^KM$
implies that 
\eq{
\pmat{\dot x(t)\\ y(t)}
&= S \pmat{x(t)\\ \phi(-y(t))}
=S^K  \pmat{x(t)\\ y(t)+\phi(-y(t))}
}
for a.e. $t\ge 0$.
Since $y+\phi(-y)\in \Lp[2](0,\infty;U)$, we have from~\citel{TucWei14}{Prop.~4.7} that $x(t) = \mathbb{T}^K_t x(0) + \Phi_t^K \Pt(y +\phi(-y))$ for $t\ge 0$.
Theorem~\ref{thm:SysNodeFeedback} shows that 
$A^K$ in~\eqref{eq:StabAK} is the generator of $\T^K$ and that $\norm{\Phi_t^K}\leq 1$ for all $t\geq 0$.
Thus $\T^K$ is strongly stable and $\T^K_tx(0)\to 0$ as $t\to\infty$. 
Since
 $y+\phi(-y)\in\L^2(0,\infty;U)$,  Lemma~\ref{Lemma_Conv_InputMap} 
finally implies that $\norm{\Phi_t^K \Pt(y +\phi(-y))}\to 0$ and
 $\norm{x(t)}\to 0$ as $t\to\infty$.
This completes the proof that all strong solutions of~\eqref{eq:nACPStabSec} satisfy $\norm{x(t)}\to 0$ as $t\to\infty$, and that their corresponding outputs satisfy $y,\phi(-y)\in \Lp[2](0,\infty;U)$.

Now let $x\in C(\zinf;X)$ be a generalised solution of~\eqref{eq:nACPStabSec} and 
let $\eps>0$ be arbitrary. By definition there exists a strong solution $x_\eps\in \Hloc[1](0,\infty;X)$ such that $\norm{x(0)-x_\eps(0)}<\eps/2$.
The first part of the proof shows that 
 $\norm{x_\eps(t)}\to 0$ as $t\to\infty$. Because of this, we can choose $t_0>0$ so that $\norm{x_\eps(t)}<\eps/2$ for all $t\geq t_0$.
We have from Theorem~\ref{thm:nonlinACPsol} and \cref{rem:nACPremAphiSingleValued} that $x$ and $x_\eps$ satisfy the estimate~\eqref{eq:NLFBContractivity}. Thus for all $t\geq t_0$ we have
\eq{
\norm{x(t)}
\leq \norm{x(t)-x_\eps(t)}+\norm{x_\eps(t)}
\leq \norm{x(0)-x_\eps(0)}+\norm{x_\eps(t)}<\frac{\eps}{2} + \frac{\eps}{2} = \eps.
}
Since $\eps>0$ was arbitrary, we have that $\norm{x(t)}\to 0$ as $t\to\infty$.

Finally, if $\Dom(A_\phi)$ is dense in $X$, then 
\cref{thm:nonlinACPsol} and \cref{rem:nACPremAphiSingleValued} imply that for every $x_0\in X$ there exists a generalised solution $x$ of~\eqref{eq:nACPStabSec} satisfying $x(0)=x_0$, and the previous part of the proof shows that  $\norm{x(t)}\to 0$ as $t\to \infty$.
\end{proof}

We close this section by presenting a variant of \cref{thm:StabSecond} 
in the case where the system node $S$ is well-posed.
It should be noted that  $\phi$  is not assumed to be monotone.

\begin{proposition}
\label{prp:StabWellPosed}
Suppose that $\Sigma =(\T,\Phi,\Psi,\F)$ is an impedance passive well-posed linear system and  that $\phi:U\to U$ satisfies Assumption~\textup{\ref{assum:phi}}.
Moreover, assume that the well-posed linear system $\Sigma^K=(\T^K,\Phi^K,\Psi^K,\F^K)$ obtained from $\Sigma$ by negative feedback $u=-y+v$ is such that the semigroup $\T^K$ is strongly stable.
Let 
 $x_0\in X$ be such that there exist $x\in C(\zinf;X)$ and $y\in \Lp[\infty](0,\infty;U)$ such that 
\begin{subequations}
\label{eq:StabThmStateOutput}
\eqn{
x(t)&= \T_tx_0 + \Phi_t\Pt \phi(-y), \qquad t\geq 0,\\
y&= \Psi_\infty x_0 + \F_\infty \phi(-y).
\label{eq:StabThmStateOutput_y}
}
\end{subequations}
Then $\norm{x(t)}\to 0$ as $t\to\infty$, and $y\in \Lp[2](0,\infty;U)$ and $\phi(-y)\in \Lp[2](0,\infty;U)$.

In particular, if $x$ and $y$
 with the above properties
 exist for every $x_0\in X$, 
then $\norm{x(t)}\to 0$ as $t\to\infty$ and $y,\phi(-y)\in \Lp[2](0,\infty;U)$ for every $x_0\in X$.
\end{proposition}

\begin{proof} 
Let $x_0\in X$ and $y\in\Lp[\infty](0,\infty;U)$
be such 
 that~\eqref{eq:StabThmStateOutput} hold. 
Since $\phi$ is locally Lipschitz continuous, we have $\phi(-y)\in \Lp[\infty](0,\infty;U)$.
Since $\Sigma$ is impedance passive, 
Definition~\ref{def:Imp_Passive} implies that $x$ and $y$ satisfy~\eqref{Inequality_Sigma_Phi_Passive}.
The same arguments as in the proof of \cref{thm:StabSecond} show that $y,\phi(-y)\in \Lp[2](0,\infty;U)$.
We have from~\citel{Sta05book}{Cor.~6.1} 
that $-I$ is an admissible output feedback operator for $\Sigma$, and that the resulting closed-loop system $\Sigma^K = (\T^K,\Phi^K,\Psi^K,\F^K)$ is such that $\norm{\Phi_t^K}\leq M$ for some $M>0$ and for all $t\geq 0$.
Therefore~\citel{TucWei14}{Prop.~5.15} implies that
for all $t\geq 0$
\eq{
\pmat{x(t)\\ \Pt y}
= \pmat{\T_t&\Phi_t\\ \Psi_t&\F_t} \pmat{x_0\\\Pt \phi(-y)}
= \pmat{\T_t^K&\Phi_t^K\\ \Psi_t^K&\F_t^K} \left( \pmat{x_0\\\Pt \phi(-y)} + \pmat{0\\\Pt y}\right)
}
and, in particular, $x(t) = \mathbb{T}^K_t x_0 + \Phi_t^K \Pt(y +\phi(-y))$ for $t\ge 0$.
The semigroup $\T^K$ is strongly stable by assumption and thus $\T^K_tx_0\to 0$ as $t\to\infty$. 
Since 
 $y+\phi(-y)\in\L^2(0,\infty;U)$, Lemma~\ref{Lemma_Conv_InputMap} implies that $\Phi_t^K \Pt(y +\phi(-y))\to 0$ as $t\to\infty$, and thus 
 $\norm{x(t)}\to 0$ as $t\to\infty$.
\end{proof}

\section{Boundary Control Systems with Nonlinear Feedback}
\label{sec:BCS}

In this section we investigate abstract boundary control systems~\cite{Sal87a,MalSta06,JacobZwart} of the form
\begin{subequations}
\label{eq:BCSlin}
\eqn{
\dot{x}(t)&= L x(t)  
\\
G x(t) &=  u(t)\\
y(t) &= K x(t)
}
\end{subequations}
for $t\geq 0$.
If $X$, $U$, and $Y$ are Hilbert spaces, then
the triple $(G,L,K)$ is called a \emph{boundary node} on the spaces $(U,X,Y)$ if
$L: \Dom(L)\subset X\to X$, $G\in \Lin(\Dom(L),U)$,  $K\in \Lin(\Dom(L),Y)$ and the following hold.
\begin{itemize}
\item The restriction $A:=L\vert_{\ker(G)}$ with domain $\Dom(A)=\ker(G)$ generates a strongly continuous semigroup $\T$ 
 on $X$.
\item The operator $G\in \Lin(\Dom(L),U)$ has a bounded right inverse, i.e., there exists $G^r\in\Lin(U,\Dom(L))$ such that $GG^r = I$.
\end{itemize}
This boundary node is \emph{impedance passive} if $Y=U$ and if
\eq{
\re \iprod{L x}{x}_X \leq \re \iprod{G x}{K x}_U , \qquad x\in \Dom(L).
}
Moreover, the \emph{transfer function} $P:\rho(A)\to \Lin(U,Y)$ of $(G,L,K)$ is defined so that $P(\gl)u = K x$, where $x\in \Dom(L)$ is the unique solution of the equations $(\gl-L)x=0$ and $Gx = u$.

It has been shown in~\cite{MalSta06} that every boundary node $(G,L,K)$ gives rise to a system node $S$ on $(U,X,Y)$. More precisely,
\citel{MalSta06}{Thm.~2.3} shows that
the operator 
\eq{
S= \pmat{L\\ K}\pmat{I\\ G}\inv, \qquad \Dom(S)= \ran \left( \pmat{I\\ G} \right)
}
is a system node on $(U,X,Y)$.
We note that the definition of $S$ in particular implies that $(x,u)^\top\in \Dom(S)$ if and only if $x\in \Dom(L)$ and $u=Gx$.
If $(G,L,K)$ is impedance passive, then the associated system node $S$ is impedance passive as well, since for $(x,u)^\top \in \Dom(S)$ we then have
\eq{
\re \iprod{L \pmat{I\\ G}\inv \pmat{x\\u}}{x} 
= 
\re \iprod{L x}{x} 
\leq \re \iprod{K x}{ G x} 
=\re \iprod{K \pmat{I\\ G}\inv \pmat{x\\u}}{u} .
}
In addition, the transfer function of $(G,L,K)$ coincides with the transfer function of the associated system node.

Because of this connection with system nodes, we can call a boundary node $(G,L,K)$ \emph{(externally) well-posed} if its associated system node $S$ is well-posed, that is, if $S$ is a system node of a well-posed linear system $\Sigma$.
In this case we can also say $\Sigma$ is the well-posed system associated with $(G,L,K)$.
If $x_0\in \Dom(L)$ and $u\in C^2(\zinf;U)$ are such that $G x_0=u(0)$, then by~\citel{MalSta06}{Lem.~2.6} there exists $x\in C^1(\zinf;X)\cap C(\zinf;\Dom(L))$ and $y\in C(\zinf;Y)$ such that $x(0)=x_0$ and such that~\eqref{eq:BCSlin} hold for all $t\geq 0$.
The well-posedness of $(G,L,K)$ is equivalent to the property that there exists $\tau,M_\tau>0$ such that all such solutions of~\eqref{eq:BCSlin} satisfy~\citel{JacobZwart}{Ch.~13}
\eq{
\norm{x(\tau)}_X^2 + \int_0^\tau\norm{y(t)}_Y^2\d t \leq M_\tau \left( \norm{x_0}_X^2 + \int_0^\tau \norm{u(t)}_Y^2 \d t \right).
}

\begin{lemma}
    \label{lem:BCSTFproperty}
    Let $(G,L,K)$ be an impedance passive boundary node on the spaces $(U,X,U)$ with $U\neq \{0\}$ and let $P$ be the transfer function of $(G,L,K)$. For every $\gl\in\Cc_+$ there exists $c_\gl>0$ such that $\re P(\gl)\ge c_\gl I$. 
\end{lemma}

\begin{proof}
Let $\gl>0$ and $u\in U$. Then $P(\gl)u = Kx$, where $x\in \Dom(L)$ satisfies $Lx=\gl x$ and $Gx=u$.
The identity $x=\gl\inv Lx $ implies that $2 \iprod{Lx}{x}=\gl\norm{x}^2+\gl\inv\norm{Lx}^2 \ge \min\{\gl,\gl\inv\}  \norm{x}_{\Dom(L)}^2$.
Since $U\neq \{0\}$ and since $G$ is surjective, we have $G\neq 0$ and
 $\norm{u}^2 = \norm{Gx}^2\leq \norm{G}_{\Lin(\Dom(L),U)}\norm{x}_{\Dom(L)}^2$.
Combining these estimates and using the impedance passivity of
$(G,L,K)$ shows that
\eq{
\re \iprod{P(\gl)u}{u}
= \re\iprod{Kx}{Gx}
\ge \re \iprod{Lx}{x}
\ge \frac{\min\{\gl,\gl\inv\}}{2\norm{G}_{\Lin(\Dom(L),U)}} \norm{u}^2.
}
Since $u\in U$ was arbitrary,  there exists $c_\gl>0$ such that $\re P(\gl)\ge c_\gl I$. Since $P$ is a positive real function~\citel{Log20}{Thm.~4.4} implies the claim for all $\gl\in\Cc_+$.
\end{proof}

We are interested in the well-posedness and stability of the equation
\begin{subequations}
\label{eq:BCSnonlin}
\eqn{
\dot{x}(t)&= L x(t) 
\\
G x(t) &=  \phi(u(t)-y(t))\\
y(t) &= K x(t)
}
\end{subequations}
for $t\geq 0$,
where $(G,L,K)$ is an impedance passive and well-posed boundary node on $(U,X,U)$ and $\phi: U\to U$.
The following definition of classical and generalised solutions of~\eqref{eq:BCSnonlin} agree with the concepts of classical and generalised solutions of the equation~\eqref{eq:SysNodeSysMain} when the system node in~\eqref{eq:Op_S} is associated with $(G,L,K)$.

\begin{definition}
\label{def:BCSStates}
Let $(G,L,K)$ be a well-posed boundary node on $(U,X,U)$.
A triple $(x,u,y)$ is called a \emph{classical solution} of \eqref{eq:BCSnonlin} on $\zinf$ if
\begin{itemize}
    \item $x\in C^1(\zinf;X)$,
$u\in C(\zinf;U)$, and
 $y\in C(\zinf;U)$;
    \item $x(t)\in \Dom(L)$ for all $t\geq 0$ and~\eqref{eq:BCSnonlin} hold for every $t\geq 0$.
\end{itemize}
A triple $(x,u,y)$ is called a \emph{generalised solution} of~\eqref{eq:BCSnonlin} on $\zinf$ if 
\begin{itemize}
    \item $x\in C(\zinf;X)$, 
    $u\in\Lploc[2](0,\infty;U)$, and $y\in\Lploc[2](0,\infty;U)$;
    \item there exists a sequence $(x^k,u^k,y^k)$ of classical solutions of  \eqref{eq:BCSnonlin} on $\zinf$ such that for every $\tau>0$ we have
$(\Pt[\tau] x^k,\Pt[\tau] u^k,\Pt[\tau] y^k)^\top\to (\Pt[\tau] x,\Pt[\tau] u,\Pt[\tau] y)^\top$ as $k\to \infty$ in $C([0,\tau];X)\times \L^2(0,\tau;U) \times \L^2(0,\tau;U)$.
\end{itemize}
\end{definition}

The following result introduces conditions for the existence of classical and generalised solutions of~\eqref{eq:BCSnonlin}.

\begin{theorem}
\label{thm:BCSwp}
Suppose that $(G,L,K)$ is an impedance passive and well-posed boundary node on $(U,X,U)$ with $U\neq \{0\}$
and suppose that $\phi : U\to U$ satisfies Assumption~\textup{\ref{ass:PhiAss}} for some $\phicoeff>0$.
If $x_0\in X$ and $u\in \Lploc[2](0,\infty;U)$, then~\eqref{eq:BCSnonlin} has a 
 generalised solution $(x,u,y)$ on $\zinf$ satisfying $x(0)=x_0$. 
This solution satisfies
\eq{
x(t) &= \mathbb{T}_t x_{0} +  \Phi_t \Pt[t]\phi(u- y), \qquad t\geq 0,\\
 y &= \Psi_\infty x_{0}  + \mathbb{F}_\infty \phi(u-y),
}
where 
 $\Sigma = (\T,\Phi,\Psi,\F)$ is the well-posed system associated to $(G,L,K)$.
If $x_0\in \Dom(L)$ and $u\in \Hloc[1](0,\infty;U)$ are such that 
$G x_0= \phi(u(0)-K x_0)$, then
 this $(x,u,y)$ is a classical solution of~\eqref{eq:BCSnonlin} on $\zinf$.

If $(x_1,u_1,y_1)$ and $(x_2,u_2,y_2)$ are two generalised solutions of~\eqref{eq:BCSnonlin}, then for all $t\geq 0$ we have
\begin{subequations}
\eqn{
\label{eq:Increm_Estim_BCS}
\MoveEqLeft\norm{x_2(t)-x_1(t)}_X^2 + \phicoeff\int_0^t \norm{\phi(u_2(s)-y_2(s))-\phi(u_1(s)-y_1(s))}_{U}^2\d s\\
&\leq \norm{x_2(0)-x_1(0)}_X^2+\frac{1}{\phicoeff}\int_0^t \norm{u_2(s)-u_1(s)}_{U}^2\d s.
}
\end{subequations}
\end{theorem}

\begin{proof}
Let $S = \pmatsmall{A\& B\\ C\& D}$ be the system node associated with the boundary node $(G,L,K)$.
Definitions~\ref{def:SysNodeStates} and~\ref{def:BCSStates} and the relationship between $(G,L,K)$ and $S$ 
imply that the classical and generalised solutions of~\eqref{eq:SysNodeSysMain} coincide with the 
 classical and generalised solutions, respectively, of~\eqref{eq:BCSnonlin}. 
Our assumptions and Lemma~\ref{lem:BCSTFproperty} imply that $S$ is impedance passive and well-posed, and its transfer function $P$ satisfies $\re P(\gl)\geq c_\gl I$ for some $\gl,c_\gl>0$.
Because of this, Theorem~\ref{thm:LPSolutionsMain} implies that~\eqref{eq:SysNodeSysMain} --- and consequently also~\eqref{eq:BCSnonlin} --- has classical and generalised solutions with the properties stated in the claim. 
In particular, $(x,u)^\top \in X\times U$ satisfies  $(x,u)^\top \in \Dom(S)$ (or equivalently $Ax+Bu\in X$) if and only if $x\in \Dom(L)$ and $Gx=u$ and we have from~\citel{MalSta06}{Thm.~2.3(iv)} that $C\& D = [K,\, 0]\vert_{\Dom(S)}$.
These properties imply that
the condition $(x_0,u(0))^\top\in \Dcomp$ in Theorem~\ref{thm:LPSolutionsMain} holds if and only if $x_0\in \Dom(L)$ and $Gx_0 = \phi(u(0)-Kx_0)$.
Finally, the estimate for the difference of two generalised solutions of~\eqref{eq:BCSnonlin} follows directly from the corresponding estimate for the pairs of solutions of~\eqref{eq:SysNodeSysMain}.
\end{proof}

In the absence of the external input $u(t)$,
 the system~\eqref{eq:BCSnonlin} becomes
\begin{subequations}
\label{eq:BCSnonlinNoInput}
\eqn{
\label{eq:BCSnonlinNoInput_state}
\dot{x}(t)&= L x(t) 
\\
\label{eq:BCSnonlinNoInput_input}
G x(t) &=  \phi(-y(t))\\
y(t) &= K x(t)
\label{eq:BCSnonlinNoInput_out}
}
\end{subequations}
for $t\geq 0$, where $(G,L,K)$ is an impedance passive (but not necessarily well-posed) boundary node on the spaces $(U,X,U)$.
We can use Theorems~\ref{thm:nonlinACPsol} and~\ref{thm:StabSecond} to analyse the well-posedness and stability of this system.
We begin by defining the strong and generalised solutions of~\eqref{eq:BCSnonlinNoInput}.

\begin{definition}
\label{def:BCSNoInputSol}
Let $(G,L,K)$ be a boundary node on the spaces $(U,X,U)$.
\begin{itemize}
\item
The function $x\in \Hloc[1](0,\infty;X)$ is called a \emph{strong solution} of~\eqref{eq:BCSnonlinNoInput} if $x(t)\in \Dom(L)$ and~\eqref{eq:BCSnonlinNoInput_input} holds for all $t\geq 0$, 
 and if~\eqref{eq:BCSnonlinNoInput_state} holds for a.e. $t\geq 0$.
The corresponding \emph{output} $y:\zinf \to U$ is defined by~\eqref{eq:BCSnonlinNoInput_out} for all $t\geq 0$.
\item
The function $x\in C(\zinf;X)$ is called a \emph{generalised solution} of \eqref{eq:BCSnonlinNoInput} on $\zinf$ if 
there exists a sequence
 $(x^k)_{k\in\N}$ 
of strong solutions of~\eqref{eq:BCSnonlinNoInput} on $\zinf$ such that $\norm{\Pt[\tau] x^k-\Pt[\tau] x}_{C([0,\tau];X)}\to 0 $ as $k\to\infty$ for all $\tau>0$.
\end{itemize}
\end{definition}

\begin{theorem}
\label{thm:BCSWPNoInput}
Suppose that $(G,L,K) $ is an impedance passive boundary node
on $(U,X,U)$ with $U\neq \{0\}$,
 that $\ker(G)\cap \ker(K)$ is dense in $X$ 
and  that $\phi:U\to U$  is a continuous monotone function.
Then for every initial state $x_0\in X$ the system~\eqref{eq:BCSnonlinNoInput} has a well-defined generalised solution $x$ on $\zinf$ satisfying $x(0)=x_0$. If $x_0\in \Dom(L) $ and $Gx_0=\phi(-Kx_0)$, then this $x$ is a strong solution of~\eqref{eq:BCSnonlinNoInput} on $\zinf$. The corresponding output $y$
 is right-continuous and $y,\phi(-y)\in \Lp[\infty](0,\infty;U)$.
If $x_1$ and $x_2$ are two generalised solutions of~\eqref{eq:BCSnonlinNoInput} on $\zinf$, then
\eq{
\norm{x_2(t)-x_1(t)}_X \leq \norm{x_2(0)-x_1(0)}_X, \qquad t\geq 0.
}

If $(G,L,K)$ is well-posed and its associated well-posed system is $\Sigma = (\T,\Phi,\Psi,\F)$, then for every $x_0\in \Dom(L) $ satisfying $Gx_0=\phi(-Kx_0)$ the strong solution $x$ and the corresponding output $y$ satisfy
\eq{
x(t) &= \mathbb{T}_t x_{0} +  \Phi_t \Pt[t]\phi(- y), \qquad t\geq 0,\\
 y &= \Psi_\infty x_{0}  + \mathbb{F}_\infty \phi(-y).
}
\end{theorem}

\begin{proof}
Let $S = \pmatsmall{A\& B\\ C\& D}$ be the system node associated with the boundary node $(G,L,K)$.
Definitions~\ref{def:SysNodeNoInputSolutions} and~\ref{def:BCSNoInputSol} and the relationship between $(G,L,K)$ and $S$ 
imply that the strong and generalised solutions of~\eqref{eq:nACP} coincide with the 
 strong and generalised solutions, respectively, of~\eqref{eq:BCSnonlinNoInput}. 
Our assumptions and Lemma~\ref{lem:BCSTFproperty} imply that $S$ is impedance passive and that its transfer function $P$ satisfies $\re P(\gl)\geq c_\gl I$ for some $\gl,c_\gl>0$.
The relationship between $(G,L,K)$ and $S$ implies that
 $(x,u)^\top \in X\times U$ satisfies  $(x,u)^\top \in \Dom(S)$ (or equivalently $Ax+Bu\in X$) if and only if $x\in \Dom(L)$ and $Gx=u$, and we have  $C\& D = [K,\, 0]\vert_{\Dom(S)}$ by~\citel{MalSta06}{Thm.~2.3(iv)}.
These properties imply that the operator $A_\phi$ in Theorem~\ref{thm:nonlinACPsol} has the form
\eq{
A_\phi x = Lx, \qquad x\in \Dom(A_\phi) = \setm{x\in \Dom(L)}{Gx=\phi(-Kx)},
}
and $A_\phi$ is in particular single-valued.
The output operator of $S$ satisfies $\ker(C)=\ker(G)\cap \ker(K)$, and therefore Lemma~\ref{lem:DomAphiDense}(a) and our assumption imply that
$\Dom(A_\phi)$ is dense in $X$.
Because of this, Theorem~\ref{thm:nonlinACPsol} and Remark~\ref{rem:nACPremAphiSingleValued} imply that~\eqref{eq:nACP} --- and therefore also~\eqref{eq:BCSnonlinNoInput} --- has strong and generalised solutions with exact the properties stated in the claim, and a generalised solution on $\zinf$ exists for every $x_0\in X$.
Moreover, the estimate for the pairs of generalised solutions follows directly from the corresponding estimate for the solutions of~\eqref{eq:nACP} in \cref{thm:nonlinACPsol}.
\end{proof}

Finally, the following theorem introduces conditions for the global asymptotic stability of the origin for the system~\eqref{eq:BCSnonlinNoInput}.

\begin{theorem}
\label{thm:BCSStabMain}
Suppose that $(G,L,K) $ is an impedance passive boundary node
on $(U,X,U)$ with $U\neq 0$,
 that $\ker(G)\cap \ker(K)$ is dense in $X$,
and that the operator $L\vert_{\ker(G+K)}$ generates a strongly stable semigroup.
Finally, assume
that $\phi:U\to U$  is a locally Lipschitz continuous monotone function, $\phi(0)=0$, and there exist 
$\alpha, \beta, \delta>0$ 
such that
\begin{itemize}
    \item $\re\langle \phi(u),u\rangle\geq \alpha\Vert u\Vert^2$ when $\Vert u\Vert<\delta$;
    \item $\re\langle \phi(u),u\rangle\geq \beta$ when $\Vert u\Vert\geq \delta$.
\end{itemize}
Then the origin is a globally asymptotically stable equilibrium point of~\eqref{eq:BCSnonlinNoInput}, i.e., 
the generalised solutions on $\zinf$ corresponding to all initial states $x(0)=x_0\in X$ exist and satisfy $\norm{x(t)}\to 0$ as $t\to\infty$.
\end{theorem}

\begin{proof}
The existence of generalised solutions corresponding to all initial states $x_0\in X$ follows from \cref{thm:BCSWPNoInput}.
Let $S = \pmatsmall{A\& B\\ C\& D}$ be the system node associated with the boundary node $(G,L,K)$.
The relationship between $(G,L,K)$ and $S$ implies that
 $(x,u)^\top \in X\times U$ satisfies  $(x,u)^\top \in \Dom(S)$ (or equivalently $Ax+Bu\in X$) if and only if $x\in \Dom(L)$ and $Gx=u$, and we have  $C\& D = [K,\, 0]\vert_{\Dom(S)}$ by~\citel{MalSta06}{Thm.~2.3(iv)}.
These properties imply that the operator $A^K$ in Theorem~\ref{thm:StabSecond} is exactly $L\vert_{\ker(G+K)}$.
As shown in the proof of Theorem~\ref{thm:BCSWPNoInput}, the operator $A_\phi$ in Theorems~\ref{thm:nonlinACPsol} and~\ref{thm:StabSecond} is such that $\Dom(A_\phi)$ is dense in $X$, and 
the generalised solutions of~\eqref{eq:BCSnonlinNoInput} coincide with those of~\eqref{eq:nACPStabSec}.
Therefore Theorem~\ref{thm:StabSecond} shows that generalised solutions corresponding to all initial states $x_0\in X$ satisfy $\norm{x(t)}\to 0$ as $t\to\infty$.
\end{proof}

\section{Partial Differential Equation Systems with Nonlinear Feedback}
\label{sec:Examples}

\subsection{Port-Hamiltonian Systems}\label{sec:App_pHs}
In this section we apply our results in the study of \emph{infinite-dimensional port-Hamiltonian systems}~\cite{JacobZwart,JacobMorrisZwart,Aug19}.
This class consists of partial differential equations of the form
\begin{subequations}
\label{eq:PHSmain}
\begin{align}
\frac{\partial x}{\partial t}(\zeta,t) &= P_1\frac{\partial}{\partial \zeta}\left(\mathcal{H}(\zeta)x(\zeta,t)\right) + P_0\left(\mathcal{H}(\zeta)x(\zeta,t)\right)
\label{pH}\\
\MoveEqLeft[7] \left[\begin{matrix} \WBi\\ \WBh\end{matrix}\right]\left[\begin{matrix}
\mathcal{H}(b)x(b,t)\\
\mathcal{H}(a)x(a,t) 
\end{matrix}\right]
=\left[\begin{matrix}I\\ 0\end{matrix}\right]u(t) 
\label{Input_pH}
\\
y(t) &= \WCo \left[\begin{matrix}
\mathcal{H}(b)x(b,t)\\
\mathcal{H}(a)x(a,t) 
\end{matrix}
\right]
\label{Output_pH}
\end{align}
\end{subequations}
for $t\geq 0$ and $\zeta\in [a,b]$, where $x(\zeta,t)\in \Cc^n$.
Here $P_1\in \Cc^{n\times n}$ is a Hermitian matrix, i.e., $P_1^\ast = P_1$, 
and $P_0\in \Cc^{n\times n}$ satisfies $\re P_0\le 0$.
The function $\mathcal{H}(\cdot)\in\L^\infty(a,b;\mathbb{C}^{n\times n})$ is assumed to satisfy $\mathcal{H}(\zeta)^* = \mathcal{H}(\zeta)$ and $ \mathcal{H}(\zeta)\ge cI$ for some $c>0$ and for a.e. $\zeta\in [a,b]$.
The matrices $\tilde W_{B,1}$, $\tilde W_{B,2}$, and $\tilde W_C$ determine the boundary input $u(t)\in \Cc^p$, homogeneous boundary conditions, and boundary output $y(t)\in\Cc^p$ based on the values of $\mc H(\cdot) x(\cdot,t)$ at the endpoints of the spatial interval $[a,b]$.
As shown in~\cite{JacobZwart}, the class of port-Hamiltonian systems~\eqref{eq:PHSmain} especially contains mathematical models of mechanical systems described by, for instance, the one-dimensional wave equation and the Timoshenko beam equation. 
We make the following assumptions on the parameters $\mc H$,  $\tilde W_{B,1}$, $\tilde W_{B,2}$, and $\tilde W_C$.

\begin{assumption}
\label{ass:PHSass}
Denote
\eq{
R_0 = \frac{1}{\sqrt{2}}\pmat{P_1&-P_1\\I&I}, \qquad \Xi = \pmat{0&I\\I&0}, 
\qquad 
 W_B=\pmat{\tilde W_{B,1}\\ \tilde W_{B,2}}R_0\inv,
}
$W_{B,1} = \tilde W_{B,1} R_0\inv$,
$W_{B,2} = \tilde W_{B,2} R_0\inv$,
and $W_C = \tilde W_C R_0\inv$.
Suppose that the following hold.
\begin{itemize}
\setlength{\itemsep}{1ex}
\item[\textup{(a)}]  We have $W_B\Xi W_B^\ast\geq 0$ and $\rank \left(\pmatsmall{W_B\\W_C}\right)=n+\rank (W_C)$.
\item[\textup{(b)}] 
$2\re \iprod{  W_{B,1} g}{W_C g}\ge \iprod{\Xi g}{g}$ for all $g\in \ker(W_{B,2})$.
\item[\textup{(c)}]  $P_1 \mc H(\zeta)= S(\zeta)\inv \Delta(\zeta)S(\zeta)$, $\zeta\in[a,b]$, where $S,\Delta\in C^1([a,b];\Cc^{n\times n})$ are such that $\Delta(\zeta)$ is a diagonal matrix and $S(\zeta)$ is nonsingular for all $\zeta\in[a,b]$.
\end{itemize}
\end{assumption}

As shown in~\citel{JacobZwart}{Ch.~11--13}, the system~\eqref{eq:PHSmain} can be recast as a boundary control system on the Hilbert space $X = \L^2(a,b;\mathbb{C}^n)$ equipped with the inner product $\langle f,g\rangle_X := \langle f,\mathcal{H}g\rangle_{\L^2(a,b;\mathbb{C}^n)}$ (see also~\cite{VilZwa05,PhiRei23arxiv}).
The boundary node $(G,L,K)$ on $(\Cc^p,X,\Cc^p)$ associated to the system~\eqref{eq:PHSmain} is defined by 
\begin{subequations}
\label{eq:PHSBCS}
\begin{align}
    Lx &= P_1\frac{\d}{\d\zeta}(\mathcal{H}x) + P_0(\mathcal{H}x),\\
    D(L) &= \{x\in X\,\bigm|\, \mathcal{H}x\in\H^1(a,b;\mathbb{C}^n),\  \WBh\left[\begin{smallmatrix}(\mathcal{H}x)(b)\\(\mathcal{H}x)(a)\end{smallmatrix}\right] = 0\},\label{D(L)_PHS}\\
    Gx &= \WBi\left[\begin{smallmatrix}(\mathcal{H}x)(b)\\(\mathcal{H}x)(a)\end{smallmatrix}\right], \quad
    Kx = \WCo\left[\begin{smallmatrix}(\mathcal{H}x)(b)\\(\mathcal{H}x)(a)\end{smallmatrix}\right].
\end{align}
\end{subequations}
Assumption~\ref{ass:PHSass}(a) implies that $(G,L,K)$ is a boundary node on $(\Cc^p,X,\Cc^p)$~\citel{JacobZwart}{Thm.~11.3.2 \& Rem.~11.3.3}.
We can check that Assumption~\ref{ass:PHSass}(b) implies that $(G,L,K)$ is impedance passive (cf.~\citel{PauLeg21}{Lem.~V.3}). Indeed, if $x\in \Dom(L)$, then $\mc H x\in \H^1(a,b;\Cc^n)$. If we define $g=R_0 \pmatsmall{(\mc Hx)(b)\\ (\mc Hx)(a)}$, then $W_{B,2}g=0$, $Gx = W_{B,1}g$, and $Kx=W_C g$.
We therefore have from~\citel{JacobZwart}{Lem.~7.2.1 \& (7.26)} and Assumption~\ref{ass:PHSass}(b) that
\eq{
\re \iprod{Lx}{x}_X
&\leq 
\frac{1}{2} \iprod{\Xi g}{g}_{\C^p}
\le   \re \iprod{W_{B,1} g}{W_C g}_{\C^p}
=   \re \iprod{Gx}{Kx}_{\C^p} .
}
Thus $(G,L,K)$ is impedance passive.
Finally, if conditions (a) and (c) in Assumption~\ref{ass:PHSass} hold, then~\citel{JacobZwart}{Thm.~13.2.2} implies  $(G,L,K)$ is well-posed, and in addition that the associated well-posed system $\Sigma$ is uniformly regular in the sense that the limit
 $\lim_{\gl\to\infty, \gl>0}P(\gl)$ exists.

We can now investigate the well-posedness and stability of the 
nonlinear partial differential equation
\begin{subequations}
\label{eq:PHSnonlin}
\begin{align}
\label{eq:PHSnonlin_x}
\frac{\partial x}{\partial t}(\zeta,t) &= P_1\frac{\partial}{\partial \zeta}\left(\mathcal{H}(\zeta)x(\zeta,t)\right) + P_0\left(\mathcal{H}(\zeta)x(\zeta,t)\right),
\\
\label{eq:PHSnonlin_u}
\MoveEqLeft[7] \left[\begin{matrix}\WBi\\ \WBh\end{matrix}\right]\left[\begin{matrix}
\mathcal{H}(b)x(b,t)\\
\mathcal{H}(a)x(a,t) 
\end{matrix}\right]
=\left[\begin{matrix}I\\ 0\end{matrix}\right]
\phi\left(u(t)-\WCo \left[\begin{matrix}
\mathcal{H}(b)x(b,t)\\
\mathcal{H}(a)x(a,t) 
\end{matrix}
\right]\right)
\\
\label{eq:PHSnonlin_y}
y(t) &= \WCo \left[\begin{matrix}
\mathcal{H}(b)x(b,t)\\
\mathcal{H}(a)x(a,t) 
\end{matrix}
\right].
\end{align}
\end{subequations}
The following result is a direct consequence of Theorem~\ref{thm:BCSwp}.

\begin{proposition}\label{prp:PHSwellposedness}
Suppose that Assumption~\textup{\ref{ass:PHSass}} holds and that $\phi:\Cc^p\to \Cc^p$ satisfies Assumption~\textup{\ref{ass:PhiAss}} for some $\phicoeff>0$, and let $(G,L,K)$ be as in~\eqref{eq:PHSBCS}.
If $x_0\in \Dom(L)$ and $u\in \Hloc[1](0,\infty;\Cc^p)$ are such that
\eqn{
\label{eq:PHScompatIC}
\MoveEqLeft[7] \left[\begin{matrix}\WBi\\ \WBh\end{matrix}\right]\left[\begin{matrix}
\mathcal{H}(b)x_0(b)\\
\mathcal{H}(a)x_0(a) 
\end{matrix}\right]
=\left[\begin{matrix}I\\ 0\end{matrix}\right]
\phi\left(u(0)-\WCo \left[\begin{matrix}
\mathcal{H}(b)x_0(b)\\
\mathcal{H}(a)x_0(a) 
\end{matrix}
\right]\right),
}
 then~\eqref{eq:PHSnonlin}  
has a
 \emph{classical solution} $(x,u,y)$ 
satisfying $x(\cdot,0)=x_0$
in the sense that
$x:[a,b]\times \zinf\to \Cc^n$ and $y\in C(\zinf;\Cc^p)$ are such that 
 $t\mapsto x(\cdot,t)\in C^1(\zinf;X) $,  
$x(\cdot,t)\in \Dom(L)$ for all $t\geq 0$, and~\eqref{eq:PHSnonlin} hold for every $t\geq 0$ and for almost every $\zeta\in[a,b]$.

If $x_0\in X$ and $u\in \Lploc[2](0,\infty;\Cc^p)$, then~\eqref{eq:PHSnonlin}
has a \emph{generalised solution} $(x,u,y)$ 
satisfying $x(\cdot,0)=x_0$
in the sense that
 $t\mapsto x(\cdot,t)\in C(\zinf;X) $ and
    $y\in\Lploc[2](0,\infty;U)$
    and
     there exists a sequence $(x^k,u^k,y^k)$ of classical solutions such that for every $\tau>0$ we have
$\sup_{t\in [0,\tau]}\norm{x^k(\cdot,t)-x(\cdot,t)}_{X}\to 0$,
$\norm{\Pt[\tau] u^k-\Pt[\tau]u}_{\Lp[2](0,\tau)}\to 0$,
and $\norm{\Pt[\tau] y^k-\Pt[\tau]y}_{\Lp[2](0,\tau)}\to 0$
as $k\to \infty$.
Each pair $(x_1,u_1,y_1)$ and $(x_2,u_2,y_2)$ of generalised solutions satisfy the estimate
\eq{
\MoveEqLeft\norm{x_2(\cdot,t)-x_1(\cdot,t)}_{\Lp[2](a,b)}^2 + \phicoeff\int_0^t \norm{\phi(u_2(s)-y_2(s))-\phi(u_1(s)-y_1(s))}_{\C^p}^2\d s\\
&\leq \norm{x_2(\cdot,0)-x_1(\cdot,0)}_{\Lp[2](a,b)}^2+\frac{1}{\phicoeff}\int_0^t \norm{u_2(s)-u_1(s)}_{\C^p}^2\d s,
\qquad t\geq 0.
}
\end{proposition}

\begin{proof}
As explained above, Assumption~\ref{ass:PHSass} implies that $(G,L,K)$ defined in~\eqref{eq:PHSBCS} is well-posed and impedance passive. 
The claims follow from the application of Theorem~\ref{thm:BCSwp} to $(G,L,K)$ in~\eqref{eq:PHSBCS}.
\end{proof}

In the case where $u\equiv 0$ we can use Theorems~\ref{thm:BCSWPNoInput} and~\ref{thm:BCSStabMain} to introduce conditions for the existence and stability of solutions of~\eqref{eq:PHSnonlin}. 
This result does not require condition (c) in Assumption~\ref{ass:PHSass}. The part on the existence of solutions is related to the earlier result on port-Hamiltonian systems with nonlinear feedback in~\citel{Aug19}{Thm.~4.2} and~\citel{Tro14}{Sec.~5.1}.
Sufficient conditions for the stability of the differential equation~\eqref{eq:PHSnonlin} with the linear boundary condition~\eqref{eq:PHSlindamping} have been presented, for instance, in~\cite{VilZwa05,AugJac14,Sch22}.

\begin{proposition}\label{prp:PHSstability}
Suppose that Assumption~\textup{\ref{ass:PHSass}(a)--(b)} hold and
that $\phi:\Cc^p\to \Cc^p$ is a continuous monotone function, and let $(G,L,K)$ be as in~\eqref{eq:PHSBCS}.
If $x_0\in \Dom(L)$ is such that
\eq{
\MoveEqLeft[7] \left[\begin{matrix}\WBi\\ \WBh\end{matrix}\right]\left[\begin{matrix}
\mathcal{H}(b)x_0(b)\\
\mathcal{H}(a)x_0(a) 
\end{matrix}\right]
=\left[\begin{matrix}I\\ 0\end{matrix}\right]
\phi\left(-\WCo \left[\begin{matrix}
\mathcal{H}(b)x_0(b)\\
\mathcal{H}(a)x_0(a) 
\end{matrix}
\right]\right),
}
 then~\eqref{eq:PHSnonlin} with $u\equiv 0$
has a
 \emph{strong solution} $x$
satisfying $x(\cdot,0)=x_0$
in the sense that
 $t\mapsto x(\cdot,t)\in \Hloc[1](0,\infty;X)$, $x(t)\in \Dom(L)$ for all $t\geq 0$,  \eqref{eq:PHSnonlin_u} holds for all $t\geq 0$, and~\eqref{eq:PHSnonlin_x} holds for almost every $t\geq 0$ and almost every $\zeta\in [a,b]$.
The corresponding \emph{output} $y:\zinf \to\Cc^p$ defined by~\eqref{eq:PHSnonlin_y} for $t\geq 0$ is right-continuous and $y,\phi(-y)\in \Lp[\infty](0,\infty;\Cc^p)$.
If $x_0\in X$ and $u\in \Lploc[2](0,\infty;\Cc^p)$, then~\eqref{eq:PHSnonlin}
has a \emph{generalised solution} $x$
satisfying $x(\cdot,0)=x_0$
in the sense that
there exists a 
sequence $(x^k)_{k\in\N}$ of strong solutions of~\eqref{eq:PHSnonlin}  such that 
$\sup_{t\in [0,\tau]}\norm{x^k(\cdot,t)-x(\cdot,t)}_{X}\to 0$ for all $\tau>0$.
Every pair $x_1$ and $x_2$ of generalised solutions of~\eqref{eq:PHSnonlin}  satisfy
\eq{
\norm{x_2(\cdot,t)-x_1(\cdot,t)}_{\Lp[2](a,b)} \leq \norm{x_2(\cdot,0)-x_1(\cdot,0)}_{\Lp[2](a,b)}, \qquad t\geq 0.
}

Assume further
that $\phi$ satisfies Assumption~\textup{\ref{assum:phi}} and
that
the linear partial differential equation~\eqref{eq:PHSnonlin_x} together with the boundary condition 
\eqn{
\label{eq:PHSlindamping}
\MoveEqLeft[7] \left[\begin{matrix}\WBi+\WCo\\ \WBh\end{matrix}\right]\left[\begin{matrix}
\mathcal{H}(b)x_0(b)\\
\mathcal{H}(a)x_0(a) 
\end{matrix}\right]
=0
}
is strongly stable in the sense that all its classical solutions $x$ satisfy $\norm{x(\cdot,t)}_X\to 0$ as $t\to\infty$.
Then the origin is a globally asymptotically stable equilibrium point of~\eqref{eq:BCSnonlinNoInput}, i.e., 
the generalised solutions corresponding to all initial states $x_0\in X$ exist and
satisfy $\norm{x(\cdot,t)}_X\to 0$ as $t\to\infty$.
\end{proposition}

\begin{proof}
Assumption~\ref{ass:PHSass}(a)--(b) imply that $(G,L,K)$ defined in~\eqref{eq:PHSBCS} is an impedance passive boundary node on $(\Cc^p,X,\Cc^p)$.
Since $\ker(G)\cap \ker(K)$ in particular contains all functions $x\in \Lp[2](a,b;\Cc^n)$ for which $\mc Hx$ is a smooth function satisfying $(\mc Hx)(a)=(\mc Hx)(b)=0$, and since $\mc H(\zeta)\geq cI$ for some $c>0$ and for a.e. $\zeta\in [a,b]$, it is easy to verify that $\ker(G)\cap \ker(K)$ is dense in $X$.
The classical solutions of~\eqref{eq:PHSnonlin_x} with the boundary condition~\eqref{eq:PHSlindamping} coincide with the classical orbits of the contraction semigroup  generated by $L\vert_{\ker(G+K)}$, and our assumption therefore implies that this semigroup is strongly stable.
 Because of this, the claims follow directly from Theorems~\ref{thm:BCSWPNoInput} and~\ref{thm:BCSStabMain}.
\end{proof}

\subsection{A Timoshenko beam with saturated boundary damping}
\label{sec:Timoshenko}
We consider the well-posedness and stability of a one-dimensional Timoshenko beam model
\begin{subequations}
\label{Timo}
\eqn{
    &\rho(\zeta)w_{tt}(\zeta,t) = \left(K(\zeta)\left(w_\zeta(\zeta,t) - \varphi(\zeta,t)\right)\right)_\zeta\label{PDE_Timo_w}\\
    &I_\rho(\zeta)\varphi_{tt}(\zeta,t) = \left(EI(\zeta)\varphi_\zeta(\zeta,t)\right)_\zeta + K(\zeta)\left(w_\zeta(\zeta,t)-\varphi(\zeta,t)\right)\label{PDE_Timo_phi}\\
    &w_t(a,t) = 0,\qquad \varphi_t(a,t) = 0\label{Zero_Vel}\\
    &EI(b)\varphi_\zeta(b,t) = \phi_{\mathrm{sat}}\left(u_1(t)-\varphi_t(b,t)\right)\label{NonLin_Damping_phi}\\
    &K(b)\left(w_\zeta(b,t)-\varphi(b,t)\right) = \phi_{\mathrm{sat}}\left(u_2(t)-w_t(b,t)\right)\label{NonLin_Damping_w}
    }
\end{subequations}
for $\zeta\in (a,b)$ and $t\geq 0$, where $w(\zeta,t)$ is the transverse displacement of the beam and $\varphi(\zeta,t)$ is the rotation angle of a filament of the beam. The physical parameters $\rho, I_\rho, EI$ and $K$ are the mass per unit length, the rotary moment of inertia of a cross section of the beam, the product of Young's modulus of elasticity and the moment of inertia of the cross section and the shear modulus, respectively. We assume that these physical parameters are positive and continuously differentiable with respect to $\zeta\in[a,b]$. 
The external input of the system is 
$u(t) = (u_1(t),u_2(t))^\top\in \C^2,\,t\geq 0$.
In \eqref{NonLin_Damping_phi}--\eqref{NonLin_Damping_w}, $\phi_{\mathrm{sat}}:\Cc \to \Cc$ is the scalar saturation function defined by $\phi_{\mathrm{sat}}(u) = u$ when $\vert u\vert < 1$ and $\phi_{\mathrm{sat}}(u) = u/\vert u\vert$ when $\vert u\vert\geq 1$. The boundary conditions
 describe the situation where the beam is clamped at $\zeta= a$ and 
is subject to saturated damping and input at $\zeta=b$.
We define the distance between the two states $(w^1,\varphi^1)$ and $(w^2,\varphi^2)$ of~\eqref{Timo} at time $t\geq 0$ as
\eq{
 D_{(w^2,\varphi^2),(w^1,\varphi^1)}(t) 
&:= \int_0^1 \Bigl[K(\zeta)\left\vert w_\zeta^2(\zeta,t)-w_\zeta^1(\zeta,t) - \varphi^2(\zeta,t)+\varphi^1(\zeta,t)\right\vert^2\\
&\qquad+ \rho(\zeta)\left\vert w_t^2(\zeta,t)-w_t^1(\zeta,t)\right\vert^2
+ EI(\zeta)\left\vert\varphi_\zeta^2(\zeta,t)-\varphi_\zeta^1(\zeta,t)\right\vert^2\\
&\qquad + I_\rho(\zeta)\left\vert\varphi_t^2(\zeta,t)-\varphi_t^1(\zeta,t)\right\vert^2\Bigr]\d\zeta.
}
We define the \emph{energy} of the state $(w,\varphi)$ of~\eqref{Timo} at time $t\geq 0$ by $E_{w,\varphi}(t) := \frac{1}{2}D_{(w,\varphi),(0,0)}(t)$. The concepts of classical and generalised solutions of \eqref{Timo} are given in the following definition.

\begin{definition}\label{def:Sol_Timo}
The triple $(w,\varphi,u)$ is called a \emph{classical solution} of \eqref{Timo} on $\zinf$
if
$u\in C(\zinf;\C^2)$, $\varphi_t(b,\cdot), w_t(b,\cdot)\in C(\zinf)$, 
$t\mapsto w_\zeta(\cdot,t)-\varphi(\cdot,t)$, $t\mapsto w_t(\cdot,t)$,
$t\mapsto\varphi_\zeta(\cdot,t)$,
$t\mapsto \varphi_t(\cdot,t)\in C^1(\zinf;\L^2(a,b))$, and $w_\zeta(\cdot,t)-\varphi(\cdot,t)$, $w_t(\cdot,t)$, $\varphi_\zeta(\cdot,t)$, $\varphi_t(\cdot,t)\in\H^1(a,b)$, $t\ge 0$, and if~\eqref{Timo} hold for all $t\geq 0$ and for a.e. $\zeta\in[a,b]$.

The triple $(w,\varphi,u)$ is called a \emph{generalized solution} of \eqref{Timo} on $\zinf$ if $u\in \Lploc[2](0,\infty;\C^2)$,
$t\mapsto w_\zeta(\cdot,t)-\varphi(\cdot,t)$, $t\mapsto w_t(\cdot,t)$, $t\mapsto\varphi_\zeta(\cdot,t)$, $t\mapsto \varphi_t(\cdot,t)\in C(\zinf;\L^2(a,b))$, and $\varphi_t(b,\cdot), w_t(b,\cdot)\in \Lploc[2](0,\infty)$ are such that there exists a sequence $(w^k,\varphi^k,u^k)$ of classical solutions such that for every $\tau>0$ we have
\begin{align*}
    &\sup_{t\in [0,\tau]} D_{(w^k,\varphi^k),(w,\varphi)}(t)\to 0,\qquad \Vert \mathbf{P}_\tau u^k - \mathbf{P}_\tau u\Vert_{\L^2(0,\tau;\mathbb{C}^2)}\to 0,\\
    &\Vert \mathbf{P}_\tau(\varphi_{t}^k(b,\cdot),w_{t}^k(b,\cdot))^\top - \mathbf{P}_\tau(\varphi_t(b,\cdot),w_t(b,\cdot))^\top\Vert_{\L^2(0,\tau;\mathbb{C}^2)}\to 0
\end{align*}
as $k\to\infty$.
\end{definition}

The following result
establishes the existence of classical and generalised solutions of the nonlinear model~\eqref{Timo} and shows that the origin is a globally asymptotically stable equilibrium point of the system.

\begin{proposition}
If 
$w_0$, $w_1$, $\varphi_0$, and $\varphi_1$ 
and $u=(u_1,u_2)^\top\in \Lploc[2](0,\infty;\C^2)$
are such that $w_0'-\varphi_0, w_1, \varphi_0', \varphi_1\in\Lp[2](a,b)$, 
then~\eqref{Timo}
has a {generalised solution} $(w,\varphi,u)$ on $\zinf$ 
satisfying $w(\cdot,0) = w_0$, $w_t(\cdot,0) = w_1$, $\varphi(\cdot,0) = \varphi_0$ and $\varphi_t(\cdot,0) = \varphi_1$.
If $u\in \Hloc[1](0,\infty;\C^2)$,
$w_0'-\varphi_0, w_1, \varphi_0', \varphi_1\in\H^1(a,b)$ and
\begin{subequations}
\label{eq:Timo_compatconds}
\begin{align}
    &w_1(a) = 0,\qquad \varphi_1(a) = 0,\\
    &EI(b)\varphi_0'(b) = \phi_{\mathrm{sat}}(u_1(0)-\varphi_1(b)),\\
    &K(b)(w_0'(b)-\varphi_0(b)) = \phi_{\mathrm{sat}}(u_2(0)-w_1(b)),
\end{align}
\end{subequations}
then this $(w,\varphi,u)$ is a classical solution of~\eqref{Timo}.
Every pair $(w^1,\varphi^1,u^1)$ and $(w^2,\varphi^2,u^2)$ of generalised solutions satisfies 
\eq{
\MoveEqLeft D_{(w^2,\varphi^2),(w^1,\varphi^1)}(t) + \int_0^t \left\Vert\phi\left(u^2(s)-\left[\begin{smallmatrix}\varphi_{t}^2(b,s)\\w_{t}^2(b,s)\end{smallmatrix}\right]\right)-\phi\left(u^1(s)-\left[\begin{smallmatrix}\varphi_t^1(b,s)\\w_t^1(b,s)\end{smallmatrix}\right]\right)\right\Vert_{\Cc^2}^2\d s\\
&\leq D_{(w^2,\varphi^2),(w^1,\varphi^1)}(0)+\int_0^t \left\Vert u^2(s)-u^1(s)\right\Vert_{\Cc^2}^2\d s,
\qquad t\geq 0.
}
Moreover, if $u\equiv 0$, then the origin of \eqref{Timo} is a globally asymptotically stable equilibrium point, i.e., 
the generalised solutions described above
satisfy
$E_{w,\varphi}(t)\to 0$ as $t\to\infty$.
\end{proposition}

\begin{proof}
We start by showing that \eqref{Timo} can be formulated as an infinite-dimensional port-Hamiltonian system of the form~\eqref{eq:PHSnonlin} which satisfies Assumption \ref{ass:PHSass}. To this end we define 
\begin{align*}
    x_1(\zeta,t) &= w_\zeta(\zeta,t) - \varphi(\zeta,t),~~ x_2(\zeta,t) = \rho(\zeta)w_t(\zeta,t),\\
    x_3(\zeta,t) &= \varphi_\zeta(\zeta,t),~~ x_4(\zeta,t) = I_\rho(\zeta)\varphi_t(\zeta,t)
\end{align*}
and $x=(x_1,x_2,x_3,x_4)^\top$.
Then \eqref{PDE_Timo_w}--\eqref{PDE_Timo_phi} has the form \eqref{pH} on $X = \L^2(a,b;\mathbb{C}^4)$, where 
\begin{align*}
    P_1 &= \left[\begin{smallmatrix}0 & 1 & 0 & 0\\
    1 & 0 & 0 & 0\\
    0 & 0 & 0 & 1\\
    0 & 0 & 1 & 0\end{smallmatrix}\right],\  P_0 = \left[\begin{smallmatrix}0 & 0 & 0 & -1\\
    0 & 0 & 0 & 0\\
    0 & 0 & 0 & 0\\
    1 & 0 & 0 & 0\end{smallmatrix}\right],\ 
    \mathcal{H}(\zeta) = \mathrm{diag}(K(\zeta), \rho(\zeta)^{-1}, EI(\zeta), I_\rho(\zeta)^{-1})^\top.
\end{align*}
We have $P_1 = P_1^*$ and $\re P_0\leq 0$. If we define $\phi:\C^2\to \C^2$ such that $\phi(u) = (\phi_{\mathrm{sat}}(u_1),\phi_{\mathrm{sat}}(u_2))^\top$ for
 $u = ( u_1 , u_2)^\top$ and define $u(t)=(u_1(t),u_2(t))^\top$, $y(t)=(\varphi_t(b,t),w_t(b,t))^\top$, and
\begin{align*}
    \tilde{W}_{B,1} &= \left[\begin{matrix}0 \;\  0 \;\ 1 \;\ 0 \;\ 0 \;\ 0 \;\ 0 \;\ 0\\
    1 \;\ 0 \;\ 0 \;\ 0 \;\ 0 \;\ 0 \;\ 0 \;\ 0\end{matrix}\right], \quad
    \tilde{W}_{B,2} = \left[\begin{matrix}0 \;\ 0 \;\ 0 \;\ 0 \;\ 0 \;\ 1 \;\ 0 \;\ 0\\
    0 \;\ 0 \;\ 0 \;\ 0 \;\ 0 \;\ 0 \;\ 0 \;\ 1\end{matrix}\right],\\
    \tilde{W}_C &= \left[\begin{matrix}0 \;\ 0 \;\ 0 \;\ 1 \;\ 0 \;\ 0 \;\ 0 \;\ 0\\
    0 \;\ 1 \;\ 0 \;\ 0 \;\ 0 \;\ 0 \;\ 0 \;\ 0\end{matrix}\right],
\end{align*}
then \eqref{Zero_Vel}--\eqref{NonLin_Damping_w} are exactly of the form~\eqref{eq:PHSnonlin_u}--\eqref{eq:PHSnonlin_y}.

We will now show that Assumption \ref{ass:PHSass} is satisfied. 
A direct computation shows that 
\begin{align*}
    W_{B,1} &= \frac{1}{\sqrt{2}}\left[\begin{matrix}0 \;\ 0 \;\ 0 \;\ 1 \;\ 0 \;\ 0 \;\ 1 \;\ 0\\ 
    0 \;\ 1 \;\ 0 \;\ 0 \;\ 1 \;\ 0 \;\ 0 \;\ 0\end{matrix}\right],\quad W_{B,2} = \frac{1}{\sqrt{2}}\left[\begin{matrix}-1 \hspace{-1ex}& 0 \hspace{-.5ex}& 0  & \hspace{-1ex} 0 \;\;\, 0 \;\;\, 1 \;\;\, 0 \;\;\, 0\\ 
    \hspace{1ex} 0 & 0 \hspace{-.5ex} & \hspace{-1ex} -1 & \hspace{-1ex}0 \;\;\, 0 \;\;\, 0 \;\;\, 0 \;\;\, 1\end{matrix}\right],\\
    W_C &= \frac{1}{\sqrt{2}}\left[\begin{matrix}0 \;\ 0 \;\ 1 \;\ 0 \;\ 0 \;\ 0 \;\ 0 \;\ 1\\ 1 \;\ 0 \;\ 0 \;\ 0 \;\ 0 \;\ 1 \;\ 0 \;\ 0\end{matrix}\right].
\end{align*}
These matrices satisfy $W_B\Xi W_B^* = 0_{4\times 4}$ and $\rank\left(\left[\begin{smallmatrix}W_B\\ W_C\end{smallmatrix}\right]\right) = 6$ where $n=4$ and $\rank(W_C) = 2$. Thus Assumption \ref{ass:PHSass}\textup{(a)} is satisfied. Now observe that
\begin{equation*}
\ker\left(W_{B,2}\right) = \bigl\{(g_1, \ g_2, \ g_3 ,\ g_4 ,\ g_5 ,\ g_6 ,\ g_7 ,\ g_8)^\top\in\mathbb{C}^8\,\bigm|\,g_1 = g_6,~g_3 = g_8\bigr\}.
\end{equation*}
Moreover, we have $\langle\Xi g,g\rangle = 2\re\left(\overline{g_1}g_5 + \overline{g_2}g_6 + \overline{g_3}g_7 + \overline{g_4}g_8\right)$ and
it is easy to check that if $g\in \ker(W_{B,2})$, then
\begin{align*}
    2\re\langle W_{B,1}g,W_C g\rangle 
    &= \re\left(\overline{g_1}(g_2+g_5) + \overline{g_3}(g_4+g_7) + \overline{g_6}(g_2+g_5)+\overline{g_8}(g_4+g_7)\right)\\
    &=\langle\Xi g,g\rangle.
\end{align*}
Thus Assumption \ref{ass:PHSass}\textup{(b)} is satisfied.
Finally, it is straightforward to check that $P_1\mathcal{H}(\zeta) = S(\zeta)^{-1}\Delta(\zeta) S(\zeta)$ for $\zeta\in[0,1]$ with 
$S\in C^1([0,1];\Cc^4)$ and with 
$\Delta = \diag(\sqrt{EI/I_\rho},-\sqrt{EI/I_\rho}, \sqrt{K/\rho},-\sqrt{K/\rho})$.
Thus Assumption \ref{ass:PHSass}\textup{(c)} is satisfied.

The properties in Examples~\ref{ex:scalarSatComplex} and~\ref{ex:scalarSatComplex_AssumStab} imply that $\phi$ satisfies Assumptions \ref{ass:PhiAss}  with $\phicoeff=1$ and Assumption~\ref{assum:phi} with some $\alpha,\beta,\delta>0$.
Let $w_0$, $w_1$, $\varphi_0$, and $\varphi_1$ 
and $u=(u_1,u_2)^\top\in \Lploc[2](0,\infty;\C^2)$
be such that $w_0'-\varphi_0, w_1, \varphi_0', \varphi_1\in\Lp[2](a,b)$, 
and define $x_0= (w_0'-\varphi_0,\rho(\cdot)w_1,\varphi',I_\rho(\cdot)\varphi_1)^\top$.
Then $x_0\in X$ and \cref{prp:PHSwellposedness} implies that~\eqref{eq:PHSnonlin} has a generalised solution on $\zinf$ satisfying $x(\cdot,0)=x_0$. 
It is straightforward to check that there exist $w:[a,b]\times \zinf\to \C$ and $\varphi:[a,b]\times \zinf\to\C$ such that 
$t\mapsto w(\cdot,t)\in C(\zinf;\H^1(a,b))$, 
$t\mapsto w_t(\cdot,t)\in C(\zinf;\Lp[2](a,b))$, 
$t\mapsto \varphi(\cdot,t)\in C(\zinf;\H^1(a,b))$, 
$t\mapsto \varphi_t(\cdot,t)\in C(\zinf;\Lp[2](a,b))$, 
and
$x(\zeta,t)=(w_\zeta(\zeta,t)-\varphi(\zeta,t),\rho(\zeta)w_t(\zeta,t),\varphi_\zeta(\zeta,t),I_\rho(\zeta)\varphi_t(\zeta,t))^\top$ for all $t\ge 0$ and for a.e. $\zeta\in [a,b]$.
The assumptions that $K, \rho$, $EI, I_\rho\in C^1([a,b])$ and $K^{-1}, \rho^{-1}, EI^{-1}, I_\rho^{-1}\in \L^\infty(a,b)$ imply that
$(w,\varphi,u)$ is a generalised solution of~\eqref{Timo} in the sense of Definition~\ref{def:Sol_Timo}.  In addition, the estimate for the pairs of generalised solutions of~\eqref{eq:PHSnonlin} implies the estimate for the pairs of generalised solutions of~\eqref{Timo} in the claim.
Moreover, if $u\in \Hloc[1](0,\infty;\C^2)$,
$w_0'-\varphi_0, w_1, \varphi_0', \varphi_1\in\H^1(a,b)$ and~\eqref{eq:Timo_compatconds} hold, then  $x_0$ satisfies $\mc H x_0\in H^1(a,b)$, $\WBh x_0=0$, and~\eqref{eq:PHScompatIC} holds. Thus Proposition~\ref{prp:PHSwellposedness} shows that $x$ is a classical solution of~\eqref{eq:PHSnonlin}, and consequently $(w,\varphi,u)$ is a classical solution of~\eqref{Timo} in the sense of Definition~\ref{def:Sol_Timo}.

We have from~\citel{Villegas}{Ex.~5.22} that the semigroup $\T^K$ associated with the partial differential equation \eqref{PDE_Timo_w}--\eqref{PDE_Timo_phi} with the boundary condition of the form \eqref{eq:PHSlindamping} is strongly stable.
Therefore the claim regarding the convergence of generalised solutions directly from Proposition~\ref{prp:PHSstability}. 
The claim about stability follows from Proposition \ref{prp:PHSstability}. 
\end{proof}

\subsection{A two-dimensional heat equation}
\label{sec:HeatEquation}

We consider a heat equation on  a two-dimensional spatial domain $\Omega\subset \R^2$. We assume that $\Omega$ is open, connected, and bounded subset of $\R^2$ and that its boundary $\partial \Omega$ is piecewise $C^2$-smooth (see~\citel{ByrGil02}{Sec.~2} for the precise definition).
Our heat equation with a scalar-valued nonlinear boundary input $u(t)\in\Cc $ and a boundary output $y(t)\in \Cc$
on $\Omega$ has the form
\begin{subequations}
\label{eq:Heat}
\eqn{
\frac{\partial x}{\partial t}(\zeta,t) &= \Delta x (\zeta,t),  &&\zeta\in\Omega,  \ t>0
    \label{eq:Heat_state},\\
    \frac{\partial x}{\partial n}(\zeta,t) &= b(\zeta)\phi\left(u(t)-y(t)\right),  &&\zeta\in \partial\Omega, \ t>0 \label{eq:Heat_input}
\\
y(t)&= \int_{\partial \Omega} b(\zeta)x( \zeta,t)\d \zeta, && t>0
\label{eq:Heat_output}
}
\end{subequations}
for $t>0$,
where 
$\frac{\partial x}{\partial n}(\zeta,t) $
 is the outward normal derivative of $x(\cdot,t)$ at $\zeta\in\doo\Omega$, 
where $b\in\L^2(\doo \Omega;\mathbb{R})$ is the profile function associated with the Neumann boundary input and the boundary output $y(t)$.
We assume that $\int_{\doo \Omega}b(\zeta)\d \zeta\neq 0$.
The following  result based on Theorems~\ref{thm:LPSolutionsMain} and~\ref{thm:StabSecond} establishes the existence of classical and generalised solutions as well as the global asymptotic stability of~\eqref{eq:Heat} in the case where $\phi:\C\to\C$ is a saturation function.

\begin{proposition}
\label{prp:HeatProposition}
Suppose that $b\in\Lp[2](\doo \Omega;\R)$ satisfies  $\int_{\doo \Omega}b(\zeta)\d \zeta\neq 0$ and let $\phi: \Cc\to \Cc$ be the saturation function defined by $\phi(u)=u$ if $\abs{u}\le 1$ and $\phi(u)=u/\abs{u}$ if $\abs{u}>1$.
Let $x_0\in \H^2(\Omega)$ and $u\in \Hloc[1](0,\infty)$ be such that 
\eqn{
\label{eq:HeatCompatBC}
    \frac{\partial x_0}{\partial n} = b(\cdot)\phi\left(u(0)-\int_{\doo\Omega} b(\zeta)x_0( \zeta)\d \zeta
\right) \quad \mbox{on} \ \ \doo \Omega.
}
Then~\eqref{eq:Heat} has a classical state trajectory $x:\Omega\times \zinf\to \Cc$
satisfying $x(\cdot,0)=x_0$
and classical output $y\in C(\zinf)$ 
in the sense that
$t\mapsto x(\cdot,t)\in C^1(\zinf;\Lp[2](\Omega))$, 
$x(\cdot,t)\in \H^2(\Omega)$ for all $t\geq 0$, 
\eqref{eq:Heat_state} holds for all $t>0$ and for a.e. $\zeta\in \Omega$,
\eqref{eq:Heat_input} holds for $t\ge 0$ and for a.e. $\zeta\in \doo\Omega$,
 and~\eqref{eq:Heat_output} holds for all $t\geq0$.

If  $x_0\in \Lp[2](\Omega)$ and $u\in \Lploc[2](0,\infty)$, then~\eqref{eq:Heat} has a generalised state trajectory
$x\in C(\zinf;\Lp[2](\Omega))$ 
satisfying $x(\cdot,0)=x_0$
and a generalised output
$y\in C(\zinf)$ 
in the sense that 
there exist
sequences $(u^k)_{k\in\N}$, $(x^k)_{k\in\N}$ and $(y^k)_{k\in\N}$ 
with the following properties.
\begin{itemize}
\item 
For each $k\in\N$, 
$x^k$ and $y^k$ are the classical state trajectory and output, respectively, of~\eqref{eq:Heat} corresponding to the initial state $x^k(0)$ and input $u^k\in C(\zinf)$
\item For all $\tau>0$ we have $(\Pt[\tau] x^k,\Pt[\tau] u^k,\Pt[\tau] y^k)^\top \to (\Pt[\tau] x,\Pt[\tau] u,\Pt[\tau] y)^\top$ as $t\to \infty$  in $C([0,\tau];\Lp[2](\Omega))\times \Lp[2](0,\tau)\times \Lp[2](0,\tau)$.
\end{itemize}
If $x_1,x_2\in C(\zinf;\Lp[2](\Omega))$ and $y_1,y_2\in C(\zinf)$ are two generalised state trajectories and outputs of~\eqref{eq:Heat} corresponding to initial states $x_{01},x_{02}\in \Lp[2](\Omega)$ and inputs $u_1,u_2\in \Lploc[2](0,\infty)$, then they satisfy the estimate~\eqref{eq:ContractivityEstim} with $X=\Lp[2](\Omega)$ and $\phicoeff =1$  for all $t\geq 0$.

Finally, if $u\equiv 0$, then for every $x_0\in \Lp[2](\Omega)$ the generalised state trajectory $x$ and the generalised output $y$ satisfy $\norm{x(\cdot,t)}_{\Lp[2](\Omega)}\to 0$ as $t\to\infty$ and $y, \phi(-y)\in \Lp[2](0,\infty)$.
\end{proposition}

\begin{proof}
The heat equation~\eqref{eq:Heat} can be represented as a system of the form~\eqref{eq:SysNodeSysMain} on the spaces $X=\Lp[2](\Omega)$ and $U=\C$
if we 
define $S$ as the system node associated with the linear heat equation~\eqref{eq:Heat_state} on $\Omega$ with the linear Neumann boundary input 
\eqn{
\label{eq:HeatLinInput}
\frac{\partial x}{\partial n}(\zeta,t)  = b(\zeta)u(t), \qquad \zeta\in \doo\Omega, \ t>0
}
and with the boundary output~\eqref{eq:Heat_output}.
We have from~\citel{ByrGil02}{Sec.~7} that 
this $S$ is defined by 
\eq{
\Dom(S) &= \Bigl\{(x,u)^\top\in \H^1(\Omega)\times \C\,\Bigm|\, \Delta x\in \Lp[2](\Omega),\ \frac{\partial x}{\partial n}=b(\cdot)u \mbox{ on} \ \doo \Omega \Bigr\}\\
S\pmat{x\\ u}
&=
\pmat{
A\& B\\
C\& D 
}
\pmat{x\\ u} = 
\pmat{
\Delta x\\
\displaystyle \int_{\doo \Omega} b(\zeta)x(\zeta)\d \zeta
}
, \qquad \pmat{x\\u}\in \Dom(S).
}
It is shown in~\citel{ByrGil02}{Cor.~1 \& Sec.~7} that $S$ is a well-posed system node on $(\C,X,\C)$ and by~\citel{ByrGil02}{Cor.~2} the associated well-posed system $\Sigma = (\T,\Phi,\Psi,\F)$ is uniformly regular in the sense that its transfer function satisfies $P(\gl)\to 0$ as $\gl\to\infty$ with $\gl>0$.
To check that $S$ is impedance passive, let $(x,u)^\top \in \Dom(S)$. Integration by parts and the assumption that $b$ is real-valued imply that
\eq{
\re \iprod{A\& B \pmatsmall{x\\ u} }{x}_X
&= \re \iprod{\Delta x}{x}_{\Lp[2](\Omega)}
= \re\int_{\doo \Omega} b(\zeta)u \overline{x(\zeta)} \d \zeta - \norm{\nabla x}_{\Lp[2](\Omega)}^2\\
&\le \re \iprod{ C\& D \pmatsmall{x\\u}}{u} .
}
Thus $S$ is impedance passive.
Since $b\neq 0$, it is easy to see that $P(\gl)\in\R$ for $\gl>0$ and $P(\gl)\neq 0$ for some $\gl>0$,
 and thus
 (since $S$ is impedance passive) 
we have $\re P(\gl)\ge c_\gl $ for some $\gl,c_\gl>0$.
Finally, we note that the saturation function $\phi$ satisfies Assumptions~\ref{ass:PhiAss} and~\ref{assum:phi}
 (see Examples~\ref{ex:scalarSatComplex} and~\ref{ex:scalarSatComplex_AssumStab}).
Since $(x_0,u(0))^\top \in \Dcomp$ if and only if $x_0\in \H^2(\Omega)$ and~\eqref{eq:HeatCompatBC} holds,
 the existence of classical and generalised state trajectories and outputs and the estimate~\eqref{eq:ContractivityEstim} follow directly from \cref{thm:LPSolutionsMain}.

To prove the claims regarding the convergence of solutions as $t\to \infty$,
we note that the operator $A^K$ 
in Theorem~\ref{thm:StabSecond} is given by
\eq{
     D(A^K) &=
 \Bigl\{x\in\H^1(\Omega)\,\Bigm|\,  \Delta x\in \Lp[2](\Omega), \
\frac{\partial x}{\partial n}= -b(\cdot)\int_{\doo \Omega} b(\zeta)x(\zeta)\d \zeta \ \ \mbox{on} \ \doo \Omega\Bigr\}\\
    A^Kx &= \Delta x,
    \quad x\in D(A^K) .
}
We have from Theorem~\ref{thm:SysNodeFeedback} and~\citel{Sta05book}{Lem.~7.4.4(iii)} that $(\gl-A^K)\inv = (\gl-A)\inv - (\gl-A)\inv B(I+P(\gl))\inv C(\gl-A)\inv$, where the first term $(\gl-A)\inv$ is compact and the second term is rank one and bounded. Thus
 $A^K$ has compact resolvent. 
Integration by parts shows that for $x\in \Dom(A^K)$ we have
\eqn{
\label{eq:HeatAKiprod}
\langle A^K x, x\rangle_X = &-\left\vert\int_{\doo \Omega} b(\zeta)x(\zeta)\d \zeta\right\vert^2 - \Vert\nabla x\Vert^2_{\L^2(\Omega)}\leq 0,
}
and thus $A^K$ is self-adjoint and $A^K\le 0$. If $A^K x=0$, then~\eqref{eq:HeatAKiprod} shows that $\norm{\nabla x}_{\Lp[2](\Omega)}=0$, which implies that $x$ is constant with respect to $\zeta\in \Omega$. But since $\int_{\doo \Omega} b(\zeta)\d \zeta\neq 0$ by assumption and since~\eqref{eq:HeatAKiprod} implies 
$\int_{\doo \Omega} b(\zeta)x(\zeta)\d \zeta=0$,  we necessarily have $x=0$. Thus $0$ is not an eigenvalue of $A^K$ and, since $A^K$ has compact resolvent, we have $0\in\rho(A)$.
This implies that $A^K$ generates an exponentially stable semigroup on $X$.
Moreover, since $\phi$ satisfies Assumption~\ref{ass:PhiAss}, we have $\re \iprod{\phi(u_2)-\phi(u_1)}{u_2-u_1}> 0$ whenever $\phi(u_1)\neq \phi(u_2)$ and the domain $\Dom(A_\phi)$ of $A_\phi$ in~\eqref{eq:nACPStabSec} is dense in $X$ by Lemma~\ref{lem:DomAphiDense}.
Because of this, 
 Theorem~\ref{thm:StabSecond} implies that 
for any $x_0\in \Lp[2](\Omega)$ and for $u\equiv 0$ the corresponding generalised state trajectory $x$ and generalised output $y$ satisfy $\norm{x(\cdot,t)}_{\Lp[2](\Omega)}\to 0$ as $t\to\infty$ and $y,\phi(-y)\in \Lp[2](0,\infty)$.
\end{proof}

The following result shows that the generalised state trajectory $x$ in \cref{prp:HeatProposition} also has an interpretation as a \emph{weak solution} of the heat equation~\eqref{eq:Heat}.

\begin{proposition}
\label{prp:HeatWeakSol}
For every $x_0\in \Lp[2](\Omega)$ and $u\in \Lploc[2](0,\infty)$
the generalised state trajectory and output in Proposition~\textup{\ref{prp:HeatProposition}} satisfy
$t\mapsto x(\cdot,t)\in \Lploc[2](0,\infty;\H^1(\Omega))$
 and 
\eq{
\iprod{x(\cdot,t)-x_0}{z}_{\Lp[2](\Omega)}  &= \int_0^t \Bigl[
 \phi(u(s)-y(s))\hspace{-.5ex}\int_{\doo\Omega} \hspace{-1ex} b(\zeta)\overline{z(\zeta)} \d \zeta
-\iprod{\nabla x(\cdot,s)}{\nabla z}_{\Lp[2](\Omega;\C^2)} 
\Bigr] \d s
\\
y(t) &= \int_{\doo\Omega} b(\zeta) x(\zeta,t)\d \zeta
}
for all  $z\in \H^1(\Omega)$, where the first identity holds for all $t\geq 0$ and the second one holds for a.e. $t\ge 0$.
\end{proposition}

\begin{proof}
Let $X=\Lp[2](\Omega)$ and let $S$ be as in the proof of \cref{prp:HeatProposition} and
let $x$ and $y$ be the generalised state trajectory and output of~\eqref{eq:Heat} in Proposition~\ref{prp:HeatProposition}.
We have from the proof of \cref{prp:HeatProposition} and \cref{thm:LPSolutionsMain} that  $x$ and $y$ satisfy~\eqref{eq:LPSemigState}.
Therefore~\citel{TucWei09book}{Rem.~4.2.6} implies that 
\eqn{
\label{eq:HeatWeakSolFirst}
\MoveEqLeft[4]\iprod{x(\cdot,t)-x_0}{z}_X
=\int_0^t \Bigl[\iprod{x(\cdot,s)}{\Delta z}_X + \phi(u(s)-y(s))\int_{\doo\Omega} b(\zeta) \overline{z(\zeta)}\d \zeta \Bigr] \d s
}
 for all $z\in \Dom(A)=\setm{x\in \H^2(\Omega)}{\frac{\partial x}{\partial n}=0 \mbox{ on } \partial \Omega}$ and $t\geq 0$.
Let $(x^k)_{k\in\N}$, $(u^k)_{k\in\N}$ and $(y^k)_{k\in\N}$ be sequences with the properties described in \cref{prp:HeatProposition}. 
Denote $x^k(t):=x^k(\cdot,t)$ and $x(t):=x(\cdot,t)$ for $t\geq 0$ and $k\in\N$ for brevity.
Integration by parts shows that
for $k,n\in\N$
the classical state trajectories $x^k$ and $x^n$  of~\eqref{eq:Heat}  satisfy
\eq{
\frac{\d }{\d s}
\MoveEqLeft\norm{x^k(s)-x^n(s)}_X^2
= 2\re \iprod{\Delta (x^k(s)- x^n(s))}{x^k(s)-x^n(s)}_X\\
&= 2\re \iprod{\phi(u^k(s)-y^k(s))-\phi(u^n(s)-y^n(s))}{ y^k(s)-y^n(s)}_{\C}\\
 &\quad -2 \norm{\nabla (x^k(s)- x^n(s))}_X^2.
}
Integrating with respect to $s$ from $0$ to $t$ and using Assumption~\ref{ass:PhiAss}  and Lemma~\ref{lem:PhiProps}(b) implies that

\eq{
 2\int_0^t\norm{\nabla x^k(s)-\nabla x^n(s)}_X^2 \d s
\leq 
\norm{x^k(0)-x^n(0)}_X^2+  \norm{\Pt u^k-\Pt u^n}_{\Lp[2](0,t)} \to 0
}
as $k,n\to \infty$. 
Since we also have
$\norm{\Pt x^k(\cdot)-\Pt x^n(\cdot)}_{\Lp[2](0,t;X)}\to 0$
 as $k,n\to \infty$, we conclude that $( \Pt x^k(\cdot))_{k\in\N}$ is  a Cauchy sequence in $\Lp[2](0,t;\H^1(\Omega))$. 
 Since we have $\sup_{s\in[0,t]}\norm{x^k(s)-x(s)}\to 0$ as $k\to\infty$, 
we can deduce that $\Pt x \in \Lp[2](0,t;\H^1(\Omega))$.
Since $t>0$ was arbitrary, we have $x\in \Lploc[2](0,\infty;\H^1(\Omega))$ as claimed.
Since $y^k(t)=\int_{\doo\Omega} b(\zeta)x^k(\zeta,t)\d \zeta$ for all $t\ge 0$ and since $\Pt x^k(\cdot)\to \Pt x(\cdot)$ in $\Lp[2](0,t;H^1(\Omega))$ as $k\to \infty$, it is easy to show that
$y(t)=\int_{\doo \Omega} b(\zeta)x(\zeta,t)\d \zeta$ for a.e. $t\ge 0$.

Finally, for $z\in \Dom(A)$ and for a.e. $s\ge 0$ we have  $\iprod{x(\cdot,s)}{\Delta z}_X = -\iprod{\nabla x(\cdot,s)}{\nabla z}_X$ by integration by parts. Since $\Dom(A)$ is dense in $\H^1(\Omega)$, the first identity in the claim follows from~\eqref{eq:HeatWeakSolFirst}.
\end{proof}


\begin{thebibliography}{10}

\bibitem{AlaAmm11}
F.~Alabau-Boussouira and K.~Ammari.
\newblock Sharp energy estimates for nonlinearly locally damped {PDE}s via
  observability for the associated undamped system.
\newblock {\em J. Funct. Anal.}, 260(8):2424--2450, 2011.

\bibitem{AlaPri17}
F.~Alabau-Boussouira, Y.~Privat, and E.~Tr\'{e}lat.
\newblock Nonlinear damped partial differential equations and their uniform
  discretizations.
\newblock {\em J. Funct. Anal.}, 273(1):352--403, 2017.

\bibitem{AmmTuc01}
K.~Ammari and M.~Tucsnak.
\newblock Stabilization of second order evolution equations by a class of
  unbounded feedbacks.
\newblock {\em ESAIM Control Optim. Calc. Var.}, 6:361--386, 2001.

\bibitem{Aug19}
B.~Augner.
\newblock Well-posedness and stability of infinite-dimensional linear
  port-{H}amiltonian systems with nonlinear boundary feedback.
\newblock {\em SIAM J. Control Optim.}, 57(3):1818--1844, 2019.

\bibitem{AugJac14}
B.~Augner and B.~Jacob.
\newblock Stability and stabilization of infinite-dimensional linear
  port-{H}amiltonian systems.
\newblock {\em Evol. Equ. Control The.}, 3(2):207--229, 2014.

\bibitem{Ben78a}
C.~Benchimol.
\newblock Feedback stabilizability in {H}ilbert spaces.
\newblock {\em Appl. Math. Optim.}, 4(3):225--248, 1978.

\bibitem{Ber12}
L.~Berrahmoune.
\newblock Stabilization of unbounded linear control systems in {H}ilbert space
  with monotone feedback.
\newblock {\em Asymptotic Anal.}, 80(1-2):93--131, 2012.

\bibitem{BinTab23}
A.~Binid, A.~M. Taboye, and M.~Laabissi.
\newblock Cone-bounded feedback laws for linear second order systems.
\newblock {\em Evol. Equ. Control The.}, 12(4):1174--1192, 2023.

\bibitem{ByrGil02}
C.I. Byrnes, D.S. Gilliam, V.I. Shubov, and G.~Weiss.
\newblock Regular {{Linear Systems Governed}} by a {{Boundary Controlled Heat
  Equation}}.
\newblock {\em J. Dyn. Control Syst.}, 8(3):341--370, 2002.

\bibitem{ChiPau23}
R.~{Chill}, L.~{Paunonen}, D.~{Seifert}, R.~{Stahn}, and Yu. {Tomilov}.
\newblock Non-uniform stability of damped contraction semigroups.
\newblock {\em Anal. PDE}, 16(5):1089--1132, 2023.

\bibitem{CurZwa16}
R.~Curtain and H.~Zwart.
\newblock Stabilization of collocated systems by nonlinear boundary control.
\newblock {\em Syst. Control Lett.}, 96:11--14, 2016.

\bibitem{CurOos01}
R.~F. Curtain and J.~C. Oostveen.
\newblock The {P}opov criterion for strongly stable distributed parameter
  systems.
\newblock {\em Internat. J. Control}, 74(3):265--280, 2001.

\bibitem{CurWei06}
R.~F. Curtain and G.~Weiss.
\newblock Exponential stabilization of well-posed systems by colocated
  feedback.
\newblock {\em SIAM J. Control Optim.}, 45(1):273--297, 2006.

\bibitem{CurWei19}
R.~F. Curtain and G.~Weiss.
\newblock Strong stabilization of (almost) impedance passive systems by static
  output feedback.
\newblock {\em Math. Control Relat. Fields}, 9(4):643--671, 2019.

\bibitem{FarWeg16}
B.~Farkas and S.-A. Wegner.
\newblock Variations on {B}arb\u{a}lat's lemma.
\newblock {\em Am. Math. Mon.}, 123(8):825--830, 2016.

\bibitem{GraCal06}
P.~Grabowski and F.~M. Callier.
\newblock On the circle criterion for boundary control systems in factor form:
  {L}yapunov stability and {L}ur'e equations.
\newblock {\em ESAIM Control Optim. Calc. Var.}, 12(1):169--197, 2006.

\bibitem{GraCal11}
P.~Grabowski and F.~M. Callier.
\newblock Lur'e feedback systems with both unbounded control and observation:
  well-posedness and stability using nonlinear semigroups.
\newblock {\em Nonlinear Anal.}, 74(10):3065--3085, 2011.

\bibitem{GuiLog19}
C.~Guiver, H.~Logemann, and M.~R. Opmeer.
\newblock Infinite-{{Dimensional Lur}}'e {{Systems}}: {{Input-To-State
  Stability}} and {{Convergence Properties}}.
\newblock {\em SIAM J. Control Optim.}, 57(1):334--365, 2019.

\bibitem{Har85}
A.~Haraux.
\newblock Stabilization of trajectories for some weakly damped hyperbolic
  equations.
\newblock {\em J. Differ. Equations}, 59(2):145--154, 1985.

\bibitem{HasCal19}
A.~Hastir, F.~Califano, and H.~Zwart.
\newblock Well-posedness of infinite-dimensional linear systems with nonlinear
  feedback.
\newblock {\em Syst. Control Lett.}, 128:19--25, 2019.

\bibitem{JacobMorrisZwart}
B.~Jacob, K.~Morris, and H.J. Zwart.
\newblock ${C}_0$-semigroups for hyperbolic partial differential equations on a
  one-dimensional spatial domain.
\newblock {\em J. Evol. Equ.}, 15:493--502, 2015.

\bibitem{JacSch20}
B.~Jacob, F.~L. Schwenninger, and L.~A. Vorberg.
\newblock Remarks on input-to-state stability of collocated systems with
  saturated feedback.
\newblock {\em Math. Control Signal.}, 32(3):293--307, 2020.

\bibitem{JacobZwart}
B.~Jacob and H.J. Zwart.
\newblock {\em Linear Port-{H}amiltonian Systems on Infinite-dimensional
  Spaces}.
\newblock Number 223 in Operator Theory: Advances and Applications. Springer,
  2012.

\bibitem{JolLau20}
R.~Joly and C.~Laurent.
\newblock Decay of semilinear damped wave equations: cases without geometric
  control condition.
\newblock {\em Ann. H. Lebesgue}, 3:1241--1289, 2020.

\bibitem{LaaTab21}
M.~Laabissi and A.~M. Taboye.
\newblock Strong stabilization of non-dissipative operators in {H}ilbert spaces
  with input saturation.
\newblock {\em Math. Control Signal}, 33(3):553--568, 2021.

\bibitem{LasSei03}
I.~Lasiecka and T.~I. Seidman.
\newblock Strong stability of elastic control systems with dissipative
  saturating feedback.
\newblock {\em Syst. Control Lett.}, 48(3):243--252, March 2003.

\bibitem{Log20}
H.~Logemann.
\newblock Some spectral properties of operator-valued positive-real functions.
\newblock {\em Syst. Control Lett.}, 145:104786, 9, 2020.

\bibitem{LogRya00}
H.~Logemann and E.~P. Ryan.
\newblock Time-varying and adaptive integral control of infinite-dimensional
  regular linear systems with input nonlinearities.
\newblock {\em SIAM J. Control Optim.}, 38(4):1120--1144, 2000.

\bibitem{MalSta06}
J.~Malinen and O.~J. Staffans.
\newblock Conservative boundary control systems.
\newblock {\em J. Differ. Equations}, 231(1):290--312, 2006.

\bibitem{MarAnd17}
S.~Marx, V.~Andrieu, and C.~Prieur.
\newblock Cone-bounded feedback laws for {$m$}-dissipative operators on
  {H}ilbert spaces.
\newblock {\em Math. Control Signal}, 29(4):Art. 18, 32, 2017.

\bibitem{MarChi20}
S.~Marx, Y.~Chitour, and C.~Prieur.
\newblock Stability analysis of dissipative systems subject to nonlinear
  damping via {L}yapunov techniques.
\newblock {\em IEEE T. Automat. Cont.}, 65(5):2139--2146, 2020.

\bibitem{MarWei25preprint}
S.~{Marx} and G.~{Weiss}.
\newblock The well-posedness of an impedance passive nonlinear system with a
  static monotone feedback operator.
\newblock In {\em Proceedings of the 5th IFAC Workshop on Control of Systems
  Governed by Partial Differential Equations}, submitted, 2025.

\bibitem{Miy92book}
I.~Miyadera.
\newblock {\em Nonlinear semigroups}, volume 109.
\newblock American Mathematical Soc., 1992.

\bibitem{NatBen16}
V.~Natarajan and J.~Bentsman.
\newblock Approximate local output regulation for nonlinear distributed
  parameter systems.
\newblock {\em Math. Control Signal}, 28(3):Art. 24, 44, 2016.

\bibitem{Pau19}
L.~Paunonen.
\newblock Stability and robust regulation of passive linear systems.
\newblock {\em SIAM J. Control Optim.}, 57(6):3827--3856, 2019.

\bibitem{PauLeg21}
L.~Paunonen, Y.~Le~Gorrec, and H.~Ramirez.
\newblock A {L}yapunov approach to robust regulation of distributed
  port–{H}amiltonian systems.
\newblock {\em IEEE T. Automat. Cont.}, 66(12):6041--6048, 2021.

\bibitem{PhiRei23arxiv}
F.~{Philipp}, T.~{Reis}, and M.~{Schaller}.
\newblock {Infinite-dimensional port-Hamiltonian systems -- a system node
  approach}.
\newblock {\em arXiv e-prints}, page arXiv:2302.05168, February 2023.

\bibitem{Sal87a}
D.~Salamon.
\newblock Infinite-dimensional linear systems with unbounded control and
  observation: {A} functional analytic approach.
\newblock {\em Trans. Amer. Math. Soc.}, 300(2):383--431, 1987.

\bibitem{Sch22}
J.~Schmid.
\newblock Stabilization of port-{H}amiltonian systems with discontinuous energy
  densities.
\newblock {\em Evol. Equ. Control The.}, 11(5):1775--1795, 2022.

\bibitem{SchZwa21}
J.~Schmid and H.~Zwart.
\newblock Stabilization of port-{H}amiltonian systems by nonlinear boundary
  control in the presence of disturbances.
\newblock {\em ESAIM Control Optim. Calc. Var.}, 27:Paper No. 53, 37, 2021.

\bibitem{SeiLi01}
T.~I. Seidman and H.~Li.
\newblock A note on stabilization with saturating feedback.
\newblock {\em Discrete Cont. Dyn. S.}, 7(2):319--328, 2001.

\bibitem{Sin23phd}
S.~Singh.
\newblock {\em Incrementally Scattering Passive Infinite Dimensional Systems on
  Hilbert Spaces}.
\newblock PhD thesis, Tel Aviv University, 2023.

\bibitem{SinWei22}
S.~Singh and G.~Weiss.
\newblock Nonlinear perturbation of a class of conservative linear system.
\newblock In {\em 2022 {{IEEE}} 61st {{Conference}} on {{Decision}} and
  {{Control}} ({{CDC}})}, pages 396--402, 2022.

\bibitem{SinWei23}
S.~Singh and G.~Weiss.
\newblock Second {{Order Systems}} on {{Hilbert Spaces}} with {{Nonlinear
  Damping}}.
\newblock {\em SIAM J. Control Optim.}, 61(4):2630--2654, 2023.

\bibitem{Sle74}
M.~Slemrod.
\newblock A note on complete controllability and stabilizability for linear
  control systems in {H}ilbert space.
\newblock {\em SIAM J. Control}, 12:500--508, 1974.

\bibitem{Sle89}
M.~Slemrod.
\newblock Feedback stabilization of a linear control system in {H}ilbert space
  with an a priori bounded control.
\newblock {\em Math. Control Signals Systems}, 2(3):265--285, 1989.

\bibitem{Sta05book}
O.~Staffans.
\newblock {\em Well-Posed Linear Systems}.
\newblock Cambridge University Press, 2005.

\bibitem{StaWei02}
O.~Staffans and G.~Weiss.
\newblock Transfer functions of regular linear systems. {II}. {T}he system
  operator and the {L}ax-{P}hillips semigroup.
\newblock {\em T. Am. Math. Soc.}, 354(8):3229--3262, 2002.

\bibitem{Sta02}
O.~J. Staffans.
\newblock Passive and conservative continuous-time impedance and scattering
  systems. {Part I}: {W}ell-posed systems.
\newblock {\em Math. Control Signal}, 15(4):291--315, 2002.

\bibitem{Tro14}
S.~Trostorff.
\newblock A characterization of boundary conditions yielding maximal monotone
  operators.
\newblock {\em J. Funct. Anal.}, 267(8):2787--2822, 2014.

\bibitem{TucWei09book}
M.~Tucsnak and G.~Weiss.
\newblock {\em Observation and Control for Operator Semigroups}.
\newblock Birkh\"auser Basel, 2009.

\bibitem{TucWei14}
M.~Tucsnak and G.~Weiss.
\newblock Well-posed systems---{T}he {LTI} case and beyond.
\newblock {\em Automatica}, 50(7):1757--1779, 2014.

\bibitem{Van98}
J.~Vancostenoble.
\newblock Weak asymptotic stability of second-order evolution equations by
  nonlinear and nonmonotone feedbacks.
\newblock {\em SIAM J. Math. Anal.}, 30(1):140--154, 1999.

\bibitem{VanFerCDC20}
N.~Vanspranghe, F.~Ferrante, and C.~Prieur.
\newblock Control of a wave equation with a dynamic boundary condition.
\newblock In {\em 59th IEEE Conference on Decision and Control (CDC)}, pages
  652--657, Jeju Island, Republic of Korea, 14--18 December 2020.

\bibitem{VilZwa05}
J.~A. Villegas, H.~Zwart, Y.~Le Gorrec, B.~Maschke, and A.~J. van~der Schaft.
\newblock Stability and stabilization of a class of boundary control systems.
\newblock In {\em Proceedings of the 44th IEEE Conference on Decision and
  Control}, pages 3850--3855, Dec 2005.

\bibitem{Villegas}
J.A. Villegas.
\newblock {\em A Port-{H}amiltonian Approach to Distributed Parameter Systems}.
\newblock PhD thesis, University of Twente, The Netherlands, May 2007.

\end{thebibliography}
\end{document}